\documentclass[review,3p,square,sort]{elsarticle}%review

\usepackage{amssymb}
\usepackage{amsmath}
\usepackage{amsthm}
\usepackage{stmaryrd}
\usepackage{bbm}
\usepackage{enumitem}

\usepackage{fancyhdr}
\usepackage{hyperref}
\hypersetup{%
  pdftitle={Doubly RBSDEs},%
  pdfsubject={DRBSDEs},%
  pdfauthor={ShengJun FAN},%
  pdfkeywords={BSDEs},%
  pdfstartview=FitH,%
  CJKbookmarks=true,%
  bookmarksnumbered=true,%
  bookmarksopen=true,%
  colorlinks=true, linkcolor=blue, urlcolor=blue, citecolor=blue, %
}

\usepackage{cleveref}
\crefname{thm}{Theorem}{Theorems}
\crefname{pro}{Proposition}{Propositions}
\crefname{lem}{Lemma}{Lemmas}
\crefname{rmk}{Remark}{Remarks}
\crefname{cor}{Corollary}{Corollaries}
\crefname{dfn}{Definition}{Definitions}
\crefname{ex}{Example}{Examples}
\crefname{section}{Section}{Sections}
\crefname{subsection}{Subsection}{Subsections}

\newcommand{\To}{\rightarrow}
\newcommand{\as}{{\rm d}\mathbb{P}\times{\rm d} t-a.e.}

\newcommand{\ps}{\mathbb{P}-a.s.}

\newcommand{\F}{\mathcal{F}}
\newcommand{\E}{\mathbb{E}}

\newcommand{\s}{\mathcal{S}}
\newcommand{\M}{{\rm M}}
\newcommand{\hcal}{\mathcal{H}}
\newcommand{\vcal}{\mathcal{V}}

\newcommand{\T}{[0,T]}
\newcommand{\LT}{{\mathbb L}^1(\F_T)}

\newcommand{\R}{{\mathbb R}}

\newcommand{\RE}{\forall}

\newcommand {\Lim}{\lim\limits_{n\rightarrow\infty}}
\newcommand {\Dis}{\displaystyle}
\newtheorem{thm}{Theorem}[section]
\newtheorem{lem}[thm]{Lemma}
\newtheorem{pro}[thm]{Proposition}
\newtheorem{rmk}[thm]{Remark}
\newtheorem{cor}[thm]{Corollary}
\newtheorem{dfn}[thm]{Definition}
\newtheorem{ex}[thm]{Example}

\journal{ArXiv}
\begin{document}
\begin{frontmatter}

\title{{\boldmath\bf
$L^1$ solutions of non-reflected BSDEs and reflected BSDEs with one and two continuous barriers under general assumptions}\tnoteref{found}}
\tnotetext[found]{Supported by the Fundamental
Research Funds for the Central Universities (No.\,2017XKQY98).\vspace{0.2cm}}

%\date{April 23, 2012}

\author{Shengjun FAN}%
\ead{f\_s\_j@126.com}

\address{School of Mathematics, China University of Mining and Technology, Xuzhou, Jiangsu, 221116, PR China\vspace{-0.8cm}}

\begin{abstract}
We establish several existence, uniqueness and comparison results for $L^1$ solutions of non-reflected BSDEs and reflected BSDEs with one and two continuous barriers under the assumption that the generator $g$ satisfies a one-sided Osgood condition together with a very general growth condition in $y$, a uniform continuity condition and/or a sub-linear growth condition in $z$, and a generalized Mokobodzki condition which relates the growth of $g$ and that of the barriers. This generalized Mokobodzki condition is proved to be necessary for existence of $L^1$ solutions of the reflected BSDEs. We also prove that the $L^1$ solutions of reflected BSDEs can be approximated by the penalization method and by some sequences of $L^1$ solutions of reflected BSDEs. These results strengthen some existing work on the $L^1$ solutions of non-reflected BSDEs and reflected BSDEs.\vspace{0.2cm}
\end{abstract}

\begin{keyword}
Reflected backward stochastic differential equation \sep Existence and uniqueness\sep
Comparison theorem\sep Stability theorem\sep
$L^1$ solution\vspace{0.2cm}

\MSC[2010] 60H10, 60H30
\end{keyword}

\end{frontmatter}

\section{Introduction}
\label{sec:1-Introduction}
\setcounter{equation}{0}

In 1990, \citet{PardouxPeng1990SCL} first introduced the notion of nonlinear backward stochastic differential equations (BSDEs for short) and established the well known existence and uniqueness result of an $L^2$ solution for a BSDE with square-integrability data under the assumption that the generator $g$ is uniformly Lipschitz continuous in $(y,z)$. Under the square-integrability assumption on data and the uniformly Lipschitz continuity assumption on generator, \citet{ElKarouiKapoudjianPardouxPengQuenez1997AoP} and \citet{CvitanicKaratzas1996AoP} respectively introduced the notion of nonlinear reflected BSDEs (RBSDEs) with one and two continuous barriers, established the existence and uniqueness of the $L^2$ solution, and explored that these equations have natural connections with the obstacle problem for PDEs, the optimal
stopping problem, mixed control problem and Dynkin games. Since then, the theory of BSDEs and reflected BSDEs
has rapidly grown and been applied in many areas such as mathematical finance, nonlinear expectation theory, stochastic control and game theory, optimality problems and others (see, e.g. \cite{BayraktarYao2015SPA},
\cite{ElKarouiPardouxQuenez1997NMIF},
\cite{ElKarouiPengQuenez1997MF},
\cite{HamadeneLepeltier2000SPA}, \cite{HamadeneLepeltierWu1999PMS},
\cite{HamadeneZhang2010SPA},
\cite{HuTang2010PTRF},
\cite{Jia2010SPA}, \cite{MaZhang2005SPA},
\cite{Pardoux1999NADEC}, \cite{Peng1997BSDEP}, \cite{Peng1999PTRF}, \cite{Peng2004LNM},
\cite{PengXu2005AIHPPS},
\cite{PengXu2010Bernoulli} and \cite{RosazzaEmanuela2006IME} etc.).

During more than two decades, for the theoretical interests of investigation and interesting applications a lot of works have been devoted to studying the existence and uniqueness of a solution for a non-reflected BSDE and a RBSDE by relaxing the square-integrability assumption on data and the uniformly Lipschitz continuity assumption on generator used in the pioneer papers \cite{PardouxPeng1990SCL}, \cite{ElKarouiKapoudjianPardouxPengQuenez1997AoP} and \cite{CvitanicKaratzas1996AoP}. For instance, the uniformly Lipschitz condition of $g$ in $(y,z)$ has been weakened to the  monotonicity and general growth condition in $y$ (see assumption \ref{A:(H1)} with $\rho(x)=k|x|$ for some constant $k\geq 0$ in \cref{sec:2-Assumptions2.3} of this paper) and the uniform continuity condition in $z$ (see assumption \ref{A:(H2)}(i) in \cref{sec:2-Assumptions2.3}), in the existence and uniqueness results for $L^2$ solutions or $L^p\ (p>1)$ solutions established respectively in, e.g., \citet{Pardoux1999NADEC}, \citet{BriandDelyonHu2003SPA}, \citet{BriandLepetierSanMrtin2007Bernoulli}, \citet{Jia2008CRA}, \citet{Jia2010SPA}, \citet{Chen2010SAA}, \citet{FanJiang2010CRA}, \citet{FanJiangDavison2010CRA} and
\citet{Fan2016SPA} for non-reflected BSDEs, and \citet{LepeltierMatoussiXu2005AdvinAP}, \citet{Klimsiak2012EJP}, \citet{RozkoszSlominski2012SPA}, \citet{Klimsiak2013BSM}, \citet{Fan2017AMS} and \citet{Fan2017arXivLpSolutionofDRBSDEs} for RBSDEs.

In the existence and uniqueness results for $L^2$ solutions or $L^p\ (p>1)$ solutions established respectively in \citet{FanJiang2013AMSE}, \citet{Fan2015JMAA},
\citet{Fan2017AMS}, \citet{Fan2016SPA} and \citet{Fan2017arXivLpSolutionofDRBSDEs}, the monotonicity condition of $g$ in $y$ was further weakened to the one-sided Osgood condition (see assumption \ref{A:(H1)}(i) in \cref{sec:2-Assumptions2.3}) and the weak monotonicity condition, which both unify the monotonicity condition, the Mao's non-Lipschitz condition (see \citet{Mao1995SPA}) and the usual Osgood condition (see \citet{FanJiangDavison2013FMC}).

In the case of concerning only the wellposedness or existence of the $L^2$ solution or $L^p\ (p>1)$ solution, the assumptions required by the generator $g$ have been further relaxed. For example, in \citet{BriandLepetierSanMrtin2007Bernoulli},
\citet{Xu2008SPA}, \citet{Fan2017AMS} and \citet{Fan2017arXivLpSolutionofDRBSDEs}, besides the (weak) monotonic condition in $y$ and the continuity condition in $(y,z)$, a general growth in $y$ and a linear growth in $z$ (see assumption \ref{A:(HH)} with $\alpha=1$ in \cref{sec:2-Assumptions2.3}) is the only requirement for the generator $g$, and in \citet{LepeltierSanMartin1997SPL}, \citet{Matoussi1997SPL} and \citet{HamadeneLepeltierMatoussi1997BSDEs} and \citet{JiaXu2014arXiv}, the generator $g$ needs only to be continuous and have a linear growth in $(y,z)$ (see assumption \ref{A:(AA)} with $\tilde\alpha=1$ in \cref{sec:2-Assumptions2.3}).

During the evolution of BSDE theory, many papers have also been interested in the existence and uniqueness of the $L^1$ solutions for non-reflected BSDEs and RBSDEs with only integrability data, see, for example,
\citet{ElKarouiPengQuenez1997MF}, \citet{BriandDelyonHu2003SPA}, \citet{BriandHu2006PTRF}, \citet{FanLiu2010SPL}, \citet{Fan2017arXivL1SolutionofBSDEsBulletin}, \citet{Klimsiak2012EJP}, \citet{RozkoszSlominski2012SPA}, \citet{Fan2017arXivL1SolutionofMultiDimensionalBSDEs}, \citet{BayraktarYao2015SPA}, and \citet{HuTang2017ArXiv} for details, where generally speaking (except \cite{HuTang2017ArXiv}), an additional sub-linear growth condition in $z$ (see assumption \ref{A:(H2)}(ii) or \ref{A:(H2')}(ii) in \cref{sec:2-Assumptions2.3}) needs to be satisfied by the generator $g$.

In order to ensure existence of a solution for RBSDEs with two barriers, a Mokobodzki condition (i.e., there exists a quasi-martingale between two barriers) or a certain regularity condition on one of the barriers usually needs to be satisfied as in \citet{CvitanicKaratzas1996AoP}, \citet{BahlaliHamadeneMezerdi2005SPA} and \citet{PengXu2005AIHPPS}. By virtue of the notion of local solutions, these two conditions were replaced with the completely separated condition of the two barriers, which can be more easily verified or checked, in \citet{HamadeneHassani2005PTRF}, \citet{HamadeneHassaniOuknine2010BSM}, \citet{ElAsriHamadeneWang2011SAA}, \citet{BayraktarYao2015SPA} and so on.

Recently, several generalized Mokobodzki conditions, see (ii) of assumptions \ref{A:(H3)}, \ref{A:(H3L)} and \ref{A:(H3U)} in \cref{sec:2-Assumptions2.3} for the case of $L^1$ solution, were put forward and proved to be sufficient and necessary to ensure existence of an $L^p\ (p>1)$ or $L^1$ solution for a RBSDE with one or two barriers when the generator $g$ has a general growth in $y$, see \citet{Klimsiak2012EJP}, \citet{Klimsiak2013BSM}, \citet{Fan2017AMS} and \citet{Fan2017arXivLpSolutionofDRBSDEs} for more details. Many efforts in this direct can also be found in \citet{LepeltierMatoussiXu2005AdvinAP}, \citet{Xu2007JoTP}, \citet{Xu2008SPA}, \citet{RozkoszSlominski2012SPA}, \citet{LiShi2016SPL} and references therein.

Enlightened by these works aforementioned, especially by \citet{PengXu2005AIHPPS}, \citet{Klimsiak2012EJP}, \citet{BayraktarYao2015SPA} and \citet{Fan2017AMS}, we dedicate this paper to the $L^1$ solution of non-reflected BSDEs and RBSDEs with one and two continuous barriers under general assumptions on the generator and the data, i.e., \ref{A:(H1)}, \ref{A:(H2)}, \ref{A:(H2')}, \ref{A:(H3)}, \ref{A:(H3L)}, \ref{A:(H3U)}, \ref{A:(HH)} and \ref{A:(AA)} mentioned above, see \cref{sec:2-Assumptions2.3} again. Our results strengthens some corresponding known works on the $L^1$ solutions of on-reflected BSDEs and RBSDEs (see \cref{rmk:7-7.1} in \cref{sec:7-ExamplesandRemarks} for more details). Our approach is based on a combination between existing methods, their refinement and perfection, but also on some novel ideas and techniques.

The rest of this paper is organized as follows. \cref{sec:2-NotationsAssumptions} contains some notations, definitions, assumptions and lemmas which will be used later. \cref{sec:3-PenalizationApproximationComparisonTheorem}
consists of four subsections, which establish three convergence results respectively on the penalization scheme and the approximation scheme for the $L^1$ solutions of RBSDEs with one and two barriers under general assumptions, and a general comparison theorem for the $L^1$ solutions of RBSDEs under assumptions \ref{A:(H1)}(i) and \ref{A:(H2)}. These elementary results will play important roles in the proof of our main results in the subsequent sections.

\cref{sec:4-ExistenceBSDEs} is devoted to the $L^1$ solution of non-reflected BSDEs. In this section, we prove an existence and uniqueness result for an $L^1$ solution of a BSDE under assumptions \ref{A:(H1)} and \ref{A:(H2)} (see \cref{thm:4-ExistenceanduniquenssofBSDEunderH2}), and an existence result for a minimal and a maximal $L^1$ solution of a BSDE with generator $g:=g^1+g^2$, where the generator $g^1$ satisfies assumptions \ref{A:(H1)}(i) and \ref{A:(HH)} (resp. \ref{A:(H1)} and \ref{A:(H2')}), and $g^2$ satisfies assumption \ref{A:(AA)} (see \cref{thm:4-ExistenceofBSDEunderHH}).

\cref{sec:5-ExistenceRBSDEs} deals with the $L^1$ solution of RBSDEs with one continuous barrier. By \cref{thm:5-ExistenceanduniquenssofRBSDEunderH2} we prove the existence and uniqueness of an $L^1$ solution for a RBSDE with one lower (resp. upper) barrier under assumptions \ref{A:(H1)}, \ref{A:(H2)} and \ref{A:(H3L)} (resp. \ref{A:(H3U)}) by the penalization method, and show the sufficient and necessary property of  \ref{A:(H3L)}(ii) (resp. \ref{A:(H3U)}(ii)). And, in \cref{thm:5-ExistenceofRBSDEunderH2'Penalization} we study the same problem, but on existence of a minimal (resp. maximal) $L^1$ solution for a RBSDE with one lower (resp. upper) barrier and  a generator $g:=g^1+g^2$, where the generator $g^1$ satisfies assumptions \ref{A:(H1)} and \ref{A:(H2')} and the generator $g^2$ satisfies assumption \ref{A:(AA)}. Furthermore, by \cref{thm:5-ExistenceofRBSDEunderH2'Approximation} we show that under the assumptions of \cref{thm:5-ExistenceofRBSDEunderH2'Penalization},  the minimal and maximal $L^1$ solution for the RBSDE with one lower or upper barrier can be both approximated by a sequence of $L^1$ solutions for RBSDEs with generators satisfying  \ref{A:(H1)} and \ref{A:(H2)}.

\cref{sec:6-ExistenceOfDRBSDEs} investigates the $L^1$ solution of RBSDEs with two continuous barriers. By \cref{thm:6-ExistenceandUniquenessUnder(H2)} we prove the existence and uniqueness of an $L^1$ solution for a doubly RBSDE under assumptions \ref{A:(H1)}, \ref{A:(H2)} and \ref{A:(H3)} by the penalization method, and show the sufficient and necessary property of \ref{A:(H3)}(ii). And, in \cref{thm:6-ExistenceofDRBSDEunderH2'Penalization} we study the same problem, but on existence of a minimal and a maximal $L^1$ solution for a doubly RBSDE with a generator as in \cref{thm:5-ExistenceofRBSDEunderH2'Penalization}. Furthermore, by \cref{thm:6-ExistenceofDRBSDEunderH2'Approximation} we prove that under the assumptions of \cref{thm:6-ExistenceofDRBSDEunderH2'Penalization},  the minimal and maximal $L^1$ solution for the doubly RBSDE can be both approximated by a sequence of $L^1$ solutions for doubly RBSDEs with generators satisfying  \ref{A:(H1)} and \ref{A:(H2)}. Finally, in \cref{sec:7-ExamplesandRemarks} we introduce several examples and remarks to illustrate further the theoretical results obtained in this paper.

\section{Notations, definitions, assumptions and lemmas}
\label{sec:2-NotationsAssumptions}
\setcounter{equation}{0}

\subsection{Notations}
\label{sec:2-Notations2.1}

Let $T>0$ be a fixed real number and $(\Omega,\F_T,\mathbb{P};(\F_t)_{t\in\T})$ be a complete filtered probability space carrying a standard $d$-dimensional Brownian motion $(B_t)_{t\in\T}$ together with the completed $\sigma$-algebra filtration $(\F_t)_{t\in\T}$ generated by $B_\cdot$. Denote by $\mathbbm{1}_{A}$ the indicator function of a set $A$ and by $A^c$ the complement of $A$. Let $\R_+:=[0,+\infty)$, $a^+:=\max\{a,0\}$ and $a^-:=(-a)^+$ for any real number $a$, and let ${\rm sgn}(x)$ represent the sign of a real number $x$ and $|y|$ the Euclidean norm of $y\in \R^{n}$ with $n\geq 1$. Furthermore, denote by $\s$ the set of all $(\F_t)$-progressively measurable and continuous processes $(Y_t)_{t\in\T}$, and for $p>0$ we denote by $\s^p$ the set of all processes $Y_\cdot\in \s$ satisfying
$$\|Y\|_{{\s}^p}:=\left(\E[\sup_{t\in\T} |Y_t|^p]\right)^{1\wedge 1/p}<+\infty.$$
$\M$ is the set of all $(\F_t)$-progressively measurable $\R^d$-valued processes $(Z_t)_{t\in\T}$ satisfying
$$\mathbb{P}\left(\int_0^T|Z_t|^2{\rm d}t<+\infty\right)=1,$$
and for $p>0$, $\M^p$ is the set of all processes $Z_\cdot\in \M$ satisfying
$$
\|Z\|_{\M^p}:=\left\{ \E\left[\left(\int_0^T |Z_t|^2{\rm d}t\right)^{p/2}\right] \right\}^{1\wedge 1/p}<+\infty.
$$
We also use the following spaces with respect to variables and processes defined on $\Omega\times\T$:
\begin{itemize}
\item $\LT$ the set of all $\F_T$-measurable random variables $\xi$ satisfying $\E[|\xi|]<+\infty;$
\item $\hcal$  the set of all $(\F_t)$-progressively measurable processes $X_\cdot$ satisfying
    $\mathbb{P}\left(\int_0^T|X_t|{\rm d}t<+\infty\right)=1;$
\item $\hcal^1$ the set of all processes $X_\cdot\in \hcal$ satisfying
    $\|X\|_{\hcal^1}:=\E\left[\int_0^T |X_t|{\rm d}t\right]<+\infty;$
\item $\vcal$ the set of all $(\F_t)$-progressively measurable and continuous processes of finite variation;
\item $\vcal^+$ the set of all $(\F_t)$-progressively measurable, continuous and increasing processes valued $0$ at $0$;
\item $\vcal^1$ (resp. $\vcal^{+,1}$) the set of all processes $V_\cdot\in \vcal$ (resp. $\vcal^+$) satisfying $\E\left[|V|_T\right]<+\infty$.

\end{itemize}
Here and hereafter, for each $(\F_t)$-stopping time $\tau$ valued in $\T$, $|V|_\tau$ represents the random finite variation of $V_\cdot\in\vcal$ on the stochastic interval $[0,\tau]$. It is clear that $|V|_\tau=V_\tau$ when $V_\cdot\in\vcal^{+}$.

For any two processes $K^1_\cdot$ and $K^2_\cdot$ in the space $\vcal^1$, we say ${\rm d}K^1\bot {\rm d}K^2$ means that there exists an $(\F_t)$-progressively measurable set $D\subset \Omega\times\T$ such that
$$\E\left[\int_0^T \mathbbm{1}_{D}(t,\omega)\ {\rm d}K^1_t(\omega)\right]
=\E\left[\int_0^T \mathbbm{1}_{D^c}(t,\omega)\ {\rm d}K^2_t(\omega)\right]=0.$$
And, we say ${\rm d}K^1\leq {\rm d}K^2$ means that for each $(\F_t)$-progressively measurable set $D\subset \Omega\times\T$,
$$\E\left[\int_0^T \mathbbm{1}_{D}(t,\omega)\ {\rm d}K^1_t(\omega)\right]\leq \E\left[\int_0^T \mathbbm{1}_{D}(t,\omega)\ {\rm d}K^2_t(\omega)\right],$$
$i.e.,$
$
K^1_t-K^1_s\leq K^2_t-K^2_s,\ 0\leq s\leq t\leq T.
$

Finally, we recall that a process $(Y_t)_{t\in\T}$ belongs to the class (D) if the family of variables $\{|Y_\tau|:\tau$ is an $(\F_t)$-stopping time bounded by $T\}$ is uniformly integrable.

In the rest of this paper, the variable $\omega$ in random elements is often omitted and all equalities and inequalities between random variables are understood to hold $\ps$ without a special illustration.

\subsection{Definitions\vspace{0.1cm}}
\label{sec:2-Definitions2.2}

In this paper, we always assume that $\xi\in\LT$, $V_\cdot\in\vcal$, $L_\cdot\in \s$ (or $L_\cdot=-\infty$), $U_\cdot\in \s$ (or $U_\cdot=+\infty$), $L_\cdot\leq U_\cdot$, and that a random function, which is usually called a generator,
$$
g(\omega,t,y,z):\Omega\times\T\times\R\times\R^d
\longmapsto \R\vspace{-0.1cm}
$$
is $(\F_t)$-progressively measurable for each $(y,z)$, and continuous in $(y,z)$ for almost each $(\omega,t)$. \vspace{0.1cm}

We use the following definition for the $L^1$ solution of non-reflected BSDEs and reflected BSDEs with one and two continuous barriers.

\begin{dfn}
\label{dfn:2-DefinitionOfBSDERBSDEDRBSDE}
By an $L^1$ solution to BSDE $(\xi,g+{\rm d}V)$ we understand a pair $(Y_t,Z_t)_{t\in\T}\in \s^\beta\times \M^\beta$ for each $\beta\in (0,1)$ such that $(Y_t)_{t\in\T}$ belongs to the class (D) and
$$
Y_t=\xi+\int_t^Tg(s,Y_s,Z_s){\rm d}s+\int_t^T{\rm d}V_s-\int_t^TZ_s \cdot {\rm d}B_s,\ \ t\in\T.
$$
By an $L^1$ solution to $\underline{R}$BSDE $(\xi,g+{\rm d}V,L)$ we understand a triple $(Y_t,Z_t,K_t)_{t\in\T}\in \s^\beta\times \M^\beta\times\vcal^{+,1}$ for each $\beta\in (0,1)$  such that $(Y_t)_{t\in\T}$ belongs to the class (D) and
$$
\left\{
\begin{array}{l}
\Dis Y_t=\xi+\int_t^Tg(s,Y_s,Z_s){\rm d}s+\int_t^T{\rm d}V_s+\int_t^T{\rm d}K_s-\int_t^TZ_s \cdot {\rm d}B_s,\ \   t\in\T,\vspace{0.1cm}\\
\Dis L_t\leq Y_t,\ \ t\in\T\ \ {\rm and} \ \int_0^T (Y_t-L_t){\rm d}K_t=0.
\end{array}
\right.
$$
By an $L^1$ solution to $\bar{R}$BSDE $(\xi,g+{\rm d}V,U)$ we understand a triple $(Y_t,Z_t,A_t)_{t\in\T}\in \s^\beta\times \M^\beta\times\vcal^{+,1}$ for each $\beta\in (0,1)$ such that $(Y_t)_{t\in\T}$ belongs to the class (D) and
$$
\left\{
\begin{array}{l}
\Dis Y_t=\xi+\int_t^Tg(s,Y_s,Z_s){\rm d}s+\int_t^T{\rm d}V_s-\int_t^T{\rm d}A_s-\int_t^TZ_s \cdot {\rm d}B_s,\ \   t\in\T,\vspace{0.1cm}\\
\Dis Y_t\leq U_t,\ \ t\in\T\ \ {\rm and} \ \int_0^T (U_t-Y_t){\rm d}A_t=0.
\end{array}
\right.
$$
By an $L^1$ solution to DRBSDE $(\xi,g+{\rm d}V,L,U)$ we understand a quadruple $(Y_t,Z_t,K_t,A_t)_{t\in\T}\in \s^\beta\times \M^\beta\times\vcal^{+,1}\times\vcal^{+,1}$ for each $\beta\in (0,1)$ such that both $Y_\cdot$ belongs to the class (D), and
$$
\left\{
\begin{array}{l}
\Dis Y_t=\xi+\int_t^Tg(s,Y_s,Z_s){\rm d}s+\int_t^T{\rm d}V_s+\int_t^T{\rm d}K_s-\int_t^T{\rm d}A_s-\int_t^TZ_s \cdot {\rm d}B_s,\ \   t\in\T,\vspace{0.1cm}\\
\Dis L_t\leq Y_t\leq U_t,\ \ t\in\T,\ \ \int_0^T (Y_t-L_t){\rm d}K_t=\int_0^T (U_t-Y_t){\rm d}A_t=0\ \ {\rm and} \ \ {\rm d}K\bot{\rm d}A.
\end{array}
\right.
$$
Furthermore, an $L^1$ solution $(Y_t,Z_t)_{t\in\T}$ of BSDE $(\xi,g+{\rm d}V)$ is called the minimal (resp. maximal) $L^1$ solution if for any $L^1$ solution $(Y'_t,Z'_t)_{t\in\T}$ of BSDE $(\xi,g+{\rm d}V)$, we have
$$
Y_t\leq Y'_t,\ \ t\in\T\ \ \ ({\rm resp.}\ \  Y_t\geq Y'_t,\ \ t\in\T).
$$
Similarly, we can define the minimal (resp. maximal) $L^1$ solution for \underline{R}BSDE $(\xi,g+{\rm d}V,L)$, $\bar{R}$BSDE $(\xi,g+{\rm d}V,U)$ and DRBSDE $(\xi,g+{\rm d}V,L,U)$.
\end{dfn}

\subsection{Assumptions\vspace{0.1cm}}
\label{sec:2-Assumptions2.3}

In this paper, we will use the following assumptions with respect to the generator, the terminal condition and the barriers.
\begin{enumerate}

\renewcommand{\theenumi}{(H\arabic{enumi})}
\renewcommand{\labelenumi}{\theenumi}

\item \label{A:(H1)}
\begin{itemize}
\item [(i)] $g$ satisfies the one-sided Osgood condition in $y$, i.e., there exists a nondecreasing and concave function $\rho(\cdot):\R_+\mapsto \R_+$ with $\rho(0)=0$, $\rho(u)>0$ for $u>0$ and $\int_{0^+} {{\rm d}u\over \rho(u)}=+\infty$ such that $\as$, $\RE\ y_1,y_2\in \R,z\in\R^{d}$,
$$
(g(\omega,t,y_1,z)-g(\omega,t,y_2,z)){\rm sgn}(y_1-y_2)\leq \rho(|y_1-y_2|).\vspace{-0.2cm}
$$
\item [(ii)] $g(\cdot,0,0)\in\hcal^1$;
\item [(iii)] $g$ has a general growth in $y$, i.e, $\as$, $\RE r>0$, $$\psi_\cdot(r):=\sup\limits_{|y|\leq r}|g(\cdot,y,0)-g(\cdot,0,0)|\ {\rm belongs\ to\ the\ space}\ \hcal.$$
\end{itemize}

\item \label{A:(H2)}
\begin{itemize}
\item [(i)] $g$ is uniformly continuous in $z$, i.e., there exists a nondecreasing and continuous function $\phi(\cdot):\R_+\mapsto \R_+$ with $\phi(0)=0$ such that $\as$, $\RE\ y\in\R, z_1,z_2\in\R^{d}$,
$$
|g(\omega,t,y,z_1)-g(\omega,t,y,z_2)|\leq \phi(|z_1-z_2|);
$$
\item [(ii)] $g$ has a stronger sub-linear growth in $z$, i.e., there exist two constants $\gamma\geq 0$ and $\alpha\in (0,1)$ together with a nonnegative process $f_\cdot\in \hcal^1$ such that $\as$, $\RE\ y\in\R, z\in\R^{d}$,
$$
|g(\omega,t,y,z)-g(\omega,t,y,0)|\leq  \gamma(f_t(\omega)+|y|+|z|)^\alpha.
$$
\end{itemize}

\renewcommand{\theenumi}{(H2')}
\renewcommand{\labelenumi}{\theenumi}

\item \label{A:(H2')}
\begin{itemize}
\item [(i)] $g$ is stronger continuous in $(y,z)$, i.e., $\as$, $\RE\ y\in \R,\ g(\omega,t,y,\cdot)$ is continuous, and $g(\omega,t,\cdot,z)$ is continuous uniformly with respect to $z$;
\item [(ii)] $g$ has a sub-linear growth in $z$, i.e., there exist three constants $\mu,\lambda\geq 0$ and $\alpha\in (0,1)$ together with a nonnegative process $f_\cdot\in \hcal^1$ such that $\as$, $\RE\ y\in\R, z\in\R^{d}$,
$$
|g(\omega,t,y,z)-g(\omega,t,y,0)|\leq  f_t(\omega)+\mu|y|+\lambda|z|^\alpha.
$$
\end{itemize}

\renewcommand{\theenumi}{(H\arabic{enumi})}
\renewcommand{\labelenumi}{\theenumi}
\setcounter{enumi}{2}

\item \label{A:(H3)}
\begin{itemize}
\item [(i)]$L_\cdot\in \s$ (or $L_\cdot=-\infty$), $U_\cdot\in \s$ (or $U_\cdot=+\infty$), $L_\cdot\leq U_\cdot$, $\xi\in \LT$ and $L_T\leq \xi\leq U_T$;

\item [(ii)] There exists two processes $(C_\cdot,H_\cdot)\in \vcal^1\times \M^\beta$ for each $\beta\in (0,1)$ such that
    $$X_t:=X_0+\int_0^t {\rm d}C_s+\int_0^t H_s\cdot {\rm d}B_s,\ \ \ t\in [0,T]$$ belongs to the class (D), $g(\cdot,X_\cdot,0)\in \hcal^1$ and $L_t\leq X_t\leq U_t$ for each $t\in \T$.
\end{itemize}

\renewcommand{\theenumi}{(H3L)}
\renewcommand{\labelenumi}{\theenumi}

\item \label{A:(H3L)}
\begin{itemize}
\item [(i)]$L_\cdot\in \s$ (or $L_\cdot=-\infty$), $\xi\in \LT$ and $L_T\leq \xi$;

\item [(ii)] There exists two processes $(C_\cdot,H_\cdot)\in \vcal^1\times \M^\beta$ for each $\beta\in (0,1)$ such that
    $$X_t:=X_0+\int_0^t {\rm d}C_s+\int_0^t H_s\cdot {\rm d}B_s,\ \ \ t\in [0,T]$$ belongs to the class (D), $g(\cdot,X_\cdot,0)\in \hcal^1$ and $L_t\leq X_t$ for each $t\in \T$.
\end{itemize}

\renewcommand{\theenumi}{(H3U)}
\renewcommand{\labelenumi}{\theenumi}

\item \label{A:(H3U)}
\begin{itemize}
\item [(i)]$U_\cdot\in \s$ (or $U_\cdot=+\infty$), $\xi\in \LT$ and $\xi\leq U_T$;

\item [(ii)] There exists two processes $(C_\cdot,H_\cdot)\in \vcal^1\times \M^\beta$ for each $\beta\in (0,1)$ such that
    $$X_t:=X_0+\int_0^t {\rm d}C_s+\int_0^t H_s\cdot {\rm d}B_s,\ \ \ t\in [0,T]$$ belongs to the class (D), $g(\cdot,X_\cdot,0)\in \hcal^1$ and $ X_t\leq U_t$ for each $t\in \T$.
\end{itemize}

\renewcommand{\theenumi}{(HH)}
\renewcommand{\labelenumi}{\theenumi}

\item \label{A:(HH)}
\begin{itemize}
\item [(i)] $g$ is stronger continuous in $(y,z)$, i.e., $\as$, $\RE\ y\in \R,\ g(\omega,t,y,\cdot)$ is continuous, and $g(\omega,t,\cdot,z)$ is continuous uniformly with respect to $z$;
\item [(ii)] $g$ has a general growth in $y$ and a sub-linear growth in $z$, i.e., there exist two constants $\lambda\geq 0$ and $\alpha\in (0,1)$, a nonnegative process $f_\cdot\in\hcal^1$ and a nonnegative function $\varphi_\cdot(r)\in {\bf S}$ such that
$$
\as,\ \RE\ y\in\R\ {\rm and}\ z\in\R^{d},\ \ |g(\omega,t,y,z)|\leq f_t(\omega)+ \varphi_t(\omega,|y|)+\lambda |z|^\alpha,
$$
here and hereafter, ${\bf S}$ denotes the set of nonnegative functions $\varphi_t(\omega,r):\Omega\times \T\times \R_+\mapsto \R_+$
satisfying the following two conditions:
\begin{itemize}
\item $\as$, the function $r\mapsto \varphi_t(\omega,r)$ is increasing and $\varphi_t(\omega,0)=0$;
\item for each $r\geq 0$, $\varphi_\cdot(\cdot,r)\in \hcal$.
\end{itemize}
\end{itemize}

\renewcommand{\theenumi}{(AA)}
\renewcommand{\labelenumi}{\theenumi}

\item \label{A:(AA)}
 $g$ has a linear growth in $y$ and a sub-linear growth in $z$, i.e., there exist three constants $\tilde\mu,\tilde\lambda\geq 0$ and $\tilde\alpha\in (0,1)$ together with a nonnegative process $\tilde f_\cdot\in \hcal^1$ such that $\as$, $\RE\ y\in\R, z\in\R^{d}$,
$$
|g(\omega,t,y,z)|\leq  \tilde f_t(\omega)+\tilde\mu|y|+\tilde\lambda
|z|^{\tilde\alpha}.
$$
\end{enumerate}

\begin{rmk}
\label{rmk:2-LinearGrowthOfRhoandPhi}
Without loss of generality, we will always assume that the functions $\rho(\cdot)$ and $\phi(\cdot)$ defined respectively in \ref{A:(H1)} and \ref{A:(H2)} are of linear growth, i.e., there exists a constant $A>0$ such that
$$\RE\ x\in \R_+,\ \ \rho(x)\leq A(x+1)\ \ {\rm and}\ \ \phi(x)\leq A(x+1).$$
It is clear that \ref{A:(H2)}(ii) can imply \ref{A:(H2')}(ii), \ref{A:(H1)}(ii)(iii) and \ref{A:(H2)}(ii) (or \ref{A:(H2')}(ii)) can imply \ref{A:(HH)}(ii), and \ref{A:(HH)}(ii) can imply \ref{A:(H1)}(ii)(iii). In addition, \ref{A:(H3)}, \ref{A:(H3L)} and \ref{A:(H3U)} are the so-called generalized Mokobodzki conditions, which relate the growth of $g$ and that of the barriers.
\end{rmk}

\subsection{Lemmas\vspace{0.1cm}}
\label{sec:2-Estimates2.4}

In this subsection, let us introduce several lemmas, which will play an important role later. Firstly, the following a priori estimate comes from Lemma 3.1 in \citet{Fan2017AMS}.

\begin{lem}
\label{lem:2-Lemma1}
Let the triple $(\bar Y_\cdot,\bar Z_\cdot,\bar V_\cdot)\in \s\times\M\times\vcal$ satisfy the following equation:
\begin{equation}
\label{eq:2-BarY=BarV}
\bar Y_t=\bar Y_T+\int_t^T {\rm d}\bar V_s-\int_t^T \bar Z_s\cdot {\rm d}B_s,\ \ t\in \T.
\end{equation}
We have
\begin{itemize}
\item [(i)] For each $p>0$, there exists a constant $C_1>0$ depending only on $p$ such that for each $t\in\T$ and each $(\F_t)$-stopping time $\tau$ valued in $\T$,
$$
\Dis\E\left[\left.\left(\int_{t\wedge\tau}
^{\tau}|\bar Z_s|^2{\rm d}s\right)^{p\over 2}\right|\F_t\right]
\leq \Dis C_1\E\left[\left.\sup\limits_{s\in [t,T]}|\bar Y_{s\wedge\tau}|^p+\sup\limits_{s\in [t,T]}\left[\left(\int_{s\wedge\tau}^{\tau} \bar Y_r{\rm d}\bar V_r\right)^+\right]^{p\over 2}\right|\F_t\right];
$$

\item [(ii)] If $\bar Y_\cdot\in \s^p$ for some $p>1$, then there exists a constant $C_2>0$ depending only on $p$ such that for each $t\in\T$ and each $(\F_t)$-stopping time $\tau$ valued in $\T$,
$$
\begin{array}{ll}
&\Dis\E\left[\left.\sup\limits_{s\in [t,T]}|\bar Y_{s\wedge\tau}|^p+
\int_{t\wedge\tau}^{\tau} |\bar Y_s|^{p-2}\mathbbm{1}_{\{|\bar Y_s|\neq 0\}}|\bar Z_s|^2{\rm d}s
\right|\F_t\right]\vspace{0.1cm}\\
\leq &\Dis C_2\E\left[\left.|\bar Y_\tau|^p+\sup\limits_{s\in [t,T]}\left(\int_{s\wedge\tau}^{\tau} |\bar Y_r|^{p-1}{\rm sgn}(\bar Y_r){\rm d}\bar V_r\right)^+\right|\F_t\right].
\end{array}\vspace{0.1cm}
$$
\end{itemize}
\end{lem}

Secondly, the following observation will be used several times later.\vspace{-0.1cm}

\begin{lem}\label{lem:2-SublinearOfg}
Let the generator $g$ satisfy \ref{A:(H1)}(i) and \ref{A:(H2')}(ii) (resp. \ref{A:(H2)}(ii)), and $(\underline X_\cdot, Y_\cdot, \bar X_\cdot,Z_\cdot)\in \s\times\s\times\s\times\M$ satisfy $\underline X_\cdot\leq Y_\cdot\leq \bar X_\cdot$. Then, $\as$,
\begin{equation}
\label{eq:2-SubLinearGrowthofg}
|g(\cdot,Y_\cdot,Z_\cdot)|\leq |g(\cdot,\underline X_\cdot,0)|+|g(\cdot,\bar X_\cdot,0)|+(\mu+A)(|\underline X_\cdot|+|\bar X_\cdot|)+ f_\cdot+A+\lambda|Z_\cdot|^\alpha
\end{equation}
$$
({\rm resp.}\ \ \ |g(\cdot,Y_\cdot,Z_\cdot)|\leq |g(\cdot,\underline X_\cdot,0)|+|g(\cdot,\bar X_\cdot,0)|+(\gamma+A)(|\underline X_\cdot|+|\bar X_\cdot|)+ \gamma (1+f_\cdot)+A+\gamma |Z_\cdot|^\alpha\ ).
$$
\end{lem}

\begin{proof}
We only prove the case of \ref{A:(H2')}. Another case is similar. Indeed, by \ref{A:(H1)}(i) and \ref{A:(H2')}(ii) together with $\underline X_\cdot\leq Y_\cdot\leq \bar X_\cdot$ and \cref{rmk:2-LinearGrowthOfRhoandPhi} we know that $\as$,
$$
\begin{array}{lll}
\Dis g(\cdot,Y_\cdot,Z_\cdot) &\leq &\Dis g(\cdot,Y_\cdot,Z_\cdot)-g(\cdot,\underline X_\cdot,Z_\cdot)+|g(\cdot,\underline X_\cdot,Z_\cdot)-g(\cdot,\underline X_\cdot,0)|+|g(\cdot,\underline X_\cdot,0)|\\
&\leq &\Dis \rho(|Y_\cdot-\underline X_\cdot|)+f_\cdot+\mu |\underline X_\cdot|+\lambda|Z_\cdot|^\alpha+|g(\cdot,\underline X_\cdot,0)|\\
&\leq &\Dis A(|\bar X_\cdot-\underline X_\cdot|)+A+f_\cdot+\mu |\underline X_\cdot|+\lambda|Z_\cdot|^\alpha+|g(\cdot,\underline X_\cdot,0)|.
\end{array}
$$
and
$$
\begin{array}{lll}
\Dis -g(\cdot,Y_\cdot,Z_\cdot) &\leq &\Dis g(\cdot,\bar X_\cdot,Z_\cdot)
-g(\cdot,Y_\cdot,Z_\cdot)+|g(\cdot,\bar X_\cdot,Z_\cdot)-g(\cdot,\bar X_\cdot,0)|+|g(\cdot,\bar X_\cdot,0)|\\
&\leq &\Dis \rho(|\bar X_\cdot -Y_\cdot|)+f_\cdot+\mu |\bar X_\cdot|+\lambda|Z_\cdot|^\alpha+|g(\cdot,\bar X_\cdot,0)|\\
&\leq &\Dis A(|\bar X_\cdot-\underline X_\cdot|)+A+f_\cdot+\mu |\bar X_\cdot|+\lambda|Z_\cdot|^\alpha+|g(\cdot,\bar X_\cdot,0)|.
\end{array}
$$
Then, the desired conclusion \eqref{eq:2-SubLinearGrowthofg} follows immediately.
\end{proof}

Thirdly, the following lemma has a close connection with the generalized Mokobodzki condition, which will be shown in subsequent sections.\vspace{-0.1cm}

\begin{lem}\label{lem:2-gBelongstoH1}
Assume that $\xi\in \LT$, $\bar V\in\vcal^1$, $g$ is a generator and $(Y_\cdot,Z_\cdot)$ is an $L^1$ solution of BSDE $(\xi,g+{\rm d}\bar V)$. If the generator $g$ satisfies \ref{A:(H1)}(i)(ii) and \ref{A:(H2')}(ii) (resp. \ref{A:(H2)}(ii)), then
\begin{equation}\label{eq:2-gBelongstoH1}
g(\cdot,Y_\cdot,Z_\cdot)\in \hcal^1\ \ {\rm and}\ \  g(\cdot,Y_\cdot,0)\in \hcal^1.
\end{equation}
\end{lem}

\begin{proof}
In view of \cref{rmk:2-LinearGrowthOfRhoandPhi} we only need to prove the case of \ref{A:(H2')}. Indeed, for each positive integer $k\geq 1$, define the following $(\F_t)$-stopping time:
$$
\tau_k:=\inf\{t\in\T:\ \ \int_0^t |Z_s|^2{\rm d}s\geq k\}\wedge T.
$$
Note that $\tau_k\To T$ as $k\To +\infty$ due to the fact that $Z_\cdot\in \M$. By It\^{o}-Tanaka's formula we deduce that
$$
-\int_0^{\tau_k}{\rm sgn}(Y_s)g(s,Y_s,Z_s){\rm d}s\leq |Y_{\tau_k}|-|Y_0|+\int_0^{\tau_k} {\rm sgn}(Y_s){\rm d}\bar V_s-\int_0^{\tau_k} {\rm sgn}(Y_s)Z_s\cdot {\rm d}B_s.
$$
Then,
$$
\begin{array}{ll}
&\Dis \int_0^{\tau_k}[\rho(|Y_s|)-{\rm sgn}(Y_s)(g(s,Y_s,Z_s)-g(s,0,Z_s))]{\rm d}s\vspace{0.1cm}\\
\leq &\Dis |Y_{\tau_k}|+|\bar V|_{\tau_k}+
\int_0^{\tau_k} (\rho(|Y_s|)+|g(s,0,Z_s)|){\rm d}s
-\int_0^{\tau_k} {\rm sgn}(Y_s)Z_s\cdot {\rm d}B_s.
\end{array}
$$
By taking mathematical expectation and letting $k\To\infty$ in the previous inequality, in view of \ref{A:(H1)}(i) , Levi's lemma and the fact that $Y_\cdot$ belongs to the class (D), we can obtain
\begin{equation}
\label{eq:2-star}
\E\left[\int_0^T|\rho(|Y_s|)-{\rm sgn}(Y_s)(g(s,Y_s,Z_s)-g(s,0,Z_s))|{\rm d}s\right]
\leq \E\left[|\xi|+|\bar V|_T+
\int_0^T (\rho(|Y_s|)+|g(s,0,Z_s)|){\rm d}s\right].\vspace{-0.2cm}
\end{equation}
Furthermore, noticing that
$$
\begin{array}{ll}
&\Dis\E\left[\int_0^T|g(s,Y_s,Z_s)-g(s,0,Z_s)|{\rm d}s\right]=\E\left[\int_0^T|{\rm sgn}(Y_s)(g(s,Y_s,Z_s)-g(s,0,Z_s))|{\rm d}s\right]\vspace{0.1cm}\\
\leq &\Dis
\E\left[\int_0^T\left[|{\rm sgn}(Y_s)(g(s,Y_s,Z_s)-g(s,0,Z_s))-
\rho(|Y_s|)|+\rho(|Y_s|)
\right]{\rm d}s\right],
\end{array}
$$
we get that, in view of \eqref{eq:2-star}, \ref{A:(H2')}(ii), \ref{A:(H1)}(ii) and \cref{rmk:2-LinearGrowthOfRhoandPhi},
$$
\begin{array}{lll}
\Dis \E\left[\int_0^T|g(s,Y_s,Z_s)|{\rm d}s\right]&\leq & \Dis \E\left[\int_0^T\left(|g(s,Y_s,Z_s)-g(s,0,Z_s)|
+|g(s,0,Z_s)|\right){\rm d}s\right]\vspace{0.1cm}\\
&\leq & \Dis \E\left[|\xi|+|\bar V|_T+
2\int_0^T (\rho(|Y_s|)+|g(s,0,Z_s)|){\rm d}s\right]
\vspace{0.1cm}\\
&\leq & \Dis \E\left[|\xi|+|\bar V|_T+
2\int_0^T (A|Y_s|+A+|g(s,0,0)|+f_s+\lambda |Z_s|^\alpha){\rm d}s\right]
\end{array}
$$
and
$$
\begin{array}{lll}
\Dis \E\left[\int_0^T|g(s,Y_s,0)|{\rm d}s\right]&\leq & \Dis \E\left[\int_0^T\left(|g(s,Y_s,0)-g(s,Y_s,Z_s)|
+|g(s,Y_s,Z_s)|\right){\rm d}s\right]
\vspace{0.1cm}\\
&\leq & \Dis \E\left[\int_0^T|g(s,Y_s,Z_s)|{\rm d}s\right]+\E\left[\int_0^T (f_s+\mu |Y_s|+\lambda |Z_s|^\alpha){\rm d}s\right].
\end{array}
$$
Finally, in view of the conditions of \cref{lem:2-gBelongstoH1} together with H\"{o}lder's inequality, we get \eqref{eq:2-gBelongstoH1}.
\end{proof}

Finally, a similar argument as in Lemma 3.4 of \citet{Fan2017AMS} yields the following two estimates.\vspace{-0.1cm}

\begin{lem}\label{lem:2-EstimateOfZandg}
Let $g$ be a generator and $(Y_\cdot,Z_\cdot,V_\cdot)\in \s\times\M\times\vcal$ satisfy the following equation:
$$
Y_t=Y_T+\int_t^T g(s,Y_s,Z_s){\rm d}s+\int_t^T {\rm d}V_s-\int_t^T Z_s\cdot {\rm d}B_s,\ \ t\in \T.
$$
Assume that there exist two constants $\bar \mu, \bar\lambda>0$ and a nonnegative process $\bar f_\cdot\in \hcal$ such that
\begin{equation}
\label{eq:2-SemiLinearGrowthCondition}
\as,\ \ {\rm sgn}(Y_\cdot)g(\cdot,Y_\cdot,Z_\cdot)\leq \bar f_\cdot+\bar\mu|Y_\cdot|+\bar\lambda|Z_\cdot|.
\end{equation}
Then for each $p>0$, there exists a nonnegative constant $\bar C$ depending only on $p,\bar\mu,\bar\lambda,T$ such that for each $t\in\T$ and each $(\F_t)$-stopping time $\tau$ valued in $\T$, we have
$$
\Dis \E\left[\left.\left(\int_{t\wedge\tau}^\tau|Z_s|^2{\rm d}s\right)^{p\over 2}+\left(\int_{t\wedge\tau}^\tau|g(s,Y_s,Z_s)|{\rm d}s\right)^p\right|\F_t\right]
\leq \Dis \bar C\E\left[\left.\sup\limits_{s\in [t,T]}|Y_{s\wedge\tau}|^p+|V|^p_\tau
+\left(\int_{t\wedge\tau}^\tau \bar f_s\ {\rm d}s\right)^p\right|\F_t\right].\vspace{0.3cm}
$$
\end{lem}

\begin{lem}
\label{lem:2-EstimateOfZKAg}
Let $g$ be a generator and $(Y_\cdot,Z_\cdot,V_\cdot,K_\cdot)\in \s\times\M\times\vcal\times\vcal^+$ satisfy the following equation:
$$
Y_t=Y_T+\int_t^T g(s,Y_s,Z_s){\rm d}s+\int_t^T {\rm d}V_s+\int_t^T {\rm d}K_s-\int_t^T Z_s\cdot {\rm d}B_s,\ \ t\in \T
$$
or
$$
Y_t=Y_T+\int_t^T g(s,Y_s,Z_s){\rm d}s+\int_t^T {\rm d}V_s-\int_t^T {\rm d}K_s-\int_t^T Z_s\cdot {\rm d}B_s,\ \ t\in \T.\vspace{0.2cm}
$$
Assume that there exist two constants $\bar \mu, \bar\lambda>0$ and a nonnegative process $\bar f_\cdot\in \hcal$ such that
\begin{equation}
\label{eq:2-LinearGrowthConditionOfg}
\as,\ \ |g(\cdot,Y_\cdot,Z_\cdot)|\leq \bar f_\cdot+\bar\mu|Y_\cdot|+\bar\lambda|Z_\cdot|.
\end{equation}
Then for each $p>0$, there exists a nonnegative constant $\bar C$ depending only on $p,\bar\mu, \bar\lambda,T$ such that for each $t\in\T$ and each $(\F_t)$-stopping time $\tau$ valued in $\T$, we have
$$
\begin{array}{ll}
&\Dis \E\left[\left.\left(\int_{t\wedge\tau}^\tau|Z_s|^2{\rm d}s\right)^{p\over 2}+|K_\tau-K_{t\wedge\tau}|^p+\left(\int_{t\wedge\tau}^\tau|g(s,Y_s,Z_s)|{\rm d}s\right)^p\right|\F_t\right]\vspace{0.1cm}\\
\leq &\Dis \bar C\E\left[\left.\sup\limits_{s\in [t,T]}|Y_{s\wedge\tau}|^p+|V|^p_\tau
+\left(\int_{t\wedge\tau}^\tau \bar f_s\ {\rm d}s\right)^p\right|\F_t\right].\vspace{0.2cm}
\end{array}
$$
\end{lem}

\section{Penalization, approximation and comparison theorem}
\label{sec:3-PenalizationApproximationComparisonTheorem}
\setcounter{equation}{0}

\subsection{Penalization for RBSDEs\vspace{0.1cm}}

In this subsection, we prove the following convergence result on the sequence of $L^1$ solutions of penalized RBSDEs with one continuous barrier.

\begin{pro} [Penalization for RBSDEs] \label{pro:3-PenalizationOfRBSDE}
Assume that $V_\cdot\in \vcal^1$, \ref{A:(H3)}(i) holds true for $L_\cdot, U_\cdot$ and $\xi$, and $g$ is a generator. We have
\begin{itemize}
\item [(i)] For each $n\geq 1$, let $( Y^n_\cdot, Z^n_\cdot, A^n_\cdot)$ be an $L^1$ solution of $\bar{R}$BSDE $(\xi,\bar g_n+{\rm d}V,U)$ with $\bar g_n(t,y,z):=g(t,y,z)+n(y-L_t)^-$, i.e.,
\begin{equation}
\label{eq:3-PenalizationForRBSDEwithSuperBarrier}
\left\{
\begin{array}{l}
\Dis{Y}^n_t=\xi+\int_t^T\bar g_n(s,{Y}^n_s,
{Z}^n_s)
{\rm d}s+\int_t^T{\rm d}V_s-\int_t^T{\rm d} A^n_s-\int_t^T{Z}^n_s \cdot {\rm d}B_s,\ \   t\in\T,\\
\Dis {Y}^n_t\leq U_t,\ t\in\T\ \ {\rm and} \ \int_0^T (U_t- Y^n_t){\rm d} A^n_t=0,\\
\Dis  K^n_t:=n\int_0^t(
{Y}^n_s-L_s)^-\ {\rm d}s,\ \ t\in\T.
\end{array}
\right.
\end{equation}
If for each $n\geq 1$, $Y^n_\cdot\leq Y^{n+1}_\cdot\leq \bar Y_\cdot$ with a process $\bar Y_\cdot\in \cap_{\beta\in (0,1)}\s^\beta$ of the class (D), ${\rm d} A^n\leq {\rm d} A^{n+1}$, $K^n_\cdot\leq \bar K^n_\cdot\in \vcal^{+,1}$ with $\sup_{n\geq 1}\E[|\bar K^n_T|^\beta]<+\infty$ for each $\beta\in (0,1)$, $\lim\limits_{j\To\infty} \bar K^{n_j}_T=\bar K_T\in \LT$ for a subsequence $\{n_j\}$ of $\{n\}$ and $\sup_{n\geq 1}\E[|\bar K^n_\tau|^2]\leq \E[|\tilde Y_{\tau}|^2]$ for a process $\tilde Y_\cdot \in\s$ and each $(\F_t)$-stopping time $\tau$ valued in $\T$, and there exist two constants $\bar\lambda>0$, $\alpha\in (0,1)$ and a nonnegative process $\bar f_\cdot\in \hcal^1$ such that for each $n\geq 1$,
\begin{equation}
\label{eq:3-SubLinearGrowthofgYnZn}
\as,\ \ |g(\cdot,Y^n_\cdot,Z^n_\cdot)|\leq \bar f_\cdot+\bar\lambda |Z^n_\cdot|^{\alpha},
\end{equation}
then there exists an $L^1$ solution $(Y_\cdot, Z_\cdot, K_\cdot, A_\cdot)$ of DRBSDE $(\xi,g+{\rm d}V,L,U)$ such that
$$
\lim\limits_{n\To \infty}\left(\| Y_\cdot^n- Y_\cdot\|_{\s^\beta}+ \| Z_\cdot^n- Z_\cdot\|_{\M^\beta}+\| A_\cdot^n- A_\cdot \|_{\s^1}\right)=0\vspace{0.1cm}
$$
holds true for each $\beta\in (0,1)$, and there exists a subsequence $\{ K_\cdot^{n_j}\}$ of $\{ K_\cdot^n\}$ such that
$$
\lim\limits_{j\To\infty}\sup\limits_{t\in\T}
| K_t^{n_j}- K_t|=0.
$$

\item [(ii)] For each $n\geq 1$, let $( Y^n_\cdot, Z^n_\cdot, K^n_\cdot)$ be an $L^1$ solution of $\underline{R}$BSDE $(\xi,\underline g_n+{\rm d}V,L)$ with $\underline g_n(t,y,z):=g(t,y,z)-n(y-U_t)^+$, i.e.,
\begin{equation}
\label{eq:3-PenalizationForRBSDEwithLowBarrier}
\left\{
\begin{array}{l}
\Dis{Y}^n_t=\xi+\int_t^T\underline g_n(s,{Y}^n_s, {Z}^n_s)
{\rm d}s+\int_t^T{\rm d}V_s+\int_t^T{\rm d} K^n_t -\int_t^T{Z}^n_s \cdot {\rm d}B_s,\ \   t\in\T,\\
\Dis L_t\leq {Y}^n_t,\ t\in\T\ \ {\rm and} \ \int_0^T ( Y^n_t-L_t){\rm d} K^n_t=0,\\
\Dis  A^n_t:=n\int_0^t(
{Y}^n_s-U_s)^+\ {\rm d}s,\ \ t\in\T.
\end{array}
\right.
\end{equation}
If for each $n\geq 1$, $Y^n_\cdot\geq Y^{n+1}_\cdot\geq \underline Y_\cdot$ with a process $\underline Y_\cdot\in \cap_{\beta\in (0,1)}\s^\beta$ of the class (D), ${\rm d} K^n\leq {\rm d} K^{n+1}$, $A^n_\cdot\leq \bar A^n_\cdot\in \vcal^{+,1}$ with $\sup_{n\geq 1}\E[|\bar A^n_T|^\beta]<+\infty$ for each $\beta\in (0,1)$, $\lim\limits_{j\To\infty} \bar A^{n_j}_T=\bar A_T\in \LT$ for a subsequence $\{n_j\}$ of $\{n\}$ and $\sup_{n\geq 1}\E[|\bar A^n_\tau|^2]\leq \E[|\tilde Y_{\tau}|^2]$ for a process $\tilde Y_\cdot \in\s$ and each $(\F_t)$-stopping time $\tau$ valued in $\T$, and there exist two constants $\bar\lambda>0$, $\alpha\in (0,1)$ and a nonnegative process $\bar f_\cdot\in \hcal^1$ such that \eqref{eq:3-SubLinearGrowthofgYnZn} holds for each $n\geq 1$, then there exists an $L^1$ solution $( Y_\cdot, Z_\cdot, K_\cdot, A_\cdot)$ of DRBSDE $(\xi,g+{\rm d}V,L,U)$ such that
$$\lim\limits_{n\To \infty}\left(\| Y_\cdot^n- Y_\cdot\|_{\s^\beta}+
\| Z_\cdot^n- Z_\cdot\|_{\M^\beta}+\| K_\cdot^n- K_\cdot \|_{\s^1}\right)=0$$
holds true for each $\beta\in (0,1)$, and there exists a subsequence $\{ A_\cdot^{n_j}\}$ of $\{ A_\cdot^n\}$ such that
$$\lim\limits_{j\To\infty}\sup\limits_{t\in\T}
| A_t^{n_j}- A_t|=0.$$
\end{itemize}
\end{pro}

\begin{proof} We only prove the claim (i). The claim (ii) can be proved in the same way. Now we assume that all the assumptions in (i) are satisfied. Since $Y_\cdot^n$ increases in $n$, we know that there exists an $(\F_t)$-progressively measurable process $Y_\cdot$ such that $Y_t^n\uparrow Y_t$ for each $t\in\T$. In view of \eqref{eq:3-PenalizationForRBSDEwithSuperBarrier} and \eqref{eq:3-SubLinearGrowthofgYnZn} with the fact that for each $n\geq 1$, $Y_\cdot^1\leq Y^n_\cdot\leq \bar Y_\cdot$ and $K^n_T\leq \bar K^n_T$ with $\sup_{n\geq 1}\E[|\bar K^n_T|^\beta]<+\infty$ for each $\beta\in (0,1)$, by \cref{lem:2-EstimateOfZKAg} we deduce that for each $\beta\in (0,1)$, there exists a $C_\beta>0$ depending only on $\beta,\bar\lambda,T$ such that
\begin{equation}\label{eq:3-BoundofZntAntandg}
\begin{array}{ll}
&\Dis \sup\limits_{n\geq 1}\E\left[\left(\int_0^T|Z^n_s|^2{\rm d}s\right)^{\beta\over 2}+|A^n_T|^\beta+\left(\int_0^T|g(s,Y^n_s,Z^n_s)|{\rm d}s\right)^\beta\right]\\
\leq &\Dis C_\beta\left\{\E\left[\sup\limits_{s\in [0,T]}(|Y^1_{s}|+|\bar Y_{s}|)^\beta+|V|^\beta_T+\left(\int_0^T \bar f_s\ {\rm d}s\right)^\beta\right]+\sup\limits_{n\geq 1} \E\left[|K^n_T|^\beta\right]\right\}<+\infty.
\vspace{0.1cm}
\end{array}
\end{equation}
For each positive integer $k\geq 1$, we define the following $(\F_t)$-stopping time:
$$
\tau_k:=\Dis \inf\left\{t\geq 0:\  |Y^1_{t}|+|\bar Y_{t}|+|V|_t+\int_0^t \bar f_s\ {\rm d}s+|\tilde Y_t|+L^+_t\geq k\right\}\wedge T.
$$
Then
\begin{equation}\label{eq:3-StabilityOfTauk}
\mathbb{P}\left(\left\{\omega:\ \exists k_0(\omega)\geq 1, \ \RE k\geq k_0(\omega),\ \tau_k(\omega)=T\right\}\right)=1.\vspace{0.1cm}
\end{equation}
Note the fact that $\sup_{n\geq 1}\E[|K^n_{\tau_k}|^2]\leq \E[|\tilde Y_{\tau_k}|^2]\leq k^2$ for each $k\geq 1$. Again by \cref{lem:2-EstimateOfZKAg} we deduce that there exists a nonnegative constant $\bar C$ depending only on $\bar\lambda,T$ such that for each $k\geq 1$,
\begin{equation}
\label{eq:3-BoundofZntAntandgTillTauk}
\begin{array}{ll}
&\Dis \sup\limits_{n\geq 1}\E\left[\int_0^{\tau_k}|Z^n_s|^2{\rm d}s+|A^n_{\tau_k}|^2+\left(\int_0^{\tau_k}
|g(s,Y^n_s,Z^n_s)|{\rm d}s\right)^2\right]\vspace{0.1cm}\\
\leq &\Dis \bar C\E\left[\sup\limits_{s\in [0,T]}(|Y^1_{s\wedge \tau_k}|+|\bar Y_{s\wedge \tau_k}|)^2+|V|^2_{\tau_k}+
\left(\int_0^{\tau_k} \bar f_s\ {\rm d}s\right)^2+|\tilde Y_{\tau_k}|^2\right]
\leq 4\bar Ck^2.
\end{array}
\end{equation}

Furthermore, since ${\rm d}A^n\leq {\rm d}A^{n+1}$, there exists an $(\F_t)$-progressively measurable and increasing process $(A_t)_{t\in\T}$ with $A_0=0$ such that $A_t^n\uparrow A_t$ for each $t\in\T$, and for each $j\geq n\geq 1$,
$$0\leq A^j_t-A^n_t\leq A^j_T-A^n_T,\ \ t\in\T.$$
Letting first $j\To\infty$, and then taking superume with respect to $t$ in $\T$, finally letting $n\To\infty$ in the previous inequality yields that
\begin{equation}
\label{eq:3-UniformConvergenceOfAn}
\Lim\sup\limits_{t\in\T}|A^n_t-A_t|=0,
\vspace{0.1cm}
\end{equation}
which means that $A_\cdot\in\vcal^+$. On the other hand, note by \eqref{eq:3-BoundofZntAntandgTillTauk} that $\sup_{n\geq 1}\E[|A^n_{\tau_k}|^2]<+\infty$ for each $k\geq 1$. It follows that for each $(\F_t)$-stopping time $\tau$ valued in $\T$ and each $k\geq 1$,
\begin{equation}\label{eq:3-ConvergenceOfAnInL1}
\Lim \E[|A^n_{\tau\wedge\tau_k}-A_{\tau\wedge\tau_k}|]=0.
\end{equation}

The rest proof of this proposition is divided into 7 steps, some ideas among them are lent from the proof of Proposition 4.3 in \citet{Fan2017arXivLpSolutionofDRBSDEs}.\vspace{0.2cm}

{\bf Step 1.}\ We show that $Y_\cdot$ is a c\`{a}dl\`{a}g process. Let us first fix a positive integer $k\geq 1$ arbitrarily. Note that $\bar f_\cdot\in \hcal^1$ and $\sup_{n\geq 1}\|Z^n_\cdot\mathbbm{1}_{\cdot\leq \tau_k}\|_{\M^2}<+\infty$ by \eqref{eq:3-BoundofZntAntandgTillTauk}. It follows from \eqref{eq:3-SubLinearGrowthofgYnZn} that there exists a subsequence $\{g(\cdot,Y^{n_j}_\cdot,Z^{n_j}_\cdot)
\mathbbm{1}_{\cdot\leq \tau_k}\}_{j=1}^{\infty}$ of the sequence $\{g(\cdot,Y^{n}_\cdot,Z^{n}_\cdot)
\mathbbm{1}_{\cdot\leq \tau_k}\}_{n=1}^{\infty}$ which converges weakly to a process ${}^kh_\cdot$ in $\hcal^1$. Then, for every $(\F_t)$-stopping time $\tau$ valued in $\T$, as $j\To \infty$, we have
\begin{equation}
\label{eq:3-WeaklyConvergenceOfGYnZn}
\int_0^\tau \mathbbm{1}_{s\leq \tau_k} g(s,Y_s^{n_j},Z_s^{n_j})
{\rm d}s\ \To\  \int_0^\tau {}^kh_s {\rm d}s\ \ {\rm weakly\ in}\ \LT.
\end{equation}
Furthermore, since
$$
\sup_{n\geq 1}\E\left[\int_0^T |Z^n_t\mathbbm{1}_{t\leq \tau_k}|^2{\rm d}t\right]<+\infty,\vspace{0.1cm}
$$
it follows from Lemma 4.4 of \citet{Klimsiak2012EJP} that there exists a process ${}^kZ_\cdot\in \M^2$ and a subsequence of the sequence $\{n_j\}_{j=1}^{\infty}$, still denoted by itself, such that for every $(\F_t)$-stopping time $\tau$ valued in $\T$,
\begin{equation}
\label{eq:3-WeaklyConvergenceOfZn}
\int_0^{\tau} \mathbbm{1}_{s\leq \tau_k}Z_s^{n_j}\cdot {\rm d}B_s\To \int_0^{\tau} {}^kZ_s\cdot {\rm d}B_s\ \ {\rm weakly\ in}\ \mathbb{L}^2(\F_T)\ {\rm and \ then\ in}\ \mathbb{L}^1(\F_T),\ \
{\rm as}\ j\To \infty.
\end{equation}
In the sequel, we define
$$
{}^kK_t:=Y_0-Y_t-\int_0^t {}^kh_s{\rm d}s-\int_0^t {\rm d}V_s-\int_0^t {\rm d}A_s+\int_0^t {}^kZ_s\cdot {\rm d}B_s, \ \ t\in\T.
$$
Then, for each $(\F_t)$-stopping time $\tau$ valued in $\T$, in view of \eqref{eq:3-ConvergenceOfAnInL1}, \eqref{eq:3-WeaklyConvergenceOfGYnZn}, \eqref{eq:3-WeaklyConvergenceOfZn} and the fact that $Y^n_{\tau}\uparrow Y_{\tau}$ in $\mathbb{L}^1(\F_T)$, we can deduce that the sequence
$$
K_{\tau\wedge \tau_k}^{n_j}=Y_0^{n_j}
-Y_{\tau\wedge \tau_k}^{n_j}-\int_0^{\tau\wedge \tau_k}g(s,Y_s^{n_j},Z_s^{n_j})
{\rm d}s-\int_0^{\tau\wedge \tau_k}{\rm d}V_s-\int_0^{\tau\wedge \tau_k} {\rm d}A_s^{n_j}+\int_0^{\tau\wedge \tau_k} Z_s^{n_j}\cdot {\rm d}B_s
$$
converges weakly to ${}^kK_{\tau\wedge \tau_k}$ in $\mathbb{L}^1(\F_T)$ as $j\To \infty$. Thus, since $K_\cdot^n\in \vcal^{+}$ for each $n\geq 1$, we know that\vspace{-0.1cm}
$$
{}^kK_{\sigma_1\wedge \tau_k}\leq {}^kK_{\sigma_2\wedge \tau_k}
\vspace{-0.2cm}
$$
for any $(\F_t)$-stopping times $\sigma_1\leq \sigma_2$ valued in $\T$. Furthermore, in view of the definition of ${}^kK_\cdot$ together with the facts that $V_\cdot\in\vcal$, $A_\cdot\in\vcal^+$, $Y^n_\cdot\uparrow Y_\cdot$ and $Y^n_\cdot\in \s$ for each $n\geq 1$, it is not hard to check that ${}^kK_{\cdot}$ is a optional process with $\ps$ upper semi-continuous paths. Thus, Lemma A.3 in \citet{BayraktarYao2015SPA} yields that ${}^kK_{\cdot\wedge \tau_k}$ is a nondecreasing process, and then it has $\ps$ right lower semi-continuous paths. Hence, ${}^kK_{\cdot\wedge \tau_k}$ is c\`{a}dl\`{a}g and so is $Y_{\cdot\wedge \tau_k}$ from the definition of ${}^kK_\cdot$. Finally, it follows from \eqref{eq:3-StabilityOfTauk} that $Y_\cdot$ is also a c\`{a}dl\`{a}g process.\vspace{0.2cm}

{\bf Step 2.}\ We show that $Y_t\geq L_t$ for each $t\in \T$ and
\begin{equation}
\label{eq:3-UniformConvergenceOfYnMinusL}
\Lim\sup\limits_{t\in\T}(Y^n_t-L_t)^-= 0.
\end{equation}
In fact, it follows from Fatou's lemma and the definition of $K^n_\cdot$ that for each $\beta\in (0,1)$,
$$0\leq\E\left[\left(\int_0^T(Y_t-L_t)^-{\rm d}t\right)^\beta\right]\leq \liminf\limits_{n\To\infty} \E\left[\left(\int_0^T(Y^n_t-L_t)^-{\rm d}t\right)^\beta\right]\leq \lim\limits_{n\To\infty}{\sup_{n\geq 1}\E[|K^n_T|^\beta]\over n^\beta}=0.$$
Since $Y_\cdot-L_\cdot$ is a c\`{a}dl\`{a}g process, it follows that $(Y_t-L_t)^-=0$ and hence
$Y_t\geq L_t$ for each $t\in [0,T)$. Moreover, $Y_T=Y^n_T=\xi\geq L_T$. Hence $$(Y^n_t-L_t)^-\downarrow 0$$
for each $t\in [0,T]$ and by Dini's theorem, \eqref{eq:3-UniformConvergenceOfYnMinusL} follows. \vspace{0.2cm}

{\bf Step 3.}\ We show the convergence of the sequence $\{Y_\cdot^n\}$. For each $n,m\geq 1$, observe that
\begin{equation}
\label{eq:3-DefinitionOfBarYBarZBarV}
\begin{array}{lll}
\Dis (\bar Y_\cdot,\bar Z_\cdot,\bar V_\cdot)&
:=&\Dis (Y_\cdot^n-Y_\cdot^m,Z_\cdot^n-Z_\cdot^m,\\
&&\Dis \ \ \int_0^\cdot \left(g(s,Y_s^n,Z_s^n)-g(s,Y_s^m,Z_s^m)\right){\rm d}s+\left(K_\cdot^n-K_\cdot^m\right)
-\left(A_\cdot^n-A_\cdot^m\right))
\end{array}
\end{equation}
satisfies equation \eqref{eq:2-BarY=BarV}. It then follows from (ii) of \cref{lem:2-Lemma1} with $p=2$, $t=0$ and $\tau=\tau_k$ that there exists a constant $C>0$ such that for each $n,m,k\geq 1$,
\begin{equation}\label{eq:3-BoundOfYnMinusYm}
\begin{array}{lll}
\Dis\E\left[\sup\limits_{t\in [0,T]} |Y_{t\wedge\tau_k}^n
-Y_{t\wedge\tau_k}^m|^2\right]
&\leq &\Dis C\E\left[|Y_{\tau_k}^n-Y_{\tau_k}^m|^2 +\sup\limits_{t\in [0,T]}\left(\int_{t\wedge \tau_k}^{\tau_k}(Y^n_s-Y_s^m)
\left({\rm d}K_s^n-{\rm d}K_s^m\right) \right)^+\right.\\
&&\hspace{1cm}\Dis +\sup\limits_{t\in [0,T]}\left(\int_{t\wedge \tau_k}^{\tau_k}(Y^n_s-Y_s^m)
\left({\rm d}A_s^m-{\rm d}A_s^n\right) \right)^+\vspace{0.1cm}\\
&&\hspace{1cm}\Dis +\left. \int_{0}^{\tau_k}|Y^n_s-Y_s^m|
\left|g(s,Y_s^n,Z_s^n)-g(s,Y_s^m,Z_s^m)\right| {\rm d}s\right].
\end{array}
\end{equation}
Furthermore, by virtue of the definition of $K_\cdot^n$ and $A_\cdot^n$ with \eqref{eq:3-PenalizationForRBSDEwithSuperBarrier} we know that for each $t\in \T$,
\begin{equation}\label{eq:3-YnYmTimesKnKmLeq}
\begin{array}{ll}
&\Dis \int_{t\wedge \tau_k}^{\tau_k}(Y^n_s-Y_s^m)
\left({\rm d}K_s^n-{\rm d}K_s^m\right)\vspace{0.1cm}\\
= &\Dis \int_{t\wedge \tau_k}^{\tau_k}\left[(Y^n_s-L_s)
-(Y_s^m-L_s)\right]
{\rm d}K_s^n-\int_{t\wedge \tau_k}^{\tau_k}\left[(Y^n_s-L_s)
-(Y_s^m-L_s)\right]
{\rm d}K_s^m\vspace{0.1cm}\\
\leq &\Dis \int_{t\wedge \tau_k}^{\tau_k}(Y_s^m-L_s)^-
{\rm d}K_s^n+\int_{t\wedge \tau_k}^{\tau_k}(Y_s^n-L_s)^-{\rm d}K_s^m\vspace{0.1cm}\\
\leq &\Dis \sup\limits_{s\in\T}(Y_{s\wedge\tau_k }^m-L_{s\wedge\tau_k})^-|K_T^n|+
\sup\limits_{s\in\T}(Y_{s\wedge\tau_k }^n-L_{s\wedge\tau_k})^-|K_T^m|
\end{array}
\end{equation}
and
\begin{equation}\label{eq:3-YnYmTimesAmAnLeq}
\begin{array}{lll}
\Dis \int_{t\wedge \tau_k}^{\tau_k}(Y^n_s-Y_s^m)
\left({\rm d}A_s^m-{\rm d}A_s^n\right)&
= & \Dis \int_{t\wedge \tau_k}^{\tau_k}\left[(U_s-Y_s^m)
-(U_s-Y_s^n)\right]\left({\rm d}A_s^m-{\rm d}A_s^n\right)\vspace{0.1cm}\\
&= & \Dis -\int_{t\wedge \tau_k}^{\tau_k}(U_s-Y_s^m)
{\rm d}A_s^n-\int_{t\wedge \tau_k}^{\tau_k}(U_s-Y_s^n)
{\rm d}A_s^m\\
&\leq & 0.
\end{array}
\end{equation}
Combining \eqref{eq:3-SubLinearGrowthofgYnZn}, \eqref{eq:3-BoundOfYnMinusYm}, \eqref{eq:3-YnYmTimesKnKmLeq} and \eqref{eq:3-YnYmTimesAmAnLeq} with H\"{o}lder's inequality yields that for each $m,n,k\geq 1$,
\begin{equation}
\label{eq:3-FurtherBoundOfYnMinusYm}
\begin{array}{lll}
\Dis\E\left[\sup\limits_{t\in [0,T]} |Y_{t\wedge\tau_k}^n
-Y_{t\wedge\tau_k}^m|^2\right]&\leq &\Dis C\E\left[|Y_{\tau_k}^n-Y_{\tau_k}^m|^2 +2\int_0^{\tau_k} |Y^n_t-Y_t^m|(\bar f_t+\bar\lambda){\rm d}t\right]\vspace{0.1cm}\\
&&\Dis +C\left(\E\left[ \sup\limits_{t\in\T}\left|(Y_{t\wedge\tau_k }^m-L_{t\wedge\tau_k})^-\right|^2\right]\right)^{1\over 2}\left(\E\left[|K_{\tau_k}^n|^2\right]\right)^{1\over 2}\vspace{0.1cm}\\
&&\Dis+C\left(\E\left[ \sup\limits_{t\in\T}\left|(Y_{t\wedge\tau_k }^n-L_{t\wedge\tau_k})^-\right|^2\right]\right)^{1\over 2}\left(\E\left[|K_{\tau_k}^m|^2\right]\right)^{1\over 2}\vspace{0.1cm}\\
&&\Dis+2C\bar\lambda\left(\E\left[\int_{0}^{
\tau_k}|Y^n_t-Y_t^m|^2{\rm d}t\right]\right)^{1\over 2} \left(\E\left[\int_{0}^{\tau_k}\left(|Z_t^n|+
|Z_t^m|\right)^2{\rm d}t\right]\right)^{1\over 2}.
\end{array}
\end{equation}
Note that $Y_\cdot^n\uparrow Y_\cdot$, $\bar f_\cdot\in \hcal^1$, $|Y^1_{\cdot\wedge\tau_k} |+|\bar Y_{\cdot\wedge\tau_k}|
+L^+_{\cdot\wedge\tau_k}\leq k$ and $\sup_{n\geq 1}(\E[|K^n_{\tau_k}|^2]+\|Z^n_\cdot\mathbbm{1}_{\cdot\leq \tau_k}\|_{\M^2})<+\infty$ for each $k\geq 1$ by \eqref{eq:3-BoundofZntAntandgTillTauk}. In view of \eqref{eq:3-UniformConvergenceOfYnMinusL}, from \eqref{eq:3-FurtherBoundOfYnMinusYm} and Lebesgue's dominated convergence theorem it follows that for each $k\geq 1$, as $n,m\To\infty$,
$$\E\left[\sup\limits_{t\in [0,T]} |Y_{t\wedge\tau_k}^n
-Y_{t\wedge\tau_k}^m|^2\right]\To 0,\vspace{0.1cm}$$
which implies that for each $k\geq 1$, as $n,m\To\infty$,
$$\sup\limits_{t\in [0,T]} |Y_{t\wedge\tau_k}^n
-Y_{t\wedge\tau_k}^m|\To 0\ {\rm in\ probability}\ \mathbb{P}.$$
And, by \eqref{eq:3-StabilityOfTauk} and the fact that $Y_\cdot^n\uparrow Y_\cdot$ we know that \begin{equation}
\label{eq:3-UniformConvergenceOfYnInProb}
\sup\limits_{t\in [0,T]} |Y_t^n
-Y_t|\To 0,\ \ {\rm as}\ n\To\infty.
\end{equation}
So, $Y_\cdot$ is a continuous process, and then belongs to the space $\s^\beta$ for each $\beta\in (0,1)$ and the class (D) due to the fact that both $Y_\cdot^1$ and $\bar Y_\cdot$ belong to them as well as $Y_\cdot^1\leq Y^n_\cdot\leq \bar Y_\cdot$. Finally, from \eqref{eq:3-UniformConvergenceOfYnInProb} and Lebesgue's dominated convergence theorem it follows that for each $\beta\in (0,1)$,
\begin{equation}
\label{eq:3-ConvergenceOfYnInSp}
\lim\limits_{n\To\infty}\|Y_\cdot^n-Y_\cdot\|_{\s^\beta}^\beta =\lim\limits_{n\To\infty}\E\left[\sup\limits_{t\in [0,T]} |Y_t^n-Y_t|^\beta\right]=0.\vspace{0.2cm}
\end{equation}

{\bf Step 4.}\ We show the convergence of the sequence $\{Z_\cdot^n\}$. Note that \eqref{eq:3-DefinitionOfBarYBarZBarV} solves \eqref{eq:2-BarY=BarV}. It follows from (i) of \cref{lem:2-Lemma1} with $t=0$ and $\tau=T$ that there exists a nonnegative constant $C'\geq 0$ such that for each $m,n\geq 1$ and $\beta\in (0,1)$, we have
$$
\begin{array}{lll}
&&\Dis\E\left[\left(\int_0^T|Z_t^n-Z_t^m|^2{\rm d}t\right)^{\beta\over 2}\right]\\
&\leq & \Dis C'\E\left[\sup\limits_{t\in [0,T]}|Y_t^n-Y_t^m|^\beta+\sup\limits_{t\in [0,T]}\left[\left(\int_t^T (Y_s^n-Y_s^m)\left({\rm d}K_s^n-{\rm d}K_s^m\right)\right)^+\right]^{\beta\over 2}\right] \\
&&\Dis +C'\E\left[\sup\limits_{t\in [0,T]}\left[\left(\int_t^T (Y_s^n-Y_s^m)\left({\rm d}A_s^m-{\rm d}A_s^n\right)\right)^+\right]^{\beta\over 2}\right]\\
&&\Dis + C'\E\left[\left(\int_{0}^T|Y^n_s-Y_s^m|
\left|g(s,Y_s^n,Z_s^n)-g(s,Y_s^m,Z_s^m)\right| {\rm d}s\right)^{\beta\over 2}\right].
\end{array}
$$
Then, it follows from H\"{o}lder's inequality together with \eqref{eq:3-YnYmTimesAmAnLeq} that
$$
\begin{array}{ll}
&\Dis\E\left[\left(\int_0^T|Z_t^n-Z_t^m|^2{\rm d}t\right)^{\beta\over 2}\right]\\
\leq & \Dis C'\E\left[\sup\limits_{t\in [0,T]}|Y_t^n-Y_t^m|^\beta\right]+\left.C'\left(\E\left[\sup\limits_{t\in [0,T]}|Y_t^n-Y_t^m|^\beta\right]\right)^{1\over 2}\right\{\left(\E\left[|K_T^n|^\beta\right]\right)^{1\over 2}\\
&\Dis+\left.\left(\E\left[|K_T^m|^\beta\right]\right)^{1\over 2}+\left(\E\left[\left(\int_{0}^T
\left(|g(t,Y_t^n,Z_t^n)|+|g(t,Y_t^m,Z_t^m)|\right) {\rm d}t\right)^{\beta}\right]\right)^{1\over 2}\right\},
\end{array}
$$
from which together with \eqref{eq:3-BoundofZntAntandg} and \eqref{eq:3-ConvergenceOfYnInSp} yields that there exists a process $(Z_t)_{t\in\T}\in \cap_{\beta\in (0,1)}\M^\beta$ satisfying, for each $\beta\in (0,1)$,
\begin{equation}
\label{eq:3-ConvergenceOfZnInMp}
\lim\limits_{n\To\infty}\|Z_\cdot^n-Z_\cdot\|_{\M^\beta}^\beta
=\lim\limits_{n\To\infty}\E\left[\left(\int_0^T|Z_t^n
-Z_t|^2{\rm d}t\right)^{\beta\over 2}\right]=0.\vspace{0.2cm}
\end{equation}

{\bf Step 5.}\ We show the convergence of the sequence $\{K_\cdot^n\}$. Since $g$ is continuous in $(y,z)$ and satisfies \eqref{eq:3-SubLinearGrowthofgYnZn}, by \eqref{eq:3-UniformConvergenceOfYnInProb} and \eqref{eq:3-ConvergenceOfZnInMp} we can deduce that there exists a subsequence $\{n_j\}$ of $\{n\}$ such that
$$\lim\limits_{j\To \infty}\int_0^T|g(t,Y_t^{n_j},Z_t^{n_j})-
g(t,Y_t,Z_t)|{\rm d}t=0,$$
and then
\begin{equation}
\label{eq:3-UniformConvergenceOfgYnZn}
\lim\limits_{j\To \infty}\sup\limits_{t\in\T}\left|\int_0^t
g(t,Y_t^{n_j},Z_t^{n_j}){\rm d}t-\int_0^t
g(t,Y_t,Z_t){\rm d}t\right|=0.\vspace{0.2cm}
\end{equation}
Thus, combining \eqref{eq:3-UniformConvergenceOfAn}, \eqref{eq:3-UniformConvergenceOfYnInProb}, \eqref{eq:3-ConvergenceOfZnInMp} and \eqref{eq:3-UniformConvergenceOfgYnZn} yields that $\ps$, for each $t\in\T$,
$$
K_t^{n_j}=Y_0^{n_j}-Y_t^{n_j}-\int_0^t
g(s,Y_s^{n_j},Z_s^{n_j}){\rm d}s-\int_0^t{\rm d}V_s-A_t^{n_j}+\int_0^tZ_s^{n_j}\cdot {\rm d}B_s
$$
tends to
$$
K_t:=Y_0-Y_t-\int_0^t
g(s,Y_s,Z_s){\rm d}s-\int_0^t{\rm d}V_s-A_t+\int_0^tZ_s\cdot {\rm d}B_s
$$
as $j\To \infty$ and that
\begin{equation}
\label{eq:3-UniformConvergenceOfKnPas}
\lim\limits_{j\To\infty}\sup\limits_{t\in\T}
|K_t^{n_j}-K_t|=0.\vspace{0.1cm}
\end{equation}
Hence, $K_\cdot\in \vcal^+$ due to $K^n_\cdot\in \vcal^+$ for each $n\geq 1$. Furthermore, note by the assumption that $K^n_T\leq \bar K^n_T$ for each $n\geq 1$ with $\lim\limits_{k\To\infty} \bar K^{n_k}_T=\bar K_T\in \LT$ for a subsequence $\{n_k\}$ of $\{n\}$. It follows that $K_T\in \LT$ and then $K_\cdot\in \vcal^{+,1}$.\vspace{0.2cm}

{\bf Step 6.}\ We show that the convergence of the sequence $\{A_\cdot^n\}$. Indeed, for each $k\geq 1$, define the following $(\F_t)$-stopping time:
$$
\sigma_k:=\inf\{t\in\T:\ \ \int_0^t |Z_s|^2{\rm d}s\geq k\}\wedge T.
$$
It is clear that $\sigma_k\To T$ as $k\To +\infty$ due to the fact that $Z_\cdot\in \M$. For each $k\geq 1$, we have
$$
A_{\sigma_k}=Y_{\sigma_k}-Y_0 +\int_0^{\sigma_k}g(s,Y_s,Z_s){\rm d}s+\int_0^{\sigma_k}{\rm d}V_s+K_{\sigma_k}-\int_0^{\sigma_k}Z_s \cdot {\rm d}B_s,
$$
and then
$$
\E[A_{\sigma_k}]\leq |Y_0|+\E\left[|Y_{\sigma_k}|+\int_0^T |g(s,Y_s,Z_s)|{\rm d}s+
|V|_T+K_T\right].\vspace{0.1cm}
$$
Letting $k\To\infty$, in view of Fatou's lemma and the fact that $Y_\cdot$ belongs to the class (D), yields that
$$
\E[A_T]\leq |Y_0|+\E\left[|\xi|+\int_0^T |g(s,Y_s,Z_s)|{\rm d}s+|V|_T+K_T\right].
$$
Furthermore, in view of \eqref{eq:3-UniformConvergenceOfgYnZn} and \eqref{eq:3-SubLinearGrowthofgYnZn}, it follows from H\"{o}lder's inequality that
$$
\begin{array}{lll}
\Dis\E\left[\int_0^T |g(s,Y_s,Z_s)|{\rm d}s\right]&=&\Dis
\E\left[\lim\limits_{j\To \infty}\int_0^T|g(t,Y_t^{n_j},Z_t^{n_j})|{\rm d}t\right]\leq \E\left[\lim\limits_{j\To \infty}\int_0^T(\bar f_t+\bar\lambda |Z_t^{n_j}|^\alpha){\rm d}t\right]\vspace{0.1cm}\\
&=& \Dis \E\left[\int_0^T(\bar f_t+\bar\lambda|Z_t|^\alpha){\rm d}t\right]\leq \|\bar f_\cdot\|_{\hcal^1}+\bar\lambda T^{2-\alpha\over 2}\|Z_\cdot\|_{\M^\alpha}<+\infty.
\end{array}
$$
Thus, we have $\E[A_T]<\infty$ and $A_\cdot\in \vcal^{+,1}$. Finally, note that $0\leq A^n_\cdot \leq A_\cdot$ for each $n\geq 1$. From \eqref{eq:3-UniformConvergenceOfAn} and Lebesgue's dominated convergence theorem it follows that
\begin{equation}\label{eq:3-ConvergenceOfAnInSp}
\Lim\|A^n_\cdot-A_\cdot\|_{\s^1}=0.
\end{equation}

{\bf Step 7.}\ We show that $(Y_\cdot,Z_\cdot, K_\cdot, A_\cdot)$ is an $L^1$ solution of RBSDE $(\xi,g+{\rm d}V,L,U)$. In fact, it has been proved that $Y_\cdot$ belongs to the class (D), $(Y_\cdot,Z_\cdot,K_\cdot,A_\cdot)\in \s^\beta\times\M^\beta\times\vcal^{+,1}
\times\vcal^{+,1}$ for each $\beta\in (0,1)$ and it solves
$$
Y_t=\xi+\int_t^Tg(s,Y_s,Z_s){\rm d}s+\int_t^T{\rm d}V_s+\int_t^T{\rm d}K_s-\int_t^T{\rm d}A_s-\int_t^TZ_s \cdot {\rm d}B_s,\ \ t\in\T.\vspace{0.1cm}
$$
By Step 2 we know that $Y_t\geq L_t$ for each $t\in\T$, and then
$$
\int_0^T(Y_t-L_t){\rm d}K_t\geq 0.
$$
On the other hand, in view of \eqref{eq:3-UniformConvergenceOfYnInProb} and \eqref{eq:3-UniformConvergenceOfKnPas}, it follows from the definition of $K_\cdot^n$ that
$$
\int_0^T(Y_t-L_t){\rm d}K_t=\lim\limits_{j\To \infty}\int_0^T(Y_t^{n_j}-L_t){\rm d}K_t^{n_j}\leq 0.\vspace{0.1cm}
$$
Consequently, we have
$$
\int_0^T(Y_t-L_t){\rm d}K_t=0.\vspace{0.1cm}
$$
Furthermore, noticing that $Y^n_\cdot\leq U_\cdot$ and $\int_0^T (U_t-Y^n_t)\ {\rm d}A^n_t=0$ for each $n\geq 1$ , from \eqref{eq:3-ConvergenceOfYnInSp} and \eqref{eq:3-ConvergenceOfAnInSp} we can deduce that $Y_t\leq U_t$ for each $t\in\T$, and
$$
\int_0^T (U_t-Y_t)\ {\rm d}A_t=\Lim \int_0^T (U_t-Y^n_t)\ {\rm d}A^n_t=0.
$$
Finally, let us show that ${\rm d}K\bot{\rm d}A$. In fact, for each $n\geq 1$, we can define the following $(\F_t)$-progressively measurable set
$$D_n:=\{(\omega,t)\subset \Omega\times\T:\ Y^n_t(\omega)\geq L_t(\omega)\}.$$
Then, from the definition of $K^n_\cdot$ we know that for each $n\geq 1$,
$$\E\left[\int_0^T \mathbbm{1}_{D_n} {\rm d}K^n_t\right]=0,$$
and, in view of $\int_0^T (U_t-Y^n_t){\rm d}A^n_t=0$,
$$
\E\left[\int_0^T \mathbbm{1}_{D_n^c}{\rm d}A^n_t\right]=\E\left[\int_0^T \mathbbm{1}_{\{Y^n_t<L_t\leq U_t\}}
|U_t-Y^n_t|^{-1}
(U_t-Y^n_t)\ {\rm d}A^n_t\right]=0.
$$
Thus, noticing that $D_n\subset D_{n+1}$ for each $n\geq1$ due to $Y^n_\cdot\leq Y^{n+1}_\cdot$, by \eqref{eq:3-UniformConvergenceOfKnPas} and \eqref{eq:3-ConvergenceOfAnInSp} we can deduce that
$$
\E\left[\int_0^T\mathbbm{1}_{\cup D_n} {\rm d}K_t\right]=\lim\limits_{j\To\infty} \E\left[\int_0^T
\mathbbm{1}_{D_{n_j}} {\rm d}K^{n_j}_t\right]=0
$$
and
$$
\E\left[\int_0^T\mathbbm{1}_{\cap D_n^c} {\rm d}A_t\right]=\Lim\E\left[\int_0^T
\mathbbm{1}_{D_n^c}{\rm d}A^n_t\right]=0.\vspace{0.1cm}
$$
Hence, ${\rm d}K\bot{\rm d}A$. \cref{pro:3-PenalizationOfRBSDE} is then proved.
\end{proof}

\subsection{Penalization for BSDEs\vspace{0.1cm}}

In this subsection, we prove the following convergence result on the sequence of $L^1$ solutions of penalized non-reflected BSDEs.\vspace{-0.1cm}

\begin{pro} [Penalization for BSDEs] \label{pro:3-PenalizationOfBSDE}
Assume that $V_\cdot\in \vcal^1$, \ref{A:(H3)}(i) holds true for $L_\cdot, U_\cdot$ and $\xi$, and $g$ is a generator. We have
\begin{itemize}
\item [(i)] Let $( Y^n_\cdot, Z^n_\cdot)$ be an $L^1$ solution of BSDE $(\xi,\bar g_n+{\rm d}V)$ with $\bar g_n(t,y,z):=g(t,y,z)+n(y-L_t)^-$ for each $n\geq 1$, i.e.,
\begin{equation}
\label{eq:3-PenalizationForBSDE}
\left\{
\begin{array}{l}
\Dis{Y}^n_t=\xi+\int_t^T\bar g_n(s,{Y}^n_s,
{Z}^n_s)
{\rm d}s+\int_t^T{\rm d}V_s-\int_t^T{Z}^n_s \cdot {\rm d}B_s,\ \   t\in\T,\vspace{0.1cm}\\
\Dis  K^n_t:=n\int_0^t(
{Y}^n_s-L_s)^-\ {\rm d}s,\ \ t\in\T.
\end{array}
\right.
\end{equation}
If for each $n\geq 1$, $Y^n_\cdot\leq Y^{n+1}_\cdot\leq \bar Y_\cdot$ with a process $\bar Y_\cdot\in \cap_{\beta\in (0,1)}\s^\beta$ of the class (D), and there exist two constants $\bar\lambda>0$, $\alpha\in (0,1)$ and a nonnegative process $\bar f_\cdot\in \hcal^1$ such that \eqref{eq:3-SubLinearGrowthofgYnZn} holds true for each $n\geq 1$, then $\sup_{n\geq 1}\E[|K^n_T|^\beta]<+\infty$ for each $\beta\in (0,1)$, $\sup_{n\geq 1}\E[|K^n_\tau|^2]\leq \E[|\tilde Y_{\tau}|^2]$ for a process
$\tilde Y_\cdot\in\s$ and each $(\F_t)$-stopping time $\tau$ valued in $\T$, there exists an $L^1$ solution $(Y_\cdot, Z_\cdot, K_\cdot)$ of $\underline R$BSDE $(\xi,g+{\rm d}V,L)$ such that for each $\beta\in (0,1)$,
$$
\lim\limits_{n\To \infty}\left(\| Y_\cdot^n- Y_\cdot\|_{\s^\beta}+ \| Z_\cdot^n- Z_\cdot\|_{\M^\beta}\right)=0\vspace{0.1cm}
$$
and there exists a subsequence $\{ K_\cdot^{n_j}\}$ of $\{ K_\cdot^n\}$ such that
$$
\lim\limits_{j\To\infty}\sup\limits_{t\in\T}
| K_t^{n_j}- K_t|=0.
$$

\item [(ii)] Let $( Y^n_\cdot, Z^n_\cdot)$ be an $L^1$ solution of BSDE $(\xi,\underline g_n+{\rm d}V)$ with $\underline g_n(t,y,z):=g(t,y,z)-n(y-U_t)^+$ for each $n\geq 1$, i.e.,
\begin{equation}
\label{eq:3-PenalizationForBSDEwithAn}
\left\{
\begin{array}{l}
\Dis{Y}^n_t=\xi+\int_t^T\underline g_n(s,{Y}^n_s, {Z}^n_s)
{\rm d}s+\int_t^T{\rm d}V_s-\int_t^T{Z}^n_s \cdot {\rm d}B_s,\ \   t\in\T,\vspace{0.1cm}\\
\Dis  A^n_t:=n\int_0^t(
{Y}^n_s-U_s)^+\ {\rm d}s,\ \ t\in\T.
\end{array}
\right.
\end{equation}
If for each $n\geq 1$, $Y^n_\cdot\geq Y^{n+1}_\cdot\geq \underline Y_\cdot$ with a process $\underline Y_\cdot\in \cap_{\beta\in (0,1)}\s^\beta$ of the class (D), and there exist two constants $\bar\lambda>0$, $\alpha\in (0,1)$ and a nonnegative process $\bar f_\cdot\in \hcal^1$ such that \eqref{eq:3-SubLinearGrowthofgYnZn} holds true for each $n\geq 1$, then $\sup_{n\geq 1}\E[|A^n_T|^\beta]<+\infty$ for each $\beta\in (0,1)$, $\sup_{n\geq 1}\E[|A^n_\tau|^2]\leq \E[|\tilde Y_{\tau}|^2]$ for a process
$\tilde Y_\cdot\in\s$ and each $(\F_t)$-stopping time $\tau$ valued in $\T$, there exists an $L^1$ solution $( Y_\cdot, Z_\cdot,A_\cdot)$ of $\bar R$BSDE $(\xi,g+{\rm d}V,U)$ such that for each $\beta\in (0,1)$,
$$\lim\limits_{n\To \infty}\left(\| Y_\cdot^n- Y_\cdot\|_{\s^\beta}+
\| Z_\cdot^n- Z_\cdot\|_{\M^\beta}\right)=0$$
and there exists a subsequence $\{ A_\cdot^{n_j}\}$ of $\{ A_\cdot^n\}$ such that
$$\lim\limits_{j\To\infty}\sup\limits_{t\in\T}
| A_t^{n_j}- A_t|=0.$$
\end{itemize}
\end{pro}

\begin{proof}
We only prove (i), the proof of (ii) is similar. Note first that for each $n\geq 1$, $Y^n_\cdot\leq Y^{n+1}_\cdot\leq \bar Y_\cdot\in \cap_{\beta\in (0,1)}\s^\beta$. In view of \eqref{eq:3-SubLinearGrowthofgYnZn}, by \cref{lem:2-EstimateOfZKAg} we can deduce that for each $\beta\in (0,1)$, there exists a nonnegative constant $C_\beta$ depending only on $\beta,\bar\lambda,T$ such that
\begin{equation}\label{eq:3-BoundofZntAntandgn}
\begin{array}{ll}
&\Dis \sup\limits_{n\geq 1}\E\left[\left(\int_0^T|Z^n_s|^2{\rm d}s\right)^{\beta\over 2}+|K^n_T|^\beta+
\left(\int_0^T|g(s,Y^n_s,Z^n_s)|{\rm d}s\right)^\beta\right]\\
\leq &\Dis C_\beta\E\left[\sup\limits_{s\in [0,T]}(|Y^1_{s}|+|\bar Y_{s}|)^\beta+|V|^\beta_T+\left(\int_0^T \bar f_s\ {\rm d}s\right)^\beta\right]<+\infty,
\end{array}
\end{equation}
and there also exists a nonnegative constant $\bar C$ depending only on $\bar\lambda,T$ such that for each $(\F_t)$-stopping time $\tau$ valued in $\T$, we have
\begin{equation}\label{eq:3-BoundofZntInL2}
\sup\limits_{n\geq 1}\E\left[\int_0^\tau|Z^n_s|^2{\rm d}s+|K^n_\tau|^2\right]\leq \bar C\E\left[\sup\limits_{s\in [0,T]}(|Y^1_{s\wedge \tau}|+|\bar Y_{s\wedge \tau}|)^2+|V|^2_\tau+\left(\int_0^\tau \bar f_s\ {\rm d}s\right)^2\right].\vspace{0.2cm}
\end{equation}
For each positive integer $k\geq 1$, define the following $(\F_t)$-stopping time:
$$
\tau_k:=\Dis \inf\left\{t\geq 0:\  |Y^1_{t}|+|\bar Y_{t}|+|V|_t+\int_0^t \bar f_s\ {\rm d}s+L^+_t\geq k\right\}\wedge T.
$$
Then
$$
\mathbb{P}\left(\left\{\omega:\ \exists k_0(\omega)\geq 1, \ \RE k\geq k_0(\omega),\ \tau_k(\omega)=T\right\}\right)=1.\vspace{0.1cm}
$$
Thus, by letting $A^n_\cdot\equiv 0$ and $U_\cdot\equiv +\infty$, a same argument as in the proof of the steps 1-5 of \cref{pro:3-PenalizationOfRBSDE} yields that there exists a triple $(Y_\cdot, Z_\cdot, K_\cdot)\in \s^\beta\times\M^\beta\times\vcal^+$ for each $\beta\in (0,1)$ satisfying
$$
K_t=Y_0-Y_t-\int_0^t
g(s,Y_s,Z_s){\rm d}s-\int_0^t{\rm d}V_s+\int_0^tZ_s\cdot {\rm d}B_s.
$$
Furthermore, for each $\beta\in (0,1)$,
$$
\lim\limits_{n\To \infty}\left(\| Y_\cdot^n- Y_\cdot\|_{\s^\beta}+ \| Z_\cdot^n- Z_\cdot\|_{\M^\beta}\right)=0\vspace{0.1cm}
$$
and there exists a subsequence $\{ K_\cdot^{n_j}\}$ of $\{ K_\cdot^n\}$ such that
$$
\lim\limits_{j\To\infty}\sup\limits_{t\in\T}
| K_t^{n_j}- K_t|=0.
$$
In the sequel, a similar proof to the step 6 of \cref{pro:3-PenalizationOfRBSDE} yields that
$$
\E[K_T]\leq |Y_0|+\E[|\xi|]+\E[|V|_T]+
\|\bar f_\cdot\|_{\hcal^1}+\bar\lambda T^{2-\alpha\over 2}\|Z_\cdot\|_{\M^\alpha}<+\infty,
$$
which means that $K_\cdot\in \vcal^{+,1}$. Finally, similar to the step 7 of \cref{pro:3-PenalizationOfRBSDE}, it is easy to prove that $(Y_\cdot, Z_\cdot, K_\cdot)$ is an $L^1$ solution of $\underline R$BSDE $(\xi,g+{\rm d}V,L)$. The proof is complete.
\end{proof}

\subsection{Approximation\vspace{0.1cm}}

In this subsection, we prove the following general approximation result for $L^1$ solutions of DRBSDEs and both RBSDEs and non-reflected BSDEs as its special cases.\vspace{-0.1cm}

\begin{pro} [Approximation] \label{pro:3-Approximation}
Assume that $V_\cdot\in \vcal^1$, \ref{A:(H3)}(i) holds true for $L_\cdot$, $U_\cdot$ and $\xi$, $g_n$ is a generator and $(Y_\cdot^n,Z_\cdot^n,K_\cdot^n,A_\cdot^n)$ is an $L^1$ solution of DRBSDE $(\xi,g_n+{\rm d}V,L,U)$ for each $n\geq 1$. If for each $n\geq 1$, $Y_\cdot^n\leq Y_\cdot^{n+1}\leq \bar Y_\cdot$, ${\rm d}A^n\leq {\rm d}A^{n+1}\leq {\rm d}\bar A$ and ${\rm d}K^{n+1}\leq {\rm d}K^n\leq {\rm d}K^1$ with $\bar Y_\cdot\in \cap_{\beta\in (0,1)}\s^\beta$ of the class (D) and $\bar A\in \vcal^{+,1}$ (resp. $\underline {Y}_\cdot\leq Y_\cdot^{n+1}\leq Y_\cdot^n$, ${\rm d}A^{n+1}\leq {\rm d}A^n\leq {\rm d} A^1$ and ${\rm d}K^n\leq {\rm d}K^{n+1}\leq {\rm d}\bar K$ with $\underline {Y}_\cdot\in \cap_{\beta\in (0,1)}\s^\beta$ of the class (D) and $\bar K\in \vcal^{+,1}$), $g_n$ tends locally uniformly in $(y,z)$ to a generator $g$ as $n\To\infty$, there exists a constant $\bar\lambda>0$ and a nonnegative process $\tilde f_\cdot\in \hcal^1$ such that for each $n\geq 1$,
\begin{equation}
\label{eq:3-SemiLinearGrowthCondition}
\as,\ \ {\rm sgn}(Y^n_\cdot)g_n(\cdot,Y^n_\cdot,Z^n_\cdot)\leq \tilde f_\cdot+\bar\lambda|Z^n_\cdot|.
\end{equation}
and for each $k\geq 1$, there exists a nonnegative process $\bar f^k_\cdot\in \hcal$ such that for each $n\geq 1$,
\begin{equation}
\label{eq:3-LinearGrowthofgnYnZn--k}
\as,\ \ |g_n(\cdot,Y^n_\cdot,Z^n_\cdot)
\mathbbm{1}_{|Y^n_\cdot|\leq k}|\leq \bar f^k_\cdot+\bar\lambda |Z^n_\cdot|,
\end{equation}
then there exists an $L^1$ solution $(Y_\cdot,Z_\cdot,K_\cdot,A_\cdot)$ of DRBSDE $(\xi,g+{\rm d}V,L,U)$ such that for each $\beta\in (0,1)$,
$$
\lim\limits_{n\To \infty}\left(\|Y_\cdot^n-Y_\cdot\|_{\s^\beta}+
\|Z_\cdot^n-Z_\cdot\|_{\M^\beta}
+\|K_\cdot^n-K_\cdot\|_{\s^1}
+\|A_\cdot^n-A_\cdot\|_{\s^1}\right)=0.
$$
\end{pro}

\begin{proof}
We only prove the case
that for each $n\geq 1$, $Y_\cdot^n\leq Y_\cdot^{n+1}\leq \bar Y_\cdot$, ${\rm d}A^n\leq {\rm d}A^{n+1}\leq {\rm d}\bar A$ and ${\rm d}K^{n+1}\leq {\rm d}K^n\leq {\rm d} K^1$ with $\bar Y_\cdot\in \cap_{\beta\in (0,1)}\s^\beta$ of the class (D) and $\bar A\in \vcal^{+,1}$. Another case can be proved in the same way. Firstly, a same argument as that in proving \eqref{eq:3-UniformConvergenceOfAn} together with Lebesgue's dominated convergence theorem yields that there exists an $(\F_t)$-progressively measurable process $(Y_t)_{t\in\T}$ together with $K_\cdot, A_\cdot\in \vcal^{+,1}$ such that $Y_t^n\uparrow Y_t$ for each $t\in\T$, and
\begin{equation}
\label{eq:3-ConvergenceOfKnandAnInS1}
\lim\limits_{n\To \infty}\left(\|K_\cdot^n-K_\cdot\|_{\s^1}
+\|A_\cdot^n-A_\cdot\|_{\s^1}\right)=0.
\end{equation}
Furthermore, in view of \eqref{eq:3-SemiLinearGrowthCondition}, by \cref{lem:2-EstimateOfZandg} we can deduce that for each $\beta\in (0,1)$, there exists a nonnegative constant $C_\beta$ depending only on $\beta,\bar\lambda,T$ such that
\begin{equation}\label{eq:3-BoundofZntAntandgn}
\begin{array}{ll}
&\Dis \sup\limits_{n\geq 1}\E\left[\left(\int_0^T|Z^n_s|^2{\rm d}s\right)^{\beta\over 2}+\left(\int_0^T|g_n(s,Y^n_s,Z^n_s)|{\rm d}s\right)^\beta\right]\\
\leq &\Dis C_\beta\E\left[\sup\limits_{s\in [0,T]}(|Y^1_{s}|+|\bar Y_{s}|)^\beta+|V|^\beta_T+|K^1_T|^\beta+|\bar A_T|^\beta+\left(\int_0^T \tilde f_s\ {\rm d}s\right)^\beta\right]<+\infty,
\end{array}
\end{equation}
and there also exists a constant $\bar C>0$ depending only on $\bar\lambda,T$ such that for each $(\F_t)$-stopping time $\tau$ valued in $\T$, we have
\begin{equation}
\label{eq:3-BoundofZntInL2}
\sup\limits_{n\geq 1}\E\left[\int_0^\tau|Z^n_s|^2{\rm d}s\right]\leq \bar C\E\left[\sup\limits_{s\in [0,T]}(|Y^1_{s\wedge \tau}|+|\bar Y_{s\wedge \tau}|)^2+|V|^2_\tau+|K^1_\tau|^2+|\bar A_\tau|^2+\left(\int_0^\tau \tilde f_s\ {\rm d}s\right)^2\right].\vspace{0.2cm}
\end{equation}

The rest proof of this proposition is divided into 3 steps, some of ideas among them are lent from the proof of Proposition 5.1 in \citet{Fan2017arXivLpSolutionofDRBSDEs}.
\vspace{0.1cm}

{\bf Step 1.}\ We show the convergence of the sequence $\{Y_\cdot^n\}$. For each positive integer $k,l\geq 1$, we introduce the following two $(\F_t)$-stopping times:
$$
\begin{array}{lll}
\sigma_k &:=& \Dis \inf\left\{t\geq 0:\  |Y^1_{t}|+|\bar Y_{t}|+|V|_t+K^1_t+\bar A_t+\int_0^t \tilde f_s\ {\rm d}s\geq k\right\}\wedge T;\vspace{0.1cm}\\
\tau_{k,l}&:=&\Dis \inf\left\{t\geq 0:\  \int_0^t \bar f^k_s\ {\rm d}s\geq l\right\}\wedge \sigma_k.
\end{array}
$$
Then we have\vspace{-0.2cm}
\begin{equation}
\label{eq:3-StabilityOfTaukl}
\mathbb{P}\left(\left\{\omega:\ \exists k_0(\omega),l_0(\omega)\geq 1, \ \RE k\geq k_0(\omega),\ \RE l\geq l_0(\omega),\ \ \tau_{k,l}(\omega)=T\right\}\right)=1.
\vspace{-0.1cm}
\end{equation}
For each $n,m\geq 1$, observe that
\begin{equation}
\label{eq:3-DefinitionOfBarYZV}
\begin{array}{lll}
\Dis (\bar Y_\cdot,\bar Z_\cdot,\bar V_\cdot)&
:=&\Dis (Y_\cdot^n-Y_\cdot^m,Z_\cdot^n-Z_\cdot^m,\\
&&\Dis \ \ \int_0^\cdot \left(g_n(s,Y_s^n,Z_s^n)-g_m(s,Y_s^m,Z_s^m)\right){\rm d}s+\left(K_\cdot^n-K_\cdot^m\right)
+\left(A_\cdot^m-A_\cdot^n\right) )
\end{array}
\end{equation}
satisfies equation \eqref{eq:2-BarY=BarV}. It then follows from (ii) of \cref{lem:2-Lemma1} with $p=2$, $t=0$ and $\tau=\tau_{k,l}$ that there exists a constant $C>0$ such that for each $n,m,k,l\geq 1$,
\begin{equation}
\label{eq:3-FirstBoundOfYnMinusYm}
\begin{array}{lll}
\Dis\E\left[\sup\limits_{t\in [0,T]} |Y_{t\wedge\tau_{k,l}}^n
-Y_{t\wedge\tau_{k,l}}^m|^2\right]&\leq &\Dis C\E\left[|Y_{\tau_{k,l}}^n-Y_{\tau_{k,l}}^m|^2 +\sup\limits_{t\in [0,T]}\left(\int_{t\wedge \tau_{k,l}}^{\tau_{k,l}}(Y^n_s-Y_s^m)
\left({\rm d}K_s^n-{\rm d}K_s^m\right) \right)^+\right.\vspace{0.1cm}\\
&&\hspace{1cm}\Dis +\sup\limits_{t\in [0,T]}\left(\int_{t\wedge \tau_{k,l}}^{\tau_{k,l}}(Y^n_s-Y_s^m)
\left({\rm d}A_s^m-{\rm d}A_s^n\right) \right)^+\vspace{0.2cm}\\
&&\hspace{1cm}\Dis +\left. \int_{0}^{\tau_{k,l}}|Y^n_s-Y_s^m|
\left|g_n(s,Y_s^n,Z_s^n)-g_m(s,Y_s^m,Z_s^m)\right| {\rm d}s\right].
\end{array}
\end{equation}
Furthermore, note that $L_\cdot\leq Y_\cdot^n\leq U_\cdot$ and that $\int_0^T(Y_t^n-L_t){\rm d}K_t^n=\int_0^T( U_t-Y_t^n){\rm d}A_t^n=0$ for each $n\geq 1$. It follows that for each $t\in \T$ and $k,l,m,n\geq 1$,
\begin{equation}
\label{eq:3-YnYmKnKmLeq0}
\begin{array}{lll}
\Dis \int_{t\wedge \tau_{k,l}}^{\tau_{k,l}}(Y^n_s-Y_s^m)
\left({\rm d}K_s^n-{\rm d}K_s^m\right)&
= & \Dis \int_{t\wedge \tau_{k,l}}^{\tau_{k,l}}\left[(Y_s^n-L_s)
-(Y_s^m-L_s)\right]\left({\rm d}K_s^n-{\rm d}K_s^m\right)\vspace{0.1cm}\\
&= & \Dis -\int_{t\wedge \tau_{k,l}}^{\tau_{k,l}}(Y_s^n-L_s)
{\rm d}K_s^m-\int_{t\wedge \tau_{k,l}}^{\tau_{k,l}}(Y_s^m-L_s)
{\rm d}K_s^n\\
&\leq & 0
\end{array}
\end{equation}
and
\begin{equation}
\label{eq:3-YnYmAnAmLeq0}
\begin{array}{lll}
\Dis \int_{t\wedge \tau_{k,l}}^{\tau_{k,l}}(Y^n_s-Y_s^m)
\left({\rm d}A_s^m-{\rm d}A_s^n\right)&
= & \Dis \int_{t\wedge \tau_{k,l}}^{\tau_{k,l}}\left[(U_s-Y_s^m)
-(U_s-Y_s^n)\right]\left({\rm d}A_s^m-{\rm d}A_s^n\right)\vspace{0.1cm}\\
&= & \Dis -\int_{t\wedge \tau_{k,l}}^{\tau_{k,l}}(U_s-Y_s^m)
{\rm d}A_s^n-\int_{t\wedge \tau_{k,l}}^{\tau_{k,l}}(U_s-Y_s^n)
{\rm d}A_s^m\\
&\leq & 0.
\end{array}
\end{equation}
By the definition of $\tau_{k,l}$ and the fact that $Y_\cdot^1\leq Y_\cdot^n\leq \bar Y_\cdot$, we know that $\mathbbm{1}_{\cdot\leq\tau_{k,l}}\leq \mathbbm{1}_{|Y^n_\cdot|\leq k}$ holds true for each $k,l,n\geq 1$. Then, combining \eqref{eq:3-LinearGrowthofgnYnZn--k}, \eqref{eq:3-FirstBoundOfYnMinusYm},  \eqref{eq:3-YnYmKnKmLeq0}, \eqref{eq:3-YnYmAnAmLeq0} and H\"{o}lder's inequality yields that
\begin{equation}
\label{eq:3-SecondBoundOfYnMinusYm}
\begin{array}{lll}
\Dis\E\left[\sup\limits_{t\in [0,T]} |Y_{t\wedge\tau_{k,l}}^n
-Y_{t\wedge\tau_{k,l}}^m|^2\right]&\leq &\Dis C\E\left[|Y_{\tau_{k,l}}^n-Y_{\tau_{k,l}}^m|^2 +2\int_0^{\tau_{k,l}} |Y^n_t-Y_t^m|\bar f^k_t{\rm d}t\right]\vspace{0.1cm}\\
&&\Dis+2C\bar\lambda\left(\E\left[\int_{0}^{
\tau_{k,l}}|Y^n_t-Y_t^m|^2{\rm d}t\right]\right)^{1\over 2} \left(\E\left[\int_{0}^{\tau_{k,l}}\left(|Z_t^n|+
|Z_t^m|\right)^2{\rm d}t\right]\right)^{1\over 2}.
\end{array}
\end{equation}
Note that $Y_\cdot^1\leq Y_\cdot^n\uparrow Y_\cdot\leq \bar Y_\cdot$. By the definition of $\tau_{k,l}$ and \eqref{eq:3-BoundofZntInL2}, it follows from \eqref{eq:3-SecondBoundOfYnMinusYm} and Lebesgue's dominated convergence theorem that for each $k,l\geq 1$, as $n,m\To\infty$,
$$\E\left[\sup\limits_{t\in [0,T]} |Y_{t\wedge\tau_{k,l}}^n
-Y_{t\wedge\tau_{k,l}}^m|^2\right]\To 0,$$
which implies that for each $k,l\geq 1$, as $n,m\To\infty$,
$$\sup\limits_{t\in [0,T]} |Y_{t\wedge\tau_{k,l}}^n
-Y_{t\wedge\tau_{k,l}}^m|\To 0\ {\rm in\ probability}\ \mathbb{P}.$$
And, by \eqref{eq:3-StabilityOfTaukl} and the monotonicity of $Y_\cdot^n$ with respect to $n$ we know that
\begin{equation}
\label{eq:3.2-UniformConvergenceOfYnPas}
\sup\limits_{t\in [0,T]} |Y_t^n
-Y_t|\To 0,\ \ {\rm as}\ n\To\infty.
\end{equation}
So, $Y_\cdot$ is a continuous process, and then belongs to the space $\s^\beta$ for each $\beta\in (0,1)$ and the class (D) due to the fact that both $Y_\cdot^1$ and $\bar Y_\cdot$ belong to them as well as $Y_\cdot^1\leq Y^n_\cdot\leq \bar Y_\cdot$. Finally, from \eqref{eq:3.2-UniformConvergenceOfYnPas} and Lebesgue's dominated convergence theorem it follows that for each $\beta\in (0,1)$,
\begin{equation}
\label{eq:3.2-ConvergenceOfYnInSp}
\lim\limits_{n\To\infty}\|Y_\cdot^n-Y_\cdot\|_{\s^\beta}^\beta =\lim\limits_{n\To\infty}\E\left[\sup\limits_{t\in [0,T]} |Y_t^n-Y_t|^\beta\right]=0.\vspace{0.2cm}
\end{equation}

{\bf Step 2.}\ We show the convergence of the sequence $\{Z_\cdot^n\}$. Note that \eqref{eq:3-DefinitionOfBarYZV} solves \eqref{eq:2-BarY=BarV}. It follows from (i) of \cref{lem:2-Lemma1} with $t=0$ and $\tau=T$ that there exists a nonnegative constant $C'\geq 0$ such that for each $m,n\geq 1$ and $\beta\in (0,1)$, we have
$$
\begin{array}{lll}
&\Dis\E\left[\left(\int_0^T|Z_t^n-Z_t^m|^2{\rm d}t\right)^{\beta\over 2}\right]\\
\leq & \Dis C'\E\left[\sup\limits_{t\in [0,T]}|Y_t^n-Y_t^m|^\beta+\sup\limits_{t\in [0,T]}\left[\left(\int_t^T (Y_s^n-Y_s^m)\left({\rm d}K_s^n-{\rm d}K_s^m\right)\right)^+\right]^{\beta\over 2}\right]\\
&\Dis +C'\E\left[\sup\limits_{t\in [0,T]}\left[\left(\int_t^T (Y_s^n-Y_s^m)\left({\rm d}A_s^m-{\rm d}A_s^n\right)\right)^+\right]^{\beta\over 2}\right]\vspace{0.1cm}\\
&\Dis + C'\E\left[\left(\int_{0}^T|Y^n_s-Y_s^m|
\left|g_n(s,Y_s^n,Z_s^n)-g_m(s,Y_s^m,Z_s^m)\right| {\rm d}s\right)^{\beta\over 2}\right].
\end{array}
$$
Then, in view of \eqref{eq:3-YnYmKnKmLeq0} and \eqref{eq:3-YnYmAnAmLeq0}, it follows from H\"{o}lder's inequality that
$$
\begin{array}{lll}
\Dis\E\left[\left(\int_0^T|Z_t^n-Z_t^m|^2{\rm d}t\right)^{\beta\over 2}\right]
&\leq &\Dis C'\E\left[\sup\limits_{t\in [0,T]}|Y_t^n-Y_t^m|^\beta\right]+C'\left(\E\left[\sup\limits_{t\in [0,T]}|Y_t^n-Y_t^m|^\beta\right]\right)^{1\over 2}\\
&&\Dis\ \ \ \ \  \cdot\left(\E\left[\left(\int_{0}^T
\left(|g_n(t,Y_t^n,Z_t^n)|+|g_m(t,Y_t^m,Z_t^m)|\right) {\rm d}t\right)^{\beta}\right]\right)^{1\over 2},
\end{array}
$$
from which together with \eqref{eq:3.2-ConvergenceOfYnInSp} and \eqref{eq:3-BoundofZntAntandgn} yields that there exists a process $(Z_t)_{t\in\T}\in \cap_{\beta\in (0,1)}\M^\beta$ satisfying, for each $\beta\in (0,1)$,\vspace{-0.1cm}
\begin{equation}
\label{eq:3.2-ConvergenceOfZnInMp}
\lim\limits_{n\To\infty}\|Z_\cdot^n-Z_\cdot\|_{\M^\beta}^\beta
=\lim\limits_{n\To\infty}\E\left[\left(\int_0^T|Z_t^n
-Z_t|^2{\rm d}t\right)^{\beta\over 2}\right]=0.\vspace{0.2cm}
\end{equation}

{\bf Step 3.}\ We show that $(Y_\cdot,Z_\cdot,K_\cdot,A_\cdot)$ is an $L^1$ solution of DRBSDE $(\xi,g+{\rm d}V,L,U)$. Since $g_n$ tends locally uniformly in $(y,z)$ to the generator $g$ as $n\To\infty$ and satisfies \eqref{eq:3-LinearGrowthofgnYnZn--k}, by \eqref{eq:3.2-UniformConvergenceOfYnPas} and \eqref{eq:3.2-ConvergenceOfZnInMp} together with \eqref{eq:3-StabilityOfTaukl} we can deduce that there exists a subsequence $\{n_j\}$ of $\{n\}$ such that
$$\lim\limits_{j\To \infty}\int_0^T|g_{n_j}(t,Y_t^{n_j},Z_t^{n_j})-
g(t,Y_t,Z_t)|{\rm d}t=0.$$
Then,
\begin{equation}
\label{eq:3-ConvergenceOfgn}
\lim\limits_{j\To \infty}\sup\limits_{t\in\T}\left|\int_0^t
g_{n_j}(t,Y_t^{n_j},Z_t^{n_j}){\rm d}t-\int_0^t
g(t,Y_t,Z_t){\rm d}t\right|=0.\vspace{0.2cm}
\end{equation}
Combining \eqref{eq:3-ConvergenceOfKnandAnInS1}, \eqref{eq:3.2-UniformConvergenceOfYnPas}, \eqref{eq:3.2-ConvergenceOfZnInMp} and \eqref{eq:3-ConvergenceOfgn}
yields that
$$
Y_t=\xi+\int_t^Tg(s,Y_s,Z_s){\rm d}s+\int_t^T{\rm d}V_s+\int_t^T{\rm d}K_s-\int_t^T{\rm d}A_s-\int_t^TZ_s \cdot {\rm d}B_s,\ \ t\in\T.
$$
Since $L_t\leq Y_t^n\leq U_t$ and $Y_t^n\uparrow Y_t$ for each $t\in\T$, we have $L_t \leq Y_t\leq U_t$ for each $t\in\T$. Furthermore, in view of \eqref{eq:3.2-ConvergenceOfYnInSp} and \eqref{eq:3-ConvergenceOfKnandAnInS1}, it follows that
$$
\int_0^T(Y_t-L_t){\rm d}K_t=\Lim\int_0^T(Y_t^{n}-L_t){\rm d}K_t^n=0\ \ {\rm and}\ \  \int_0^T(U_t-Y_t){\rm d}K_t=\Lim\int_0^T(U_t-Y_t^{n}){\rm d}A_t^n=0.
$$
Finally, let us show that ${\rm d}K\bot {\rm d}A$. In fact, for each $n\geq 1$, since ${\rm d}K^n\bot {\rm d}A^n$, there exists an $(\F_t)$-progressively measurable set
$D_n\subset \Omega\times\T$ such that
$$
\E\left[\int_0^T \mathbbm{1}_{D_n} {\rm d}K^n_t\right]=\E\left[\int_0^T \mathbbm{1}_{D_n^c} {\rm d}A^n_t\right]=0.
$$
Then, in view of \eqref{eq:3-ConvergenceOfKnandAnInS1} and the fact that ${\rm d}K\leq {\rm d}K^n$ for each $n\geq 1$,
$$
0\leq \E\left[\int_0^T\mathbbm{1}_{\cup D_n} {\rm d}K_t\right]\leq \sum\limits_{n=1}^\infty \E\left[\int_0^T \mathbbm{1}_{D_n} {\rm d}K_t\right]\leq \sum\limits_{n=1}^\infty \E\left[\int_0^T\mathbbm{1}_{D_n} {\rm d}K^n_t\right]=0
$$
and
$$
0\leq \E\left[\int_0^T\mathbbm{1}_{\cap D_n^c} {\rm d}A_t\right]=\lim\limits_{m\To\infty}\E\left[\int_0^T
\mathbbm{1}_{\cap D_n^c}{\rm d}A^m_t\right]\leq \lim\limits_{m\To\infty}\E\left[\int_0^T
\mathbbm{1}_{D_m^c}{\rm d}A^m_t\right]=0.\vspace{0.1cm}
$$
Hence, ${\rm d}K\bot{\rm d}A$. \cref{pro:3-Approximation} is then proved.
\end{proof}

\begin{rmk}
\label{rmk:3-rmk3.4}
Observe that if there exists a constant $\bar\lambda>0$ and a nonnegative process $\bar f_\cdot\in \hcal^1$ such that for each $n\geq 1$,\vspace{-0.1cm}
\begin{equation}
\label{eq:3-LinearGrowthofgnYnZn}
\as,\ \ |g_n(\cdot,Y^n_\cdot,Z^n_\cdot)|\leq \bar f_\cdot+\bar\lambda |Z^n_\cdot|,
\end{equation}
then both \eqref{eq:3-SemiLinearGrowthCondition} and \eqref{eq:3-LinearGrowthofgnYnZn--k} are satisfied.
\end{rmk}

\subsection{Comparison theorem\vspace{0.1cm}}

We now establish a general comparison theorem for $L^1$ solutions of RBSDEs with one and two continuous barriers as well as non-reflected BSDEs.\vspace{-0.1cm}

\begin{pro} [Comparison Theorem]
\label{pro:3-ComparisonTheoremofDRBSDE}
Assume that $V_\cdot^j\in \vcal^1$, \ref{A:(H3)}(i) holds for $L_\cdot^j,U_\cdot^j$ and $\xi^j$, $g^j$ is a generator and  $(Y_\cdot^j,Z_\cdot^j,K_\cdot^j,A_\cdot^j)$ is an $L^1$ solution of DRBSDE $(\xi^j,g^j+{\rm d}V^j,L^j,U^j)$ for $j=1,2$. If $\xi^1\leq \xi^2$, ${\rm d}V^1\leq {\rm d}V^2$, $L^1_\cdot\leq L^2_\cdot$, $U^1_\cdot\leq U^2_\cdot$, and either
$$
\left\{
\begin{array}{l}
g^1\ satisfies\ \ref{A:(H1)}(i)\ and \ \ref{A:(H2)};\\
\as,\ \ \mathbbm{1}_{\{Y^1_t>Y^2_t\}}
\left(g^1(t,Y_t^2,Z_t^2)
-g^2(t,Y_t^2,Z_t^2)\right)\leq 0
\end{array}
\right.
$$
or
$$
\left\{
\begin{array}{l}
g^2\ satisfies\ \ref{A:(H1)}(i)\ and \ \ref{A:(H2)};\\
\as,\ \ \mathbbm{1}_{\{Y^1_t>Y^2_t\}}
\left(g^1(t,Y_t^1,Z_t^1)
-g^2(t,Y_t^1,Z_t^1)\right)\leq 0
\end{array}
\right.\vspace{0.2cm}
$$
is satisfied, then $Y_t^1\leq Y_t^2$ for each $t\in \T$.
\end{pro}

\begin{proof}
For each positive integer $k\geq 1$, define the following $(\F_t)$-stopping time:
$$
\tau_k:=\inf\{t\in\T:\ \ \int_0^t (|Z_s^1|^2+|Z_s^2|^2){\rm d}s\geq k\}\wedge T.
$$
It follows from It\^{o}-Tanaka's formula that for each $t\in\T$ and $k\geq 1$,
$$
\begin{array}{lll}
\Dis  (Y^1_{t\wedge\tau_k}-Y^2_{t\wedge\tau_k})^+ &\leq &\Dis (Y^1_{\tau_k}-Y^2_{\tau_k})^+ +\int_{t\wedge\tau_k}^{\tau_k} {\rm sgn}((Y^1_s-Y^2_s)^+)({\rm d}V_s^1-{\rm d}V_s^2)\vspace{0.1cm}\\
&&\Dis +\int_{t\wedge\tau_k}^{\tau_k}{\rm sgn}((Y^1_s-Y^2_s)^+)\left(g^1(s,Y_s^1,Z_s^1)
-g^2(s,Y_s^2,Z_s^2)\right){\rm d}s\vspace{0.1cm}\\
&&\Dis+\int_{t\wedge\tau_k}^{\tau_k} {\rm sgn}((Y^1_s-Y^2_s)^+)\left({\rm d}K_s^1-{\rm d}K_s^2\right)+\int_{t\wedge\tau_k}^{\tau_k} {\rm sgn}((Y^1_s-Y^2_s)^+)\left({\rm d}A_s^2-{\rm d}A_s^1\right)\vspace{0.1cm}\\
&&\Dis +\int_{t\wedge\tau_k}^{\tau_k} {\rm sgn}((Y^1_s-Y^2_s)^+)(Z_s^1-Z_s^2)\cdot {\rm d}B_s.
\end{array}
$$
Since $L_t^1\leq L_t^2\leq Y_t^2$, $L_t^1\leq Y_t^1$, $t\in \T$ and $\int_0^T(Y^1_s-L_s^1){\rm d}K_s^1=0$, we have
$$
\begin{array}{lll}
\Dis \int_{t\wedge\tau_k}^{\tau_k} {\rm sgn}((Y^1_s-Y^2_s)^+)\left({\rm d}K_s^1-{\rm d}K_s^2\right)
&\leq & \Dis \int_{t\wedge\tau_k}^{\tau_k} {\rm sgn}((Y^1_s-Y^2_s)^+){\rm d}K_s^1\leq \int_{t\wedge\tau_k}^{\tau_k} {\rm sgn}((Y^1_s-L_s^1)^+){\rm d}K_s^1\vspace{0.1cm}\\
&=&\Dis\int_{t\wedge\tau_k}^{\tau_k} \mathbbm{1}_{\{Y^1_s>L_s^1\}}|Y^1_s-L_s^1|^{-1}(Y^1_s-L_s^1){\rm d}K_s^1=0.
\end{array}
$$
Similarly, since $Y^1_t\leq U_t^1\leq U_t^2$, $Y^2_t\leq U_t^2$, $t\in \T$ and $\int_0^T(U^2_s-Y_s^2){\rm d}A_s^2=0$, we have
$$
\begin{array}{lll}
\Dis \int_{t\wedge\tau_k}^{\tau_k} {\rm sgn}((Y^1_s-Y^2_s)^+)\left({\rm d}A_s^2-{\rm d}A_s^1\right)
&\leq & \Dis \int_{t\wedge\tau_k}^{\tau_k} {\rm sgn}((Y^1_s-Y^2_s)^+){\rm d}A_s^2\leq \int_{t\wedge\tau_k}^{\tau_k} {\rm sgn}((U^2_s-Y_s^2)^+){\rm d}A_s^2\vspace{0.1cm}\\
&=&\Dis\int_{t\wedge\tau_k}^{\tau_k} \mathbbm{1}_{\{U^2_s>Y_s^2\}}|U^2_s-Y_s^2|^{-1}
(U^2_s-Y_s^2){\rm d}A_s^2=0.
\end{array}
$$
Thus, noticing that ${\rm d}V^1\leq {\rm d}V^2$, by virtue of the previous three inequalities we get that
$$
\begin{array}{lll}
\Dis  (Y^1_{t\wedge\tau_k}-Y^2_{t\wedge\tau_k})^+ &\leq &\Dis (Y^1_{\tau_k}-Y^2_{\tau_k})^++\int_{t\wedge\tau_k}^{\tau_k}{\rm sgn}((Y^1_s-Y^2_s)^+)\left(g^1(s,Y_s^1,Z_s^1)
-g^2(s,Y_s^2,Z_s^2)\right){\rm d}s\\
&&\Dis +\int_{t\wedge\tau_k}^{\tau_k} {\rm sgn}((Y^1_s-Y^2_s)^+)(Z_s^1-Z_s^2)\cdot {\rm d}B_s,\ \ t\in \T.
\end{array}
$$
Finally, in view of the assumptions of $g^1$ and $g^2$ together with $\xi^1\leq \xi^2$, the rest proof runs as the proof of Theorem 2.4 and Theorem 2.1 in \citet{Fan2016SPA} with $u(t)=v(t)\equiv 1$ and $\lambda(t)\equiv \gamma$, which is omitted.
\end{proof}

\begin{rmk}
\label{rmk:3-RemarkOfComparisionTheorem}
Observe that in the proof of \cref{pro:3-ComparisonTheoremofDRBSDE} the following two assumptions are not utilized:
$$
\int_0^T(Y^2_s-L_s^2){\rm d}K_s^2=0\ \ {\rm and}\ \
\int_0^T(U^1_s-Y_s^1){\rm d}A_s^1=0.\vspace{0.2cm}
$$
\end{rmk}

By virtue of \cref{pro:3-ComparisonTheoremofDRBSDE}, the following corollary follows immediately.

\begin{cor}
\label{cor:3-CorollaryOfComparisonTheorem}
Assume that $V_\cdot^j\in \vcal^1$, \ref{A:(H3)}(i) holds for $L_\cdot^j,U_\cdot^j$ and $\xi^j$, $g^j$ is a generator and $(Y_\cdot^j,Z_\cdot^j,K_\cdot^j,A_\cdot^j)$ is an $L^1$ solution of DRBSDE $(\xi^j,g^j+{\rm d}V^j,L^j,U^j)$ for $j=1,2$. If $\xi^1\leq \xi^2$, ${\rm d}V^1\leq {\rm d}V^2$, $L^1_\cdot\leq L^2_\cdot$, $U^1_\cdot\leq U^2_\cdot$, $g^1$ or $g^2$ satisfies \ref{A:(H1)}(i) and \ref{A:(H2)}, and for each $(y,z)\in \R\times \R^d$,
$$\as,\ \ g^1(t,y,z)\leq g^2(t,y,z),$$
then $Y_t^1\leq Y_t^2$ for each $t\in \T$.
\end{cor}

\begin{thm}[Uniqueness]
\label{thm:3-UniquenessOfSolutions}
Let $V_\cdot\in\vcal^1$, \ref{A:(H3)}(i) hold true for $L_\cdot,U_\cdot$ and $\xi$, and the generator $g$ satisfy assumptions \ref{A:(H1)}(i) and \ref{A:(H2)}. Then DRBSDE $(\xi,g+{\rm d}V,L,U)$ admits at most one $L^1$ solution, i.e, if both $(Y_\cdot,Z_\cdot,K_\cdot,A_\cdot)$ and $(Y'_\cdot,Z'_\cdot,K'_\cdot,A'_\cdot)$ are $L^1$ solutions of DRBSDE $(\xi,g+{\rm d}V,L,U)$, then $\as$,
$$
Y_\cdot=Y'_\cdot,\ Z_\cdot=Z'_\cdot\ K_\cdot=K'_\cdot\ \  {\rm and}\ \  A_\cdot=A'_\cdot.
$$
\end{thm}

\begin{proof}
The conclusion follows from \cref{cor:3-CorollaryOfComparisonTheorem}, It\^{o}'s formula and the Ham-Bananch composition of sign measure.
\end{proof}

\section{Existence, uniqueness and approximation for $L^1$ solutions of BSDEs}
\label{sec:4-ExistenceBSDEs}
\setcounter{equation}{0}

In this section, we will establish some existence, uniqueness and approximation results on $L^1$ solutions of BSDEs under general assumptions.\vspace{0.1cm}

We need the following lemma, which is a direct corollary of Theorem 6.5 in \citet{Fan2016SPA}.\vspace{-0.1cm}

\begin{lem}
\label{lem:4-lemma4.1}
Let $\xi\in \LT$ and the generator $g$ satisfy assumptions \ref{A:(H1)}(i) and \ref{A:(H2)}(i). If $g$ is also bounded, then BSDE $(\xi,g)$ admits a unique $L^1$ solution.
\end{lem}

Let us start with the following existence and uniqueness result.\vspace{-0.1cm}

\begin{thm}
\label{thm:4-ExistenceanduniquenssofBSDEunderH2}
Let $\xi\in\LT$, $V_\cdot\in\vcal^1$ and the generator $g$ satisfy assumptions \ref{A:(H1)} and \ref{A:(H2)}. Then BSDE $(\xi,g+{\rm d}V)$ admits a unique $L^1$ solution.
\end{thm}

\begin{proof}
The uniqueness part follows immediately from \cref{thm:3-UniquenessOfSolutions} with $L_\cdot\equiv -\infty$ and $U_\cdot=+\infty$. In the sequel, we prove the existence part. Let $\xi\in\LT$, $V_\cdot\in\vcal^1$ and $g$ satisfy \ref{A:(H1)} and \ref{A:(H2)}.

We first assume that $g$ is bounded. Note that $V_\cdot\in\vcal^1$. It follows from \cref{lem:4-lemma4.1} that the following BSDE
$$
\bar Y_s=\xi+V_T+\int_t^T g(s,\bar Y_s-V_s,\bar Z_s){\rm d}s-\int_t^T\bar Z_s{\rm d}B_s,\ \ t\in\T\vspace{0.1cm}
$$
admits a unique $L^1$ solution $(\bar Y_\cdot,\bar Z_\cdot)$. Then the pair $(Y_\cdot,Z_\cdot):=(\bar Y_\cdot-V_\cdot,\bar Z_\cdot)$ is just the unique $L^1$ solution of BSDE $(\xi,g+{\rm d}V)$.

Now suppose that $g$ is bounded from below. Write $g_n=g\wedge n$. Then $g_n$ is bounded, nondecreasing in $n$ and tends locally uniformly to $g$ as $n\To\infty$, and it is not difficult to check that all $g_n$ satisfy \ref{A:(H1)} and \ref{A:(H2)} with the same $\rho(\cdot)$, $\psi_\cdot(r)$, $\phi(\cdot)$, $\gamma$, $f_\cdot$ and $\alpha$. Then by the first step of the proof there exists a unique $L^1$ solution $(Y_\cdot^n,Z_\cdot^n)$ of BSDE $(\xi,g_n+{\rm d}V)$. Furthermore, in view of \cref{rmk:2-LinearGrowthOfRhoandPhi},
it follows from \ref{A:(H1)}(i) and \ref{A:(H2)}(ii) of $g_n$ that $\as$, for each $n\geq 1$ and $(y,z)\in\R\times\R^d$,
\begin{equation}
\label{eq:4-gnSatisfyA}
\begin{array}{lll}
\Dis {\rm sgn}(y)g_n(\cdot,y,z)&\leq &\Dis {\rm sgn}(y)(g_n(\cdot,y,z)-g_n(\cdot,0,z))+|g_n(\cdot,0,z)-g_n(\cdot,0,0)|
+|g_n(\cdot,0,0)|\\
&\leq & \Dis \rho(|y|)+\gamma(f_\cdot+|z|)^{\alpha} +|g(\cdot,0,0)|\\
&\leq & \Dis A+ \gamma(1+f_\cdot)+|g(\cdot,0,0)|+A|y|+\gamma(1+|z|)^\alpha
=:\bar g(\cdot,y,z).\vspace{0.1cm}
\end{array}
\end{equation}
Note that $\xi\in\LT$, $V_\cdot\in \vcal^1$, $f_\cdot\in \hcal^1$, $g(\cdot,0,0)\in \hcal^1$, and the generator $\bar g$ is uniformly Lipschitz in $(y,z)$ and has a sub-linear growth in $z$. By Theorems 3.9 and 3.11 in \citet{Klimsiak2013BSM} we know that BSDE $(|\xi|,\bar g+{\rm d}|V|)$ admits a unique $L^1$ solution $(\bar Y_\cdot,\bar Z_\cdot)$ with $\bar Y_\cdot\geq 0$, and BSDE $(-|\xi|,-\bar g-{\rm d}|V|)$ admits a unique $L^1$ solution $(\underline Y_\cdot,\underline Z_\cdot)$ with $\underline Y_\cdot\leq 0$. Furthermore, note by \eqref{eq:4-gnSatisfyA} that $\as$,
$$
\mathbbm{1}_{\{Y^n_t>\bar Y_t\}}\left(g_n(t,Y^n_t,Z^n_t)-\bar g(t,Y^n_t,Z^n_t)\right)\leq 0\ \ {\rm and}\ \ \mathbbm{1}_{\{\underline Y_t> Y^n_t\}}\left(-\bar g(t,Y^n_t,Z^n_t)- g_n(t,Y^n_t,Z^n_t)\right)\leq 0.
$$
It follows from \cref{pro:3-ComparisonTheoremofDRBSDE} and \cref{cor:3-CorollaryOfComparisonTheorem} with $L_\cdot=-\infty$ and $U_\cdot=+\infty$ that $\underline Y_\cdot\leq Y^n_\cdot\leq Y^{n+1}_\cdot\leq \bar Y_\cdot$ for each $n\geq 1$. Thus, by \eqref{eq:4-gnSatisfyA} we know that \eqref{eq:3-SemiLinearGrowthCondition} holds true. In addition, in view of assumptions \ref{A:(H2)}(ii) and \ref{A:(H1)}(iii), we have for each $n,k\geq 1$,
$$
\begin{array}{lll}
&&\Dis|g_n(\cdot,Y^n_\cdot,Z^n_\cdot)
\mathbbm{1}_{|Y^n_\cdot|\leq k}|\leq  \Dis |g(\cdot,Y^n_\cdot,Z^n_\cdot)|
\mathbbm{1}_{|Y^n_\cdot|\leq k}\\
&\leq &\Dis |g(\cdot,Y^n_\cdot,Z^n_\cdot)
-g(\cdot,Y^n_\cdot,0)|\mathbbm{1}_{|Y^n_\cdot|\leq k}+
|g(\cdot,Y^n_\cdot,0)-g(\cdot,0,0)|
\mathbbm{1}_{|Y^n_\cdot|\leq k}+|g(\cdot,0,0)|\\
&\leq &\Dis \gamma \left(1+f_\cdot+
|Y^n_\cdot|\mathbbm{1}_{|Y^n_\cdot|\leq k}+|Z^n_\cdot|\right)+\psi_\cdot(k)+|g(\cdot,0,0)|\\
&\leq &\Dis |g(\cdot,0,0)|+\gamma \left(1+f_\cdot+k\right)+\psi_\cdot(k)+\gamma |Z^n_\cdot|.
\end{array}
$$
Hence, \eqref{eq:3-LinearGrowthofgnYnZn--k} holds also true since $\psi_\cdot(k)\in \hcal$ and $f_\cdot, g(\cdot,0,0)\in \hcal^1$. Now, we have checked all the conditions in \cref{pro:3-Approximation} with $L_\cdot=-\infty$, $U_\cdot=+\infty$ and $K^n_\cdot=A^n_\cdot\equiv 0$, and it follows that BSDE $(\xi,g+{\rm d}V)$ admits an $L^1$ solution.

Finally, in the general case, we can approximate $g$ by the sequence $g_n$, where $g_n:=g\vee (-n),\ n\geq 1$. By the previous step there exists a unique $L^1$ solution $(Y^n_\cdot,Z^n_\cdot)$ of BSDE $(\xi,g_n+{\rm d}V)$ for each $n\geq 1$. Repeating arguments in the proof of the previous step yields that $(Y^n_\cdot,Z^n_\cdot)$ converges in $\s^\beta\times \M^\beta$ for each $\beta\in (0,1)$ to the unique $L^1$ solution $(Y_\cdot,Z_\cdot)$ of BSDE $(\xi,g+{d}V)$.
\end{proof}

\begin{thm}
\label{thm:4-ExistenceofBSDEunderHH}
Let $\xi\in\LT$, $V_\cdot\in\vcal^1$, $g^1$ satisfy assumptions \ref{A:(H1)}(i) and \ref{A:(HH)} (resp. \ref{A:(H1)} and \ref{A:(H2')}), $g^2$ satisfy assumption \ref{A:(AA)} and the generator $g:=g^1+g^2$. Then BSDE $(\xi,g+{\rm d}V)$ admits a minimal (resp. maximal) $L^1$ solution.
\end{thm}

\begin{proof}
We only prove the case of the minimal solution under the assumptions \ref{A:(H1)}(i) and \ref{A:(HH)} of $g^1$. In the same way, we can prove another case, and in view of \cref{rmk:2-LinearGrowthOfRhoandPhi}, the case under the assumptions \ref{A:(H1)} and \ref{A:(H2')} of $g^1$ holds true naturally.

Now, we assume that $\xi\in\LT$, $V_\cdot\in\vcal^1$, $g^1$ satisfies \ref{A:(H1)}(i) and \ref{A:(HH)} with $\rho(\cdot)$, $f_\cdot$, $\varphi_\cdot(r)$, $\lambda$ and $\alpha$, $g^2$ satisfies \ref{A:(AA)} with $\tilde f_\cdot$, $\tilde\mu$, $\tilde\lambda$ and $\tilde\alpha$, and the generator $g:=g^1+g^2$. In view of assumptions of $g$,
it is not very hard to prove that for each $n\geq 1$ and $(y,z)\in \R\times\R^d$, the following function
$$
g_n(\omega,t,y,z):=g^1_n(\omega,t,y,z)
+g^2_n(\omega,t,y,z)
$$
with
\begin{equation}
g^1_n(\omega,t,y,z):=\inf\limits_{u\in\R^d}
\left[g^1(\omega,t,y,u)+(n+2\lambda)|u-z|^\alpha\right]
\end{equation}
and
\begin{equation}
g^2_n(\omega,t,y,z):=\inf\limits_{(u,v)\in\R\times\R^d}
\left[g^2(\omega,t,u,v)+(n+2\tilde\mu)|u-y|
+(n+2\tilde\lambda)|v-z|^{\tilde\alpha}\right]
\vspace{0.1cm}
\end{equation}
is well defined and $(\F_t)$-progressively measurable, $\as$, $g_n$ increases in $n$, is continuous in $(y,z)$, and converges locally uniformly in $(y,z)$ to the generator $g$ as $n\To \infty$, $g^1_n$ satisfies \ref{A:(H1)}(i) with the same $\rho(\cdot)$, \ref{A:(HH)} with the same $f_\cdot$, $\varphi_\cdot(r)$, $\lambda$ and $\alpha$, \ref{A:(H1)}(ii) with $|g^1_n(\cdot,0,0)|\leq f_\cdot$,
\ref{A:(H1)}(iii) with the same $\psi_\cdot(r):=2f_\cdot+\varphi_\cdot(r)$ and \ref{A:(H2)} with $\phi(x):=(n+2\lambda)|x|^\alpha$, $\gamma:=n+2\lambda$, $f_\cdot:\equiv 0$ and $\alpha$, $g^2_n$ satisfies \ref{A:(H1)}(i) with $\rho(x):=(n+2\tilde\mu)x$, \ref{A:(HH)} with the same $\tilde f_\cdot$, $\tilde\mu r$, $\tilde\lambda$ and $\tilde\alpha$, \ref{A:(H1)}(ii) with $|g^2_n(\cdot,0,0)|\leq \tilde f_\cdot$,
\ref{A:(H1)}(iii) with the same $\psi_\cdot(r):=2\tilde f_\cdot+\tilde\mu r$ and \ref{A:(H2)} with $\phi(x):=(n+2\tilde\lambda)|x|^{\tilde\alpha}$, $\gamma:=n+2\tilde\lambda$, $f_\cdot:\equiv 0$ and $\tilde\alpha$. Hence, both \ref{A:(H1)} and \ref{A:(H2)} are satisfied by the generator $g_n$ for each $n\geq 1$. It then follows from \cref{thm:4-ExistenceanduniquenssofBSDEunderH2} that BSDE $(\xi,g_n+{\rm d}V)$ admits a unique $L^1$ solution $(Y^n_\cdot,Z^n_\cdot)$ for each $n\geq 1$. Furthermore, in view of \cref{rmk:2-LinearGrowthOfRhoandPhi},
it follows from \ref{A:(H1)}(i) and \ref{A:(HH)} of $g^1_n$ together with \ref{A:(HH)} of $g^2_n$ that $\as$, for each $n\geq 1$ and $(y,z)\in\R\times\R^d$,
\begin{equation}
\label{eq:4-gnSatisfyA-HH}
\begin{array}{lll}
\Dis {\rm sgn}(y)g_n(\cdot,y,z)&\leq &\Dis {\rm sgn}(y)(g^1_n(\cdot,y,z)-g^1_n(\cdot,0,z))
+|g^1_n(\cdot,0,z)|+|g^2_n(\cdot,y,z)|\\
&\leq & \Dis \rho(|y|)+f_\cdot+\lambda|z|^{\alpha}
+\tilde f_\cdot+\tilde\mu|y|+\tilde\lambda|z|^{\tilde\alpha} \\
&\leq & \Dis A+f_\cdot+\tilde f_\cdot +(A+\tilde\mu)|y|+\lambda(1+|z|)^\alpha
+\tilde\lambda(1+|z|)^{\tilde\alpha}
=:\bar g(\cdot,y,z).\vspace{0.1cm}
\end{array}
\end{equation}
In the sequel, in the same way as in the proof of \cref{thm:4-ExistenceanduniquenssofBSDEunderH2}, we can deduce that BSDE $(|\xi|,\bar g+{\rm d}|V|)$ admits a unique $L^1$ solution $(\bar Y_\cdot,\bar Z_\cdot)$ with $\bar Y_\cdot\geq 0$, BSDE $(-|\xi|,-\bar g-{\rm d}|V|)$ admits a unique $L^1$ solution $(\underline Y_\cdot,\underline Z_\cdot)$ with $\underline Y_\cdot\leq 0$, and in view of  \eqref{eq:4-gnSatisfyA-HH} and the fact that $\as$,
$$
\mathbbm{1}_{\{Y^n_t>\bar Y_t\}}\left(g_n(t,Y^n_t,Z^n_t)-\bar g(t,Y^n_t,Z^n_t)\right)\leq 0\ \ {\rm and}\ \ \mathbbm{1}_{\{\underline Y_t> Y^n_t\}}\left(-\bar g(t,Y^n_t,Z^n_t)- g_n(t,Y^n_t,Z^n_t)\right)\leq 0,
$$
it follows from \cref{pro:3-ComparisonTheoremofDRBSDE} and \cref{cor:3-CorollaryOfComparisonTheorem} that $\underline Y_\cdot\leq Y^n_\cdot\leq Y^{n+1}_\cdot\leq \bar Y_\cdot$ for each $n\geq 1$. Thus, by \eqref{eq:4-gnSatisfyA-HH} we deduce that \eqref{eq:3-SemiLinearGrowthCondition} holds. In addition, in view of \ref{A:(HH)} of $g^1_n$ and $g^2_n$, we have for each $n,k\geq 1$,
\begin{equation}
\label{eq:5-GeneralGrowthOfgnYnZn}
\begin{array}{lll}
\Dis|g_n(\cdot,Y^n_\cdot,Z^n_\cdot)
\mathbbm{1}_{|Y^n_\cdot|\leq k}|
&\leq &\Dis f_\cdot+\varphi_\cdot(|Y^n_\cdot|)
\mathbbm{1}_{|Y^n_\cdot|\leq k}+ \lambda|Z^n_\cdot|^\alpha+\tilde f_\cdot+\tilde\mu |Y^n_\cdot|\mathbbm{1}_{|Y^n_\cdot|\leq k}
+\tilde\lambda |Z^n_\cdot|^{\tilde\alpha}\vspace{0.1cm}\\
&\leq &\Dis f_\cdot+\tilde f_\cdot+\varphi_\cdot(k)+\tilde\mu k+\lambda+\tilde\lambda
+(\lambda+\tilde\lambda)|Z^n_\cdot|.
\end{array}
\end{equation}
Hence, \eqref{eq:3-LinearGrowthofgnYnZn--k} holds also true since $\varphi_\cdot(k)\in \hcal$ and $f_\cdot,\tilde f_\cdot\in \hcal^1$. Now, we have checked all the conditions in \cref{pro:3-Approximation} with $L_\cdot=-\infty$, $U_\cdot=+\infty$ and $K^n_\cdot=A^n_\cdot\equiv 0$, and it follows that BSDE $(\xi,g+{\rm d}V)$ admits an $L^1$ solution $(Y_\cdot,Z_\cdot)$ such that for each $\beta\in (0,1)$,
\begin{equation}
\label{eq:4-ConvergenceOfYnZnInSp}
\lim\limits_{n\To\infty}(\|Y_\cdot^n
-Y_\cdot\|_{\s^\beta}+\|Z_\cdot^n
-Z_\cdot\|_{\M^\beta})=0.
\end{equation}

Finally, we show that $(Y_\cdot,Z_\cdot)$ is just the minimal $L^1$ solution of BSDE $(\xi,g+{\rm d}V)$. In fact, if $(Y'_\cdot,Z'_\cdot)$ is also an $L^1$ solution of BSDE $(\xi,g+{\rm d}V)$, then noticing that $g_n\leq g$ and $g_n$ satisfies \ref{A:(H1)} and \ref{A:(H2)} for each $n\geq 1$, it follows from \cref{cor:3-CorollaryOfComparisonTheorem} that
that $Y_t^n\leq Y'_t$ for each $t\in\T$ and $n\geq 1$. Thus, by \eqref{eq:4-ConvergenceOfYnZnInSp} we know that for each $t\in \T$,
$$Y_t\leq Y'_t.$$
\cref{thm:4-ExistenceofBSDEunderHH} is then proved.
\end{proof}

By \cref{cor:3-CorollaryOfComparisonTheorem} together with the proof of \cref{thm:4-ExistenceofBSDEunderHH} it is easy to verify that under \ref{A:(H1)}(i) and \ref{A:(HH)} (resp. \ref{A:(H1)} and \ref{A:(H2')}) together with \ref{A:(AA)}, the comparison theorem for the maximal (resp. minimal) $L^1$ solutions of the BSDEs holds true. More precisely, we have

\begin{cor}
\label{cor:4-ComparisonOfMinSolutionUnderHH}
Assume that for $j=1,2$, $\xi^j\in\LT$, $V_\cdot^j\in \vcal^1$, $g^{j,1}$ satisfies \ref{A:(H1)}(i) and \ref{A:(HH)} (resp. \ref{A:(H1)} and \ref{A:(H2')}), $g^{j,2}$ satisfies \ref{A:(AA)}, $g^j:=g^{j,1}+g^{j,2}$ and that $(Y_\cdot^j,Z_\cdot^j)$ is the maximal (resp. minimal) $L^1$ solution of BSDE $(\xi^j,g^j+{\rm d}V^j)$ (recall \cref{thm:4-ExistenceofBSDEunderHH}). If $\xi^1\leq \xi^2$, ${\rm d}V^1\leq {\rm d}V^2$, and
$$\as,\ \RE\ (y,z)\in \R\times \R^d,\ \  g^{1,1}(t,y,z)\leq g^{2,1}(t,y,z)\ \ {\rm and}\ \ g^{1,2}(t,y,z)\leq g^{2,2}(t,y,z),$$
then $Y_t^1\leq Y_t^2$ for each $t\in \T$.
\end{cor}

\begin{rmk}
Observe that either $g^1\equiv 0$ or $g^2\equiv 0$ is a special case of the generator $g:=g^1+g^2$ in
\cref{cor:4-ComparisonOfMinSolutionUnderHH} and
\cref{thm:4-ExistenceofBSDEunderHH}. Hence, they generalize some known results on the $L^1$ solution of BSDEs.
\end{rmk}

\section{Existence, uniqueness and approximation for $L^1$ solutions of RBSDEs}
\label{sec:5-ExistenceRBSDEs}
\setcounter{equation}{0}

In this section, we will establish some existence, uniqueness and approximation results on $L^1$ solutions of RBSDEs with one continuous barrier under general assumptions.\vspace{-0.1cm}

\begin{thm}
\label{thm:5-ExistenceanduniquenssofRBSDEunderH2}
Let $V_\cdot\in\vcal^1$ and the generator $g$ satisfy assumptions \ref{A:(H1)} and \ref{A:(H2)}.\vspace{-0.1cm}
\begin{itemize}
\item [(i)] Assume that \ref{A:(H3L)}(i) holds true for $L_\cdot$ and $\xi$. Then $\underline{R}$BSDE $(\xi,g+{\rm d}V,L)$ admits an $L^1$ solution iff \ref{A:(H3L)}(ii) is satisfied. Furthermore, if \ref{A:(H3L)}(ii) holds also true, then $\underline{R}$BSDE $(\xi,g+{\rm d}V,L)$ admits a unique $L^1$ solution $(Y_\cdot,Z_\cdot,K_\cdot)$ such that for each $\beta\in (0,1)$,
    \begin{equation}
    \label{eq:5-ConvergenceOfYnZnKn}
    \Lim \left(\|Y^n_\cdot-Y_\cdot\|_{\s^\beta}
    +\|Z^n_\cdot-Z_\cdot\|_{\M^\beta}+
    \|K^n_\cdot-K_\cdot\|_{\s^\beta}\right)=0,
    \end{equation}
    where for each $n\geq 1$, $(Y^n_\cdot,Z^n_\cdot)$ is the unique $L^1$ solution of BSDE $(\xi,\bar {g}_n+{\rm d}V)$ with $\bar {g}_n(t,y,z):=g(t,y,z)+n(y-L_t)^-$, i.e., \eqref{eq:3-PenalizationForBSDE}, (recall \cref{thm:4-ExistenceanduniquenssofBSDEunderH2}),  and
    \begin{equation}
    \label{eq:5-DefinitionOfKn}
    K^n_t:=n\int_0^t(Y^n_s-L_s)^-{\rm d}s, \ \ t\in\T;
    \end{equation}
\item [(ii)] Assume that \ref{A:(H3U)}(i) holds true for $U_\cdot$ and $\xi$. Then $\bar {R}$BSDE $(\xi,g+{\rm d}V,U)$ admits an $L^1$ solution iff \ref{A:(H3U)}(ii) is satisfied. Furthermore, if \ref{A:(H3U)}(ii) holds also true, then $\bar{R}$BSDE $(\xi,g+{\rm d}V,U)$ admits a unique $L^1$ solution $(Y_\cdot,Z_\cdot,A_\cdot)$ such that for each $\beta\in (0,1)$,
    \begin{equation}
     \label{eq:5-ConvergenceOfYnZnAn}
    \Lim \left(\|Y^n_\cdot-Y_\cdot\|_{\s^\beta}
    +\|Z^n_\cdot-Z_\cdot\|_{\M^\beta}+
    \|A^n_\cdot-A_\cdot\|_{\s^\beta}\right)=0,
    \end{equation}
    where for each $n\geq 1$, $(Y^n_\cdot,Z^n_\cdot)$ is the unique $L^1$ solution of BSDE $(\xi,\underline {g}_n+{\rm d}V)$ with $\underline {g}_n(t,y,z):=g(t,y,z)-n(y-U_t)^+$, i.e., \eqref{eq:3-PenalizationForBSDEwithAn}, (recall \cref{thm:4-ExistenceanduniquenssofBSDEunderH2}),  and
    \begin{equation}
    \label{eq:5-DefinitionOfAn}
    A^n_t:=n\int_0^t(Y^n_s-U_s)^+{\rm d}s, \ \ t\in \T.
    \end{equation}
\end{itemize}
\end{thm}

\begin{proof}
We only prove the case of (i), the proof of (ii) is similar. We assume that $V_\cdot\in\vcal^1$, the generator $g$ satisfies \ref{A:(H1)} and \ref{A:(H2)}, and \ref{A:(H3L)}(i) holds true for $L_\cdot$ and $\xi$. If $\underline{R}$BSDE $(\xi,g+{\rm d}V,L)$ admits an $L^1$ solution $(Y_\cdot,Z_\cdot,K_\cdot)$, then from \cref{lem:2-gBelongstoH1} we know that $g(\cdot,Y_\cdot,Z_\cdot)\in \hcal^1$ and $g(\cdot,Y_\cdot,0)\in \hcal^1$. Thus, \ref{A:(H3L)}(ii) is satisfied with \vspace{-0.2cm}
$$
(C_\cdot,H_\cdot):=(-\int_0^\cdot g(s,Y_s,Z_s){\rm d}s-V_\cdot-K_\cdot,\  Z_\cdot)\vspace{-0.1cm}
$$
and $X_\cdot:=Y_\cdot$. The necessity is proved.

We further assume that \ref{A:(H3L)}(ii) holds. The uniqueness of the $L^1$ solution of $\underline{R}$BSDE $(\xi,g+{\rm d}V,L)$ follows from \cref{pro:3-ComparisonTheoremofDRBSDE}. In the sequel, we prove \eqref{eq:5-ConvergenceOfYnZnKn}. For each $n\geq 1$, let $(Y^n_\cdot,Z^n_\cdot)$ be the unique $L^1$ solution of BSDE $(\xi,\bar {g}_n+{\rm d}V)$ with $\bar {g}_n(t,y,z):=g(t,y,z)+n(y-L_t)^-$ and \eqref{eq:5-DefinitionOfKn}. We first show that there exists a process $\bar X_\cdot\in \cap_{\beta\in (0,1)}\s^\beta$ of the class (D) such that for each $n\geq 1$,
\begin{equation}
\label{eq:5-Y1leqYnleqBarX}
Y^1_\cdot\leq Y^n_\cdot\leq Y^{n+1}_\cdot\leq\bar X_\cdot.\vspace{-0.2cm}
\end{equation}
In fact, it follows from \ref{A:(H3L)}(ii) that there exists two processes $(C_\cdot,H_\cdot)\in \vcal^1\times \M^\beta$ for each $\beta\in (0,1)$ such that
\begin{equation}
\label{eq:5-RepresentationOfXt}
X_t=X_T-\int_t^T{\rm d}C_s-\int_t^TH_s\cdot {\rm d}B_s,\ \ t\in\T\vspace{0.1cm}
\end{equation}
belongs to the class (D), $g(\cdot,X_\cdot,0)\in \hcal^1$ and $L_t\leq X_t$ for each $t\in \T$. And, by \ref{A:(H2)}(ii) together with H\"{o}lder's inequality we know that $\as$,\vspace{-0.1cm}
$$
|g(\cdot,X_\cdot,H_\cdot)|\leq |g(\cdot,X_\cdot,0)|+\gamma
(f_\cdot+|X_\cdot|+H_\cdot)^\alpha\ \in \hcal^1.\vspace{-0.1cm}
$$
Then, the equation \eqref{eq:5-RepresentationOfXt} can be rewritten in the form
$$
\begin{array}{lll}
X_t&=&\Dis X_T+\int_t^Tg(s,X_s,H_s){\rm ds}+\int_t^T{\rm d}V_s-\int_t^T\left(g^+(s,X_s,H_s){\rm d}s+{\rm d}C^+_s+{\rm d}V^+_s\right)\vspace{0.1cm}\\
&&\Dis +\int_t^T\left(g^-(s,X_s,H_s){\rm d}s+{\rm d}C^-_s+{\rm d}V^-_s\right)-\int_t^TH_s\cdot {\rm d}B_s,\ \ t\in\T,
\end{array}
$$
where $g^+:=g\vee 0$, $g^-:=(-g)\vee 0$, $V_\cdot-V_0=V^+_\cdot-V^-_\cdot$ and
$C_\cdot-C_0=C^+_\cdot-C^-_\cdot$ with $V^+_\cdot,V^-_\cdot,C^+_\cdot,C^-_\cdot\in \vcal^{+,1}$. On the other hand, in view of \ref{A:(H1)} and \ref{A:(H2)}, by \cref{thm:4-ExistenceanduniquenssofBSDEunderH2} we know that there exists a unique $L^1$ solution $(\bar X_\cdot,\bar Z_\cdot)$ of the BSDE
$$
\begin{array}{lll}
\bar X_t&=&\Dis X_T\vee \xi+\int_t^Tg(s,\bar X_s,\bar Z_s){\rm ds}+\int_t^T{\rm d}V_s\\
&&\Dis +\int_t^T\left(g^-(s,X_s,H_s){\rm d}s+{\rm d}C^-_s+{\rm d}V^-_s\right)-\int_t^T\bar Z_s\cdot {\rm d}B_s,\ \ t\in\T.\vspace{0.1cm}
\end{array}
$$
And, it follows from \cref{pro:3-ComparisonTheoremofDRBSDE} that $L_t\leq X_t\leq \bar X_t$ for each $t\in \T$. Therefore, for each $n\geq 1$,
$$
\begin{array}{lll}
\bar X_t&=&\Dis X_T\vee \xi+\int_t^Tg(s,\bar X_s,\bar Z_s){\rm ds}+\int_t^T{\rm d}V_s+n\int_t^T\left(\bar X_s-L_s\right)^-{\rm d}s\vspace{0.1cm}\\
&&\Dis +\int_t^T\left(g^-(s,X_s,H_s){\rm d}s+{\rm d}C^-_s+{\rm d}V^-_s\right)-\int_t^T\bar Z_s\cdot {\rm d}B_s,\ \ t\in\T.\vspace{0.1cm}
\end{array}
$$
Thus, by \cref{cor:3-CorollaryOfComparisonTheorem} we know that \eqref{eq:5-Y1leqYnleqBarX} holds true.

In the sequel, in view of assumptions \ref{A:(H1)} and \ref{A:(H2)}, it follows from \cref{lem:2-gBelongstoH1} that $g(\cdot,Y^1_\cdot,0)\in \hcal^1$ and $g(\cdot,\bar X_\cdot,0)\in \hcal^1$, then from \cref{lem:2-SublinearOfg} together with \eqref{eq:5-Y1leqYnleqBarX} that \eqref{eq:3-SubLinearGrowthofgYnZn} holds true for each $n\geq 1$, with
$$\bar f_\cdot:=|g(\cdot,Y^1_\cdot,0)|+|g(\cdot,\bar X_\cdot,0)|+(\gamma+A)(|Y^1_\cdot|+|\bar X_\cdot|)+ \gamma (1+f_\cdot)+A\ \in\hcal^1,
$$
$\bar \lambda :=\gamma$ and $\alpha$. Thus, we have verified that all conditions in \cref{pro:3-PenalizationOfBSDE} (i) are satisfied, and then it follows that there exists an $L^1$ solution $(Y_\cdot, Z_\cdot, K_\cdot)$ of $\underline R$BSDE $(\xi,g+{\rm d}V,L)$ such that for each $\beta\in (0,1)$,
\begin{equation}
\label{eq:5-ConvergenceOfYnZn}
\lim\limits_{n\To \infty}\left(\| Y_\cdot^n- Y_\cdot\|_{\s^\beta}+ \| Z_\cdot^n- Z_\cdot\|_{\M^\beta}\right)=0\vspace{0.1cm}
\end{equation}
and there exists a subsequence $\{ K_\cdot^{n_j}\}$ of $\{ K_\cdot^n\}$ such that
$$
\lim\limits_{j\To\infty}\sup\limits_{t\in\T}
| K_t^{n_j}- K_t|=0.
$$

Finally, in view of \eqref{eq:5-ConvergenceOfYnZn}, in order to prove \eqref{eq:5-ConvergenceOfYnZnKn} we need only to show that for each $\beta\in (0,1)$
\begin{equation}
\label{eq:5-5.8}
\lim\limits_{n\To\infty}\left\| \int_0^\cdot g(s,Y_s^n,Z_s^n){\rm d}s-\int_0^\cdot g(s,Y_s,Z_s){\rm d}s\right\|_{\s^\beta}=0.\vspace{0.1cm}
\end{equation}
The proof is similar to that of Theorem 5.8 in \citet{Fan2017AMS}, but for readers' convenience we list it as follows. In fact, it follows from \ref{A:(H2)} (i) that $\as$, for each $n\geq 1$,
$$
\begin{array}{lll}
\Dis |g(\cdot,Y_\cdot^n,Z_\cdot^n)-g(\cdot,Y_\cdot,Z_\cdot)|
&\leq &\Dis  |g(\cdot,Y_\cdot^n,Z_\cdot^n)-g(\cdot,Y_\cdot^n,Z_\cdot)| +|g(\cdot,Y_\cdot^n,Z_\cdot)-g(\cdot,Y_\cdot,Z_\cdot)|\\
&\leq &\Dis |g(\cdot,Y_\cdot^n,Z_\cdot)-g(\cdot,Y_\cdot,Z_\cdot)|
+\phi(|Z_\cdot^n-Z_\cdot|).
\end{array}
$$
Thus, making use of the following basic inequality (see \citet{FanJiang2010CRA} for details)
$$
\phi(x)\leq (m+2A)x+\phi\left({2A\over m+2A}\right),\ \ \RE\ x\geq 0,\ \ \RE m\geq 1
$$
together with H\"{o}lder's inequality, we get that for each $n,m\geq 1$ and $\beta\in (0,1)$,
\begin{equation}
\label{eq:5-5.9}
\begin{array}{lll}
&&\Dis \left\| \int_0^\cdot g(s,Y_s^n,Z_s^n){\rm d}s-\int_0^\cdot g(s,Y_s,Z_s){\rm d}s\right\|_{\s^\beta}\leq \Dis \E\left[\left(
\int_0^T |g(t,Y_t^n,Z_t^n)
-g(t,Y_t,Z_t)| {\rm d}t\right)^\beta\right]\vspace{0.1cm}\\
&\leq &\Dis  \E\left[\left(
\int_0^T |g(t,Y_t^n,Z_t)
-g(t,Y_t,Z_t)| {\rm d}t\right)^\beta\right]
+(m+2A)^\beta T^{\beta\over 2}\|Z_\cdot^n-Z_\cdot\|_{\M^\beta}+\phi^\beta
({2A\over m+2A}) T^\beta.
\end{array}
\end{equation}
Furthermore, in view of assumptions \ref{A:(H1)} and \ref{A:(H2)} together with \eqref{eq:5-Y1leqYnleqBarX}, it follows from \cref{lem:2-SublinearOfg} and \cref{lem:2-gBelongstoH1} that for each $n\geq 1$,\vspace{-0.1cm}
$$
|g(\cdot,Y^n_\cdot,Z_\cdot)-
g(\cdot,Y_\cdot,Z_\cdot)|\leq |g(\cdot,Y^n_\cdot,Z_\cdot)|+
|g(\cdot,Y_\cdot,Z_\cdot)|\leq \bar f_\cdot+\gamma |Z_\cdot|^\alpha + |g(\cdot,Y_\cdot,Z_\cdot)| \in \hcal^1. \vspace{-0.1cm}
$$
Then, Lebesgue's dominated convergence theorem yields that for each $\beta\in (0,1)$,
\begin{equation}
\label{eq:5-5.10}
\lim\limits_{n\To\infty}\E\left[\left(
\int_0^T |g(t,Y_t^n,Z_t)
-g(t,Y_t,Z_t)| {\rm d}t\right)^\beta\right]=0.
\end{equation}
Thus, letting first $n\To \infty$, and then $m\To\infty$ in \eqref{eq:5-5.9}, in view of \eqref{eq:5-5.10}, \eqref{eq:5-ConvergenceOfYnZn} and the fact that $\phi(\cdot)$ is continuous and $\phi(0)=0$, we get \eqref{eq:5-5.8}. The proof of \cref{thm:5-ExistenceanduniquenssofRBSDEunderH2} is then completed.
\end{proof}

\begin{cor}
\label{cor:5-ComparisonForKofRBSDE}
Assume that $\xi^1,\xi^2\in\LT$ with $\xi^1\leq \xi^2$, $V_\cdot^1, V_\cdot^2\in\vcal^{1}$ with ${\rm d}V^1\leq {\rm d}V^2$, and both generators $g^1$ and $g^2$ satisfy \ref{A:(H1)} and \ref{A:(H2)} with\vspace{-0.1cm}
$$
\as ,\ \RE\ (y,z)\in \R\times \R^d,\ \ g^1(t,y,z)\leq g^2(t,y,z).\vspace{-0.3cm}
$$
We have
\begin{itemize}
\item [(i)] For $i=1,2$, let \ref{A:(H3L)} hold for $\xi^i$, $L_\cdot^i$ and $X_\cdot^i$ associated with $g^i$, and $(Y_\cdot^i,Z_\cdot^i,K_\cdot^i)$ be the unique $L^1$ solution of $\underline {R}$BSDE $(\xi^i,g^i+{\rm d}V^i,L^i)$ (recall \cref{thm:5-ExistenceanduniquenssofRBSDEunderH2}). If $L^1_\cdot=L^2_\cdot$, then ${d}K^1\geq {\rm d}K^2$.
\item [(ii)] For $i=1,2$, let \ref{A:(H3U)} hold for $\xi^i$, $U_\cdot^i$ and $X_\cdot^i$ associated with $g^i$, and $(Y_\cdot^i,Z_\cdot^i,A_\cdot^i)$ be the unique $L^1$ solution of $\bar{R}$BSDE $(\xi^i,g^i+{\rm d}V^i,U^i)$ (recall \cref{thm:5-ExistenceanduniquenssofRBSDEunderH2}). If $U^1_\cdot=U^2_\cdot$, then ${d}A^1\leq {\rm d}A^2$.
\end{itemize}
\end{cor}

\begin{proof}
We only prove (i). The proof is classical, and we list it for readers' convenience. For $n\geq 1$ and $i=1,2$, by \cref{thm:4-ExistenceanduniquenssofBSDEunderH2} we let $(Y_\cdot^{i,n},Z_\cdot^{i,n})$ be the unique $L^1$ solution of the following penalization BSDE:
$$
Y_t^{i,n}=\xi^i+\int_t^T g^i(s,Y_s^{i,n},Z_s^{i,n}){\rm d}s+\int_t^T {\rm d}V_s^i+\int_t^T {\rm d}K_s^{i,n}-\int_t^T Z_s^{i,n}\cdot {\rm d}B_s,\ \ t\in \T
$$
with
$$
K_t^{i,n}:=n\int_0^t\left(Y_s^{i,n}-L_s^i\right)^-{\rm d}s,\ \ t\in\T.\vspace{0.1cm}
$$
In view of the assumptions of \cref{cor:5-ComparisonForKofRBSDE}, it follows from \cref{cor:3-CorollaryOfComparisonTheorem} that for each $n\geq 1$, $Y^{1,n}_\cdot\leq Y^{2,n}_\cdot$, and then
$$
K_{t_2}^{1,n}-K_{t_1}^{1,n}=n\int_{t_1}^{t_2}
\left(Y^{1,n}_s-L^1_s\right)^-{\rm d}s\geq n\int_{t_1}^{t_2}
\left(Y^{2,n}_s-L^2_s\right)^-{\rm d}s=K_{t_2}^{2,n}-K_{t_1}^{2,n}\vspace{0.1cm}
$$
for every $n\geq 1$ and $0\leq t_1\leq t_2\leq T$. Since for each $\beta\in (0,1)$, both
$\|K^{1,n}_\cdot-K^1_\cdot\|_{\s^\beta}\ \ {\rm and}\ \ \|K^{2,n}_\cdot-K^2_\cdot\|_{\s^\beta}$ converge to zero as $n\To\infty$ by \cref{thm:5-ExistenceanduniquenssofRBSDEunderH2}, it follows that $K_{t_2}^1-K_{t_1}^1\geq K_{t_2}^2-K_{t_1}^2$ for every $0\leq t_1\leq t_2\leq T$, which proves the desired result.\vspace{0.1cm}
\end{proof}

\begin{thm}
\label{thm:5-ExistenceofRBSDEunderH2'Penalization}
Let $V_\cdot\in\vcal^1$, $g^1$ satisfy assumptions \ref{A:(H1)} and \ref{A:(H2')}, $g^2$ satisfy assumption \ref{A:(AA)} and the generator $g:=g^1+g^2$.\vspace{-0.1cm}
\begin{itemize}
\item [(i)] Assume that \ref{A:(H3L)}(i) holds true for $L_\cdot$ and $\xi$. Then $\underline{R}$BSDE $(\xi,g+{\rm d}V,L)$ admits an $L^1$ solution iff \ref{A:(H3L)}(ii) is satisfied for $X_\cdot$, $L_\cdot$ and $g$ (or $g^1$). Furthermore, if \ref{A:(H3L)}(ii) holds also true for $X_\cdot$, $L_\cdot$ and $g$ (or $g^1$), then $\underline{R}$BSDE $(\xi,g+{\rm d}V,L)$ admits a minimal $L^1$ solution (resp. an $L^1$ solution) $(Y_\cdot,Z_\cdot,K_\cdot)$ such that for each $\beta\in (0,1)$,\vspace{-0.1cm}
    \begin{equation}
    \label{eq:5-ConvergenceOfYnZnH2'}
    \lim\limits_{n\To \infty}\left(\| Y_\cdot^n- Y_\cdot\|_{\s^\beta}+ \| Z_\cdot^n- Z_\cdot\|_{\M^\beta}\right)=0
    \end{equation}
    and there exists a subsequence $\{ K_\cdot^{n_j}\}$ of $\{ K_\cdot^n\}$ such that
    $$
    \lim\limits_{j\To\infty}\sup\limits_{t\in\T}
    |K_t^{n_j}- K_t|=0,
    $$
    where for each $n\geq 1$, $(Y^n_\cdot,Z^n_\cdot)$ is the minimal (resp. maximal) $L^1$ solution of BSDE $(\xi,\bar {g}_n+{\rm d}V)$ with $\bar {g}_n(t,y,z):=g(t,y,z)+n(y-L_t)^-$, i.e., \eqref{eq:3-PenalizationForBSDE}, (recall \cref{thm:4-ExistenceofBSDEunderHH}),  and
    \begin{equation}
    \label{eq:5-DefinitionOfKnH2'}
    K^n_t:=n\int_0^t(Y^n_s-L_s)^-{\rm d}s, \ \ t\in\T;
    \end{equation}
\item [(ii)] Assume that \ref{A:(H3U)}(i) holds true for $U_\cdot$ and $\xi$. Then $\bar {R}$BSDE $(\xi,g+{\rm d}V,U)$ admits an $L^1$ solution iff \ref{A:(H3U)}(ii) is satisfied for $X_\cdot$, $U_\cdot$ and $g$ (or $g^1$). Furthermore, if \ref{A:(H3U)}(ii) holds also true for $X_\cdot$, $U_\cdot$ and $g$ (or $g^1$), then $\bar{R}$BSDE $(\xi,g+{\rm d}V,L)$ admits a maximal
    $L^1$ solution (resp. an $L^1$ solution) $(Y_\cdot,Z_\cdot,A_\cdot)$ such that for each $\beta\in (0,1)$,\vspace{-0.1cm}
    \begin{equation}
    \label{eq:5-ConvergenceOfYnZnH2'}
    \lim\limits_{n\To \infty}\left(\| Y_\cdot^n- Y_\cdot\|_{\s^\beta}+ \| Z_\cdot^n- Z_\cdot\|_{\M^\beta}\right)=0
    \end{equation}
    and there exists a subsequence $\{ A_\cdot^{n_j}\}$ of $\{A_\cdot^n\}$ such that
    $$
    \lim\limits_{j\To\infty}\sup\limits_{t\in\T}
    |A_t^{n_j}-A_t|=0,
    $$
    where for each $n\geq 1$, $(Y^n_\cdot,Z^n_\cdot)$ is the maximal (resp. minimal) $L^1$ solution of BSDE $(\xi,\underline {g}_n+{\rm d}V)$ with $\underline {g}_n(t,y,z):=g(t,y,z)-n(y-U_t)^+$, i.e., \eqref{eq:3-PenalizationForBSDEwithAn}, (recall \cref{thm:4-ExistenceofBSDEunderHH}),  and
    \begin{equation}
    \label{eq:5-DefinitionOfAnH2'}
    A^n_t:=n\int_0^t(Y^n_s-U_s)^+{\rm d}s, \ \ t\in \T.
    \end{equation}
\end{itemize}
\end{thm}

\begin{proof}
We only prove (i), and (ii) can be proved in the same way. In view of \cref{thm:4-ExistenceofBSDEunderHH}, \cref{cor:4-ComparisonOfMinSolutionUnderHH},
\cref{lem:2-SublinearOfg}, \cref{lem:2-gBelongstoH1} and \cref{pro:3-PenalizationOfBSDE}, by a similar argument to that in the proof of \cref{thm:5-ExistenceanduniquenssofRBSDEunderH2} we can prove that all conclusions in (i) of \cref{thm:5-ExistenceofRBSDEunderH2'Penalization} hold true except for the minimal property of the $L^1$ solution $(Y_\cdot,Z_\cdot,K_\cdot)$ of $\underline{R}$BSDE $(\xi,g+{\rm d}V,L)$ when $(Y^n_\cdot,Z^n_\cdot)$ is the minimal $L^1$ solution of penalized BSDE $(\xi,\bar {g}_n+{\rm d}V)$ for each $n\geq 1$. Now, we will show this property.

Indeed, for any $L^1$ solution $(Y'_\cdot,Z'_\cdot,K'_\cdot)$  of
$\underline{R}$BSDE $(\xi,g+{\rm d}V,L)$, it is not hard to check that $(Y'_\cdot,Z'_\cdot)$ is an $L^1$ solution of BSDE $(\xi,\bar {g}_n+{\rm d}\bar V)$ with $\bar V_\cdot:=V_\cdot+K'_\cdot$  for each $n\geq 1$. Thus, in view of the assumption that $(Y^n_\cdot,Z^n_\cdot)$ is the minimal $L^1$ solution of penalized BSDE $(\xi,\bar {g}_n+{\rm d}V)$ for each $n\geq 1$, \cref{cor:4-ComparisonOfMinSolutionUnderHH} yields that for each $n\geq 1$,
$$Y^n_t\leq Y'_t,\ \ t\in \T.$$
Furthermore, since $\Lim \|Y^n_\cdot-Y_\cdot\|_{\s^\beta}=0$ for each $\beta\in (0,1)$, we know that
$$Y_t\leq Y'_t,\ \ t\in \T,$$
which is the desired result.
\end{proof}

\begin{thm}
\label{thm:5-ExistenceofRBSDEunderH2'Approximation}
Let $V_\cdot\in\vcal^1$, $g^1$ satisfy assumptions \ref{A:(H1)} and \ref{A:(H2')}, $g^2$ satisfy assumption \ref{A:(AA)} and the generator $g:=g^1+g^2$.
\begin{itemize}
\item [(i)] Assume that \ref{A:(H3L)} holds true for $L_\cdot$, $\xi$, $X_\cdot$ and $g$ (or $g^1$). Then $\underline{R}$BSDE $(\xi,g+{\rm d}V,L)$ admits a maximal (resp. minimal) $L^1$ solution $(Y_\cdot,Z_\cdot,K_\cdot)$ such that for each $\beta\in (0,1)$,
$$
\lim\limits_{n\To \infty}\left(\|Y_\cdot^n-Y_\cdot\|_{\s^\beta}+
\|Z_\cdot^n-Z_\cdot\|_{\M^\beta}
+\|K_\cdot^n-K_\cdot\|_{\s^1}
\right)=0,
$$
where, for each $n\geq 1$, $(Y^n_\cdot,Z^n_\cdot,K^n_\cdot)$ is the unique $L^1$ solution of
$\underline R$BSDE $(\xi,g_n+{\rm d}V,L)$ with a generator $g_n$ satisfying \ref{A:(H1)}, \ref{A:(H2)} and \ref{A:(H3L)} (recall \cref{thm:5-ExistenceanduniquenssofRBSDEunderH2}(i));

\item [(ii)] Assume that \ref{A:(H3U)} holds true for $U_\cdot$, $\xi$, $X_\cdot$ and $g$ (or $g^1$). Then $\bar {R}$BSDE $(\xi,g+{\rm d}V,U)$ admits a maximal (resp. minimal) $L^1$ solution  $(Y_\cdot,Z_\cdot,A_\cdot)$ such that for each $\beta\in (0,1)$,
$$
\lim\limits_{n\To \infty}\left(\|Y_\cdot^n-Y_\cdot\|_{\s^\beta}+
\|Z_\cdot^n-Z_\cdot\|_{\M^\beta}
+\|A_\cdot^n-A_\cdot\|_{\s^1}
\right)=0,
$$
where, for each $n\geq 1$, $(Y^n_\cdot,Z^n_\cdot,A^n_\cdot)$ is the unique $L^1$ solution of
$\bar R$BSDE $(\xi,g_n+{\rm d}V,U)$ with a generator $g_n$ satisfying \ref{A:(H1)}, \ref{A:(H2)} and \ref{A:(H3U)} (recall \cref{thm:5-ExistenceanduniquenssofRBSDEunderH2}(ii)).
\end{itemize}
\end{thm}

\begin{proof}
We only prove (i) and consider the case of the maximal $L^1$ solution. Now, we assume that $V_\cdot\in\vcal^1$, $g^1$ satisfies \ref{A:(H1)} and \ref{A:(H2')} with $\rho(\cdot)$, $\psi_\cdot(r)$, $f_\cdot$, $\mu$, $\lambda$ and $\alpha$, $g^2$ satisfies \ref{A:(AA)} with $\tilde f_\cdot$, $\tilde\mu$, $\tilde\lambda$ and $\tilde\alpha$, the generator $g:=g^1+g^2$, and \ref{A:(H3L)} holds true for $L_\cdot$, $\xi$, $X_\cdot$ and $g$ (or $g^1$). In view of assumptions of $g$,
it is not very hard to prove that for each $n\geq 1$ and $(y,z)\in \R\times\R^d$, the following function
$$
g_n(\omega,t,y,z):=g^1_n(\omega,t,y,z)
+g^2_n(\omega,t,y,z)
$$
with
\begin{equation}
g^1_n(\omega,t,y,z):=\sup\limits_{u\in\R^d}
\left[g^1(\omega,t,y,u)-(n+2\lambda)|u-z|^\alpha\right]
\end{equation}
and
\begin{equation}
g^2_n(\omega,t,y,z):=\sup\limits_{(u,v)\in\R\times\R^d}
\left[g^2(\omega,t,u,v)-(n+2\tilde\mu)|u-y|
-(n+2\tilde\lambda)|v-z|^{\tilde\alpha}\right]
\vspace{0.1cm}
\end{equation}
is well defined and $(\F_t)$-progressively measurable, $\as$, $g_n$ decreases in $n$, is continuous in $(y,z)$, and converges locally uniformly in $(y,z)$ to the generator $g$ as $n\To \infty$, $g_n$ satisfies \ref{A:(H1)} and \ref{A:(H2)}, and $\as$, for each $n\geq 1$ and $(y,z)\in \R\times\R^d$,
\begin{equation}
\label{eq:5-SublinearOfg1n}
|g^1_n(\cdot,y,z)-g^1(\cdot,y,0)|\leq  f_\cdot+\mu|y|+\lambda|z|^\alpha,
\end{equation}
and
\begin{equation}
\label{eq:5-SublinearOfg2n}
|g^2_n(\cdot,y,z)|\leq  \tilde f_\cdot+\tilde\mu|y|+\tilde\lambda
|z|^{\tilde\alpha}.
\end{equation}
Then, in view of \eqref{eq:5-SublinearOfg1n} and \eqref{eq:5-SublinearOfg2n}, we know that $\as$, $\RE\ n\geq 1$, $\RE\ (y,z)\in\R\times\R^{d}$,
\begin{equation}
\label{eq:5-gnyzLeq}
\Dis|g_n(\cdot,y,z)|\leq |g^1_n(\cdot,y,z)|+|g^2_n(\cdot,y,z)|\leq  |g^1(\cdot,y,0)|+f_\cdot+\mu|y|+\lambda |z|^\alpha+\tilde f_\cdot+\tilde\mu |y|+\tilde\lambda |z|^{\tilde\alpha}.
\end{equation}
Hence, $g_n(\cdot,X_\cdot,0)\in\hcal^1$ and \ref{A:(H3L)} holds true for $L_\cdot$, $\xi$, $X_\cdot$ and $g_n$. It then follows from \cref{thm:5-ExistenceanduniquenssofRBSDEunderH2}(i) that there exists a unique $L^1$ solution $(Y_\cdot^n,Z_\cdot^n, K_\cdot^n)$ of $\underline{R}$BSDE $(\xi,g_n+{\rm d}V,L)$ for each $n\geq 1$.

In the sequel, let
\begin{equation}
\label{eq:5-DefintionOfUnderlineG}
\underline{g}(\cdot,y,z):=g^1(\cdot,y,0)-(f_\cdot+\tilde f_\cdot)-(\mu+\tilde\mu)|y|-\lambda |z|^\alpha-\tilde\lambda |z|^{\tilde\alpha}
\end{equation}
and
\begin{equation}
\label{eq:5-DefintionOfBarG}
\bar g(\cdot,y,z):=g^1(\cdot,y,0)+(f_\cdot+\tilde f_\cdot)+(\mu+\tilde\mu)|y|+
\lambda |z|^\alpha+\tilde\lambda |z|^{\tilde\alpha}.
\vspace{0.2cm}
\end{equation}
Then by \eqref{eq:5-SublinearOfg1n} and \eqref{eq:5-SublinearOfg2n}, $\underline{g}\leq g_n\leq \bar g$ for each $n\geq 1$, and
both $\underline{g}$ and $\bar g$ satisfy \ref{A:(H1)} and \ref{A:(H2)} with
$$
\underline{g}(\cdot, X_\cdot,0)=g^1(\cdot, X_\cdot,0)-(f_\cdot+\tilde f_\cdot)-(\mu+\tilde\mu)|X_\cdot|\in \hcal^1,
$$
$$
\bar g(\cdot, X_\cdot,0)=g^1(\cdot, X_\cdot,0)+(f_\cdot+\tilde f_\cdot)+(\mu+\tilde\mu)|X_\cdot|\in \hcal^1.\vspace{0.1cm}
$$
Thus, \ref{A:(H3L)} holds also true for $L_\cdot$, $\xi$, $X_\cdot$, $\underline {g}$ and $\bar g$. It then follows from \cref{thm:5-ExistenceanduniquenssofRBSDEunderH2} that $\underline{R}$BSDE $(\xi,\underline{g}+{\rm d}V,L)$ and $\underline{R}$BSDE $(\xi,\bar g+{\rm d}V,L)$ admit respectively a unique $L^1$ solution
$(\underline{Y}_\cdot,\underline{Z}_\cdot, \underline{K}_\cdot)$ and $(\bar Y_\cdot,\bar Z_\cdot, \bar K_\cdot)$, and by \cref{cor:3-CorollaryOfComparisonTheorem} and \cref{cor:5-ComparisonForKofRBSDE} we know that for each $n\geq 1$,
\begin{equation}
\label{eq:5-BoundOfYnKn}
\underline{Y}_\cdot\leq Y^{n+1}_\cdot\leq Y^n_\cdot\leq \bar Y_\cdot\ \ {\rm and}\ \ {\rm d}\bar K\leq {\rm d}K^n\leq {\rm d}K^{n+1}\leq {\rm d}\underline K.
\end{equation}

Furthermore, it follows from \cref{lem:2-gBelongstoH1} that $\underline {g}(\cdot,\underline{Y}_\cdot,0)\in \hcal^1$ and $\bar {g}(\cdot,\bar{Y}_\cdot,0)\in \hcal^1$, and then from \eqref{eq:5-DefintionOfUnderlineG} and \eqref{eq:5-DefintionOfBarG} that
$g^1(\cdot,\underline{Y}_\cdot,0)\in \hcal^1\ \  {\rm and}\ \ g^1(\cdot,\bar{Y}_\cdot,0)\in \hcal^1$. And, in view of \eqref{eq:5-BoundOfYnKn} together with assumptions \ref{A:(H1)} and \ref{A:(H2')} of $g^1$, it follows from \cref{lem:2-SublinearOfg} that for each $n\geq 1$,
\begin{equation}
\label{eq:5-g1nynznLeq}
|g^1(\cdot,Y^n_\cdot,Z^n_\cdot)|\leq |g^1(\cdot,\underline Y_\cdot,0)|+|g^1(\cdot,\bar Y_\cdot,0)|+(\mu+A)(|\underline Y_\cdot|+|\bar Y_\cdot|)+f_\cdot+A+\lambda|Z^n_\cdot|^\alpha.
\end{equation}
Then, by \eqref{eq:5-g1nynznLeq}, \eqref{eq:5-SublinearOfg2n} and \eqref{eq:5-BoundOfYnKn} we can conclude that
\eqref{eq:3-LinearGrowthofgnYnZn} holds true with
$$
\bar f_\cdot:=|g^1(\cdot,\underline Y_\cdot,0)|+|g^1(\cdot,\bar Y_\cdot,0)|+(\mu+A+\tilde\mu)(|\underline Y_\cdot|+|\bar Y_\cdot|)+f_\cdot+\tilde f_\cdot+A+\lambda+\tilde\lambda\ \in \hcal^1
$$
and $\bar\lambda:=\lambda+\tilde\lambda$. Thus, in view of \cref{rmk:3-rmk3.4}, we have checked all the conditions in \cref{pro:3-Approximation} with $U_\cdot=+\infty$ and $A^n_\cdot\equiv 0$, and it follows that $\underline {R}$BSDE $(\xi,g+{\rm d}V,L)$ admits an $L^1$ solution $(Y_\cdot,Z_\cdot,K_\cdot)$ such that for each $\beta\in (0,1)$,\vspace{-0.1cm}
\begin{equation}
\label{eq:5-ConvergenceOfYnZnKnInSp}
\lim\limits_{n\To\infty}(\|Y_\cdot^n
-Y_\cdot\|_{\s^\beta}+\|Z_\cdot^n
-Z_\cdot\|_{\M^\beta}+\|K_\cdot^n
-K_\cdot\|_{\s^1})=0.\vspace{0.1cm}
\end{equation}

Finally, we show that $(Y_\cdot,Z_\cdot,K_\cdot)$ is just the maximal $L^1$ solution of $\underline {R}$BSDE $(\xi,g+{\rm d}V,L)$. In fact, if $(Y'_\cdot,Z'_\cdot,K'_\cdot)$ is also an $L^1$ solution of $\underline {R}$BSDE $(\xi,g+{\rm d}V,L)$, then noticing that $g_n\geq g$ and $g_n$ satisfies \ref{A:(H1)} and \ref{A:(H2)} for each $n\geq 1$, it follows from \cref{cor:3-CorollaryOfComparisonTheorem} that
$Y_t^n\geq Y'_t$ for each $t\in\T$ and $n\geq 1$. Thus, by \eqref{eq:5-ConvergenceOfYnZnKnInSp} we know that for each $t\in \T$,
$$Y_t\geq Y'_t.$$
\cref{thm:5-ExistenceofRBSDEunderH2'Approximation} is then proved.
\end{proof}

By \cref{cor:3-CorollaryOfComparisonTheorem},
\cref{cor:5-ComparisonForKofRBSDE} and the proof of \cref{thm:5-ExistenceofRBSDEunderH2'Approximation},
it is not hard to verify the following comparison result for the minimal (resp. maximal) $L^1$ solutions of Reflected BSDEs.\vspace{-0.1cm}

\begin{cor}
\label{cor:5-ComparisonForYandKofRBSDEUnderH2'}
Assume that $\xi^1,\xi^2\in\LT$ with $\xi^1\leq \xi^2$, $V_\cdot^1, V_\cdot^2\in\vcal^{1}$ with ${\rm d}V^1\leq {\rm d}V^2$, $g^{1,1}$ and $g^{2,1}$ satisfy \ref{A:(H1)} and \ref{A:(H2')}, $g^{1,2}$ and $g^{2,2}$ satisfy \ref{A:(AA)}, $g^1:=g^{1,1}+g^{1,2}$ and $g^2:=g^{2,1}+g^{2,2}$ with\vspace{-0.2cm}
$$
\as,\ \RE\ (y,z)\in \R\times \R^d,\ \  g^{1,1}(t,y,z)\leq g^{2,1}(t,y,z)\ \ {\rm and}\ \ g^{1,2}(t,y,z)\leq g^{2,2}(t,y,z).\vspace{-0.2cm}
$$
We have
\begin{itemize}
\item [(i)] For $i=1,2$, let \ref{A:(H3L)} hold for $\xi^i$, $L_\cdot^i$ and $X_\cdot^i$ associated with $g^i$ (or $g^{i,1}$), and $(Y_\cdot^i,Z_\cdot^i,K_\cdot^i)$ be the minimal (resp. maximal) $L^1$ solution of $\underline {R}$BSDE $(\xi^i,g^i+{\rm d}V^i,L^i)$ (recall \cref{thm:5-ExistenceofRBSDEunderH2'Approximation}). If $L^1_\cdot\leq L^2_\cdot$, then $Y^1_t\leq Y^2_t$ for each $t\in\T$, and if $L^1_\cdot=L^2_\cdot$, then ${d}K^1\geq {\rm d}K^2$;
\item [(ii)] For $i=1,2$, let \ref{A:(H3U)} hold for $\xi^i$, $U_\cdot^i$ and $X_\cdot^i$ associated with $g^i$ (or $g^{i,1}$), and $(Y_\cdot^i,Z_\cdot^i,A_\cdot^i)$ be the minimal (resp. maximal) $L^1$ solution of $\bar{R}$BSDE $(\xi^i,g^i+{\rm d}V^i,U^i)$ (recall \cref{thm:5-ExistenceofRBSDEunderH2'Approximation}). If $U^1_\cdot\leq U^2_\cdot$, then $Y^1_t\leq Y^2_t$ for each $t\in\T$, and if $U^1_\cdot=U^2_\cdot$, then ${d}A^1\leq {\rm d}A^2$.\vspace{0.1cm}
\end{itemize}
\end{cor}

The following corollary follows immediately from \cref{thm:5-ExistenceofRBSDEunderH2'Approximation}.
\vspace{-0.1cm}

\begin{cor}
\label{cor:5-ExistenceofRBSDEunderAA}
Let $V_\cdot\in\vcal^1$, \ref{A:(H3)}(i) hold true for $L_\cdot$, $U_\cdot$ and $\xi$, and the generator $g$ satisfy \ref{A:(AA)}.
\begin{itemize}
\item [(i)] If $L^+_\cdot\in \s^1$, then $\underline{R}$BSDE $(\xi,g+{\rm d}V,L)$ admits a minimal (resp. maximal) $L^1$ solution;
\item [(ii)] If $U^-_\cdot\in \s^1$, then $\bar {R}$BSDE $(\xi,g+{\rm d}V,U)$ admits a minimal (resp. maximal) $L^1$ solution.
\end{itemize}
\end{cor}

\section{Existence, uniqueness and approximation for $L^1$ solutions of DRBSDEs}
\label{sec:6-ExistenceOfDRBSDEs}
\setcounter{equation}{0}

In this section, we will establish some existence, uniqueness and approximation results on $L^1$ solutions of RBSDEs with two continuous barriers under general assumptions.\vspace{-0.1cm}

\begin{thm}
\label{thm:6-ExistenceandUniquenessUnder(H2)}
Assume that $V_\cdot\in\vcal^1$, the generator $g$ satisfies assumptions \ref{A:(H1)} and \ref{A:(H2)}, and assumption \ref{A:(H3)}(i) holds true for $L_\cdot,U_\cdot$ and $\xi$. Then, DRBSDE $(\xi,g+{\rm d}V,L,U)$ admits an $L^1$ solution iff \ref{A:(H3)}(ii) is satisfied. And, if \ref{A:(H3)}(ii) holds also true, then DRBSDE $(\xi,g+{\rm d}V,L,U)$ admits a unique $L^1$ solution $(Y_\cdot,Z_\cdot,K_\cdot,A_\cdot)$. Moreover,
\begin{itemize}
\item [(i)] Let $(\underline Y^n_\cdot,\underline Z^n_\cdot,\underline A^n_\cdot)$ be the unique $L^1$ solution of $\bar{R}$BSDE $(\xi,\bar g_n+{\rm d}V,U)$ with $\bar g_n(t,y,z):=g(t,y,z)+n(y-L_t)^-$ for each $n\geq 1$, i.e.,
    \begin{equation}
\label{eq:6-PenalizationForRBSDEwithSuperBarrier}
\left\{
\begin{array}{l}
\Dis\underline{Y}^n_t=\xi+\int_t^T\bar g_n(s,\underline{Y}^n_s,
\underline{Z}^n_s)
{\rm d}s+\int_t^T{\rm d}V_s-\int_t^T{\rm d}\underline A^n_s-\int_t^T\underline{Z}^n_s \cdot {\rm d}B_s,\ \   t\in\T,\\
\Dis \underline{Y}^n_t\leq U_t,\ t\in\T\ \ {\rm and} \ \int_0^T (U_t-\underline Y^n_t){\rm d}\underline A^n_t=0,\\
\Dis \underline K^n_t:=n\int_0^t(
\underline{Y}^n_s-L_s)^-\ {\rm d}s,\ \ t\in\T
\end{array}
\right.
\end{equation}
(Recall \cref{thm:5-ExistenceanduniquenssofRBSDEunderH2}(ii)). Then, for each $\beta\in (0,1)$,
\begin{equation}
\label{eq:6-ConvergenceOfUnderlineYnZnKnAn}
\lim\limits_{n\To \infty}\left(\|\underline Y_\cdot^n- Y_\cdot\|_{\s^\beta}+\|\underline Z_\cdot^n- Z_\cdot\|_{\M^\beta}+\|\underline K_\cdot^n- K_\cdot \|_{\s^\beta}+\|\underline A_\cdot^n- A_\cdot \|_{\s^1}\right)=0.
\end{equation}

\item [(ii)] Let $(\bar Y^n_\cdot,\bar Z^n_\cdot,\bar K^n_\cdot)$ be the unique $L^1$ solution of $\underline{R}$BSDE $(\xi,\underline g_n+{\rm d}V,L)$ with $\underline g_n(t,y,z):=g(t,y,z)-n(y-U_t)^+$ for each $n\geq 1$, i.e., \begin{equation}
\label{eq:6-PenalizationForRBSDEwithLowerBarrier}
\left\{
\begin{array}{l}
\Dis\bar{Y}^n_t=\xi+\int_t^T\underline g_n(s,\bar{Y}^n_s,
\bar{Z}^n_s)
{\rm d}s+\int_t^T{\rm d}V_s+\int_t^T{\rm d}\bar K^n_s-\int_t^T\bar{Z}^n_s \cdot {\rm d}B_s,\ \   t\in\T,\\
\Dis L_t\leq \bar{Y}^n_t,\ t\in\T\ \ {\rm and} \ \int_0^T (\bar Y^n_t-L_t){\rm d}\bar K^n_t=0,\\
\Dis \bar A^n_t:=n\int_0^t(
\bar{Y}^n_s-U_s)^+\ {\rm d}s,\ \ t\in\T
\end{array}
\right.
\end{equation}
(Recall \cref{thm:5-ExistenceanduniquenssofRBSDEunderH2}(i)). Then, for each $\beta\in (0,1)$,
\begin{equation}
\label{eq:6-ConvergenceOfBarYnZnKnAn}
\lim\limits_{n\To \infty}\left(\|\bar Y_\cdot^n- Y_\cdot\|_{\s^\beta}+\|\bar Z_\cdot^n- Z_\cdot\|_{\M^\beta}+\|\bar K_\cdot^n- K_\cdot \|_{\s^1}+\|\bar A_\cdot^n- A_\cdot \|_{\s^\beta}\right)=0.
\end{equation}

\item [(iii)] Let $( Y^n_\cdot, Z^n_\cdot)$ be the unique $L^1$ solution of BSDE $(\xi,g_n+{\rm d}V)$ with $g_n(t,y,z):=g(t,y,z)+n(y-L_t)^- -n(y-U_t)^+$ for each $n\geq 1$, i.e.,
   \begin{equation}
\label{eq:6-PenalizationForBSDE}
\left\{
\begin{array}{l}
\Dis{Y}^n_t=\xi+\int_t^Tg_n(s,{Y}^n_s,
{Z}^n_s){\rm d}s+\int_t^T{\rm d}V_s-\int_t^T{Z}^n_s \cdot {\rm d}B_s,\ \   t\in\T,\vspace{0.2cm}\\
\Dis  K^n_t:=n\int_0^t(
{Y}^n_s-L_s)^-\ {\rm d}s\ \ {\rm and}\ \  A^n_t:=n\int_0^t(
{Y}^n_s-U_s)^+\ {\rm d}s,\ \ t\in\T
\end{array}
\right.
\end{equation}
(Recall \cref{thm:4-ExistenceanduniquenssofBSDEunderH2}). Then, for each $\beta\in (0,1)$,
\begin{equation}
\label{eq:6-ConvergenceOfYnZnKnAn}
\lim\limits_{n\To \infty}\left(\|Y_\cdot^n- Y_\cdot\|_{\s^\beta}+\|Z_\cdot^n- Z_\cdot\|_{\M^\beta}+\|(K_\cdot^n-A_\cdot^n)- (K_\cdot-A_\cdot)\|_{\s^\beta}\right)=0.
\end{equation}
\end{itemize}
\end{thm}

\begin{proof}
We assume that $V_\cdot\in\vcal^1$, the generator $g$ satisfies \ref{A:(H1)} and \ref{A:(H2)}, and \ref{A:(H3)}(i) holds true for $L_\cdot$, $U_\cdot$ and $\xi$. If DRBSDE $(\xi,g+{\rm d}V,L,U)$ admits an $L^1$ solution $(Y_\cdot,Z_\cdot,K_\cdot,A_\cdot)$, then from \cref{lem:2-gBelongstoH1} we know that $g(\cdot,Y_\cdot,Z_\cdot)\in \hcal^1$ and $g(\cdot,Y_\cdot,0)\in \hcal^1$. Thus, \ref{A:(H3)}(ii) is satisfied with $$(C_\cdot,H_\cdot):=(-\int_0^\cdot g(s,Y_s,Z_s){\rm d}s-V_\cdot-K_\cdot+A_\cdot,\  Z_\cdot)$$
and $X_\cdot:=Y_\cdot$. The necessity is proved.

We further assume that \ref{A:(H3)}(ii) holds. The uniqueness of the $L^1$ solution of DRBSDE $(\xi,g+{\rm d}V,L,U)$ follows from \cref{pro:3-ComparisonTheoremofDRBSDE}. In what follows, it follows from \ref{A:(H3)}(ii) that there exists two processes $(C_\cdot,H_\cdot)\in \vcal^1\times \M^\beta$ for each $\beta\in (0,1)$ such that
\begin{equation}
\label{eq:6-RepresentationOfXt}
X_t=X_T-\int_t^T{\rm d}C_s-\int_t^TH_s\cdot {\rm d}B_s,\ \ t\in\T\vspace{0.1cm}
\end{equation}
belongs to the class (D), $g(\cdot,X_\cdot,0)\in \hcal^1$ and $L_t\leq X_t\leq U_t$ for each $t\in \T$. And, by \ref{A:(H2)}(ii) together with H\"{o}lder's inequality we know that $\as$, $|g(\cdot,X_\cdot,H_\cdot)|\leq |g(\cdot,X_\cdot,0)|+\gamma
(f_\cdot+|X_\cdot|+H_\cdot)^\alpha\ \in \hcal^1$, and then
$$
\check{K}_\cdot:=\int_0^\cdot g^-(s,X_s,H_s){\rm d}s+\int_0^\cdot {\rm d}C^-_s+\int_0^\cdot {\rm d}V^-_s \in \vcal^{+,1}\ \
$$
and
$$
\check{A}_\cdot:=\int_0^\cdot g^+(s,X_s,H_s){\rm d}s+\int_0^\cdot {\rm d}C^+_s+\int_0^\cdot {\rm d}V^+_s\in \vcal^{+,1},\vspace{0.2cm}
$$
where $g^+:=g\vee 0$, $g^-:=(-g)\vee 0$, $V_\cdot-V_0=V^+_\cdot-V^-_\cdot$ and
$C_\cdot-C_0=C^+_\cdot-C^-_\cdot$ with $V^+_\cdot,V^-_\cdot,C^+_\cdot,C^-_\cdot\in \vcal^{+,1}$. Thus, the equation \eqref{eq:6-RepresentationOfXt} can be rewritten in the form
$$
X_t=\Dis X_T+\int_t^Tg(s,X_s,H_s){\rm ds}+\int_t^T{\rm d}V_s+
\int_t^T{\rm d}\check{K}_s- \int_t^T{\rm d}\check{A}_s-\int_t^TH_s\cdot {\rm d}B_s,\ \ t\in\T.
$$

Furthermore, in view of \ref{A:(H1)} and \ref{A:(H2)}, by \cref{thm:4-ExistenceanduniquenssofBSDEunderH2} we can let $(\underline{X}_\cdot,\underline{Z}_\cdot)$ be the unique $L^1$ solution of the following BSDE
$$
\underline{X}_t=\Dis X_T\wedge \xi +\int_t^Tg(s,\underline{X}_s,\underline{Z}_s){\rm ds}+\int_t^T{\rm d}V_s-\int_t^T{\rm d}\check{A}_s\Dis -\int_t^T\underline{Z}_s\cdot {\rm d}B_s,\ \ t\in\T
$$
and
$(\bar X_\cdot,\bar Z_\cdot)$ be the unique $L^1$ solution of the BSDE
$$
\bar X_t=\Dis X_T\vee \xi+\int_t^Tg(s,\bar X_s,\bar Z_s){\rm ds}+\int_t^T{\rm d}V_s\Dis +\int_t^T{\rm d}\check{K}_s-\int_t^T\bar Z_s\cdot {\rm d}B_s,\ \ t\in\T.
$$
It follows from \cref{cor:3-CorollaryOfComparisonTheorem} that $\underline X_\cdot\leq X_\cdot\leq \bar X_\cdot$. And, for each $n\geq 1$, by \cref{thm:4-ExistenceanduniquenssofBSDEunderH2} again we can let $(\dot Y^n_\cdot, \dot Z^n_\cdot)$ and $(\ddot Y^n_\cdot, \ddot Z^n_\cdot)$ be respectively the unique $L^1$ solution of the following BSDEs:
$$
\dot{Y}^n_t=\Dis X_T\wedge \xi +\int_t^Tg(s,\dot{Y}^n_s,\dot{Z}^n_s){\rm ds}+\int_t^T{\rm d}V_s+n\int_t^T(\dot Y^n_s-L_s)^-{\rm d}s-\int_t^T{\rm d}\check{A}_s\Dis -\int_t^T\dot{Z}^n_s\cdot {\rm d}B_s
$$
and
$$
\ddot Y^n_t=\Dis X_T\vee \xi+\int_t^Tg(s,\ddot Y^n_s,\ddot Z^n_s){\rm ds}+\int_t^T{\rm d}V_s+\int_t^T{\rm d}\check{K}_s-n\int_t^T(\ddot Y^n_s-U_s)^+{\rm d}s\Dis -\int_t^T\ddot Z^n_s\cdot {\rm d}B_s
$$
with
$$
\dot K^n_t:= n\int_0^t(\dot Y^n_s-L_s)^-{\rm d}s\ \ {\rm and}\ \ \ddot A^n_t:= n\int_0^t(\ddot Y^n_s-U_s)^+{\rm d}s,\ \ t\in\T.\vspace{0.2cm}
$$
In view of $L_\cdot\leq X_\cdot\leq U_\cdot$, it follows from \cref{cor:3-CorollaryOfComparisonTheorem} that for each $n\geq 1$,
\begin{equation}
\label{eq:6-DotYnLeqXLeqU}
\underline X_\cdot\leq \dot Y^n_\cdot\leq X_\cdot\leq U_\cdot\ \ {\rm and}\ \ L_\cdot\leq X_\cdot\leq \ddot Y^n_\cdot\leq \bar X_\cdot.
\end{equation}
Note that \ref{A:(H3L)} holds true for $L_\cdot$, $X_T\wedge\xi$ and $X_\cdot$, and \ref{A:(H3U)} holds true for $U_\cdot$, $X_T\vee\xi$ and $X_\cdot$. In view of \ref{A:(H1)} and \ref{A:(H2)}, it follows from \cref{thm:5-ExistenceanduniquenssofRBSDEunderH2}
together with
\cref{pro:3-PenalizationOfBSDE} that for each $\beta\in (0,1)$,
\begin{equation}
\label{eq:6-DotKnBetaInfinity}
\sup_{n\geq 1}(\E[|\dot K^n_T|^\beta]+\E[|\ddot A^n_T|^\beta])<+\infty,
\end{equation}
for a subsequence $\{n_j\}$ of $\{n\}$,
\begin{equation}
\lim\limits_{j\To\infty} \dot K^{n_j}_T=\dot K_T\in \LT\ \ {\rm and}\ \ \lim\limits_{j\To\infty} \ddot A^{n_j}_T=\dot A_T\in \LT,
\end{equation}
and for a process $\tilde Y_\cdot \in\s$ and each $(\F_t)$-stopping time $\tau$ valued in $\T$,
\begin{equation}
\label{eq:6-DotKn2Leq}
\sup_{n\geq 1}(\E[|\dot K^n_\tau|^2]+\E[|\ddot A^n_\tau|^2])\leq \E[|\tilde Y_{\tau}|^2].
\end{equation}

In the sequel, let $(\underline Y^n_\cdot,\underline Z^n_\cdot,\underline A^n_\cdot)$, $(\bar Y^n_\cdot,\bar Z^n_\cdot,\bar K^n_\cdot)$ and $(Y^n_\cdot, Z^n_\cdot)$ be respectively defined in (i), (ii) and (iii) of \cref{thm:6-ExistenceandUniquenessUnder(H2)} for each $n\geq 1$. Firstly, in view of \eqref{eq:6-DotYnLeqXLeqU}, it follows from \cref{cor:3-CorollaryOfComparisonTheorem} and \cref{cor:5-ComparisonForKofRBSDE} that for each $n\geq 1$,
\begin{equation}
\label{eq:6-UnderlineYnIncrease}
\underline Y^1_\cdot\leq \underline Y^n_\cdot\leq \underline Y^{n+1}_\cdot\leq \bar X_\cdot,\ \ {\rm d}\underline A^n\leq {\rm d}\underline A^{n+1}
\end{equation}
and
\begin{equation}\label{eq:6-BarYnDecrease}
\underline X_\cdot\leq \bar Y^{n+1}_\cdot\leq \bar Y^{n}_\cdot\leq \bar Y^1_\cdot,\ \ {\rm d}\bar K^n\leq {\rm d}\bar K^{n+1}.
\end{equation}
It then follows from \cref{lem:2-SublinearOfg} that for each $n\geq 1$, in view of \ref{A:(H1)} and \ref{A:(H2)},
\begin{equation}
|g(\cdot,\underline Y^n_\cdot,\underline Z^n_\cdot)|\leq
|g(\cdot,\underline Y^1_\cdot,0)|+|g(\cdot,\bar X_\cdot,0)|+(\gamma+A)(|\underline Y^1_\cdot|+|\bar X_\cdot|)+ \gamma (1+f_\cdot)+A+\gamma |\underline Z^n_\cdot|^\alpha
\end{equation}
and
\begin{equation}
|g(\cdot,\bar Y^n_\cdot,\bar Z^n_\cdot)|\leq
|g(\cdot,\underline X_\cdot,0)|+|g(\cdot,\bar Y^1_\cdot,0)|+(\gamma+A)(|\underline X_\cdot|+|\bar Y^1_\cdot|)+ \gamma (1+f_\cdot)+A+\gamma |\bar Z^n_\cdot|^\alpha
\end{equation}
with, by \cref{lem:2-gBelongstoH1},
\begin{equation}
\label{eq:6-gUnderlineX+BarXBelongstoH1}
g(\cdot,\underline Y^1_\cdot,0)\in\hcal^1,\ \ g(\cdot,\bar X_\cdot,0)\in\hcal^1,\ \  g(\cdot,\underline X_\cdot,0)\in\hcal^1\ \  {\rm and}\ \ g(\cdot,\bar Y^1_\cdot,0)\in\hcal^1.
\end{equation}
And, in view of \eqref{eq:6-DotYnLeqXLeqU}, by \cref{pro:3-ComparisonTheoremofDRBSDE}
with \cref{rmk:3-RemarkOfComparisionTheorem}
we deduce that for each $n\geq 1$,
$$
\dot Y^n_\cdot\leq \underline Y^n_\cdot\ \ {\rm and}\ \  \bar Y^n_\cdot\leq \ddot Y^n_\cdot,
$$
which means that
\begin{equation}
\label{eq:6-DotYnLeqUnderlinYn}
\underline K^n_\cdot=n\int_0^\cdot (\underline Y^n_s-L_s)^-{\rm d}s\leq n\int_0^\cdot (\dot Y^n_s-L_s)^-{\rm d}s=\dot K^n_\cdot
\end{equation}
and
\begin{equation}
\label{eq:6-DotYnLeqUnderlinYn}
\bar A^n_\cdot=n\int_0^\cdot (\bar Y^n_s-U_s)^+{\rm d}s\leq n\int_0^\cdot (\ddot Y^n_s-U_s)^+{\rm d}s=\ddot A^n_t.\vspace{0.2cm}
\end{equation}
Thus, in view of \eqref{eq:6-UnderlineYnIncrease}-
\eqref{eq:6-DotYnLeqUnderlinYn} together with
\eqref{eq:6-DotKnBetaInfinity}-
\eqref{eq:6-DotKn2Leq},
all conditions in \cref{pro:3-PenalizationOfRBSDE} are satisfied, and it follows that there exists an $L^1$ solution $(Y_\cdot,Z_\cdot,K_\cdot,A_\cdot)$, indeed a unique $L^1$ solution, of
DRBSDE $(\xi,g+{\rm d}V,L,U)$ such that,
for each $\beta\in (0,1)$,
$$
\lim\limits_{n\To \infty}\left(\|\underline Y_\cdot^n- Y_\cdot\|_{\s^\beta}+\|\underline Z_\cdot^n- Z_\cdot\|_{\M^\beta}+\|\underline A_\cdot^n- A_\cdot \|_{\s^1}+\|\bar Y_\cdot^n- Y_\cdot\|_{\s^\beta}+\|\bar Z_\cdot^n- Z_\cdot\|_{\M^\beta}+\|\bar K_\cdot^n- K_\cdot \|_{\s^1}\right)=0,
$$
and there exists a subsequence $\{\underline K_\cdot^{n_j}\}$ (resp. $\{\bar A_\cdot^{n_j}\}$ ) of $\{\underline K_\cdot^n\}$ (resp. $\{\bar A_\cdot^n\}$) such that
$$
\lim\limits_{j\To\infty}\sup\limits_{t\in\T}
\left(|\underline K_t^{n_j}- K_t|+|\bar A_t^{n_j}-A_t|\right)=0.
$$
Furthermore, in the same way as in the proof of \cref{thm:5-ExistenceanduniquenssofRBSDEunderH2} we can prove that for each $\beta\in (0,1)$,
\begin{equation}
\label{eq:6-LimitOfgnUnderline}
\lim\limits_{n\To\infty}\left(\left\| \int_0^\cdot g(s,\underline Y_s^n,\underline Z_s^n){\rm d}s-\int_0^\cdot g(s,Y_s,Z_s){\rm d}s\right\|_{\s^\beta}+\left\| \int_0^\cdot g(s,\bar Y_s^n,\bar Z_s^n){\rm d}s-\int_0^\cdot g(s,Y_s,Z_s){\rm d}s\right\|_{\s^\beta}\right)=0.
\end{equation}
Thus, \eqref{eq:6-ConvergenceOfUnderlineYnZnKnAn}
and \eqref{eq:6-ConvergenceOfBarYnZnKnAn} follow immediately.\vspace{0.1cm}

Finally, in view of the fact that $\underline Y^n_\cdot\leq U_\cdot$ and $L_\cdot\leq  \bar Y^n_\cdot$ for each $n\geq 1$, it follows from \cref{cor:3-CorollaryOfComparisonTheorem} that for each $n\geq 1$,
$$
\underline Y^n_\cdot\leq Y^n_\cdot\leq \bar Y^n_\cdot,
$$
which means that
$$
K^n_\cdot=n\int_0^\cdot (Y^n_s-L_s)^-{\rm d}s\leq n\int_0^\cdot (\underline Y^n_s-L_s)^-{\rm d}s=\underline K^n_\cdot
$$
and
$$
A^n_\cdot=n\int_0^\cdot (Y^n_s-U_s)^+{\rm d}s\leq n\int_0^\cdot (\bar Y^n_s-U_s)^+{\rm d}s=\bar A^n_\cdot.\vspace{0.2cm}
$$
Thus, by \eqref{eq:6-ConvergenceOfUnderlineYnZnKnAn}
and \eqref{eq:6-ConvergenceOfBarYnZnKnAn} we know that for each $\beta\in (0,1)$,
\begin{equation}
\label{eq:6-ConvergenceOfYnKnAn}
\lim\limits_{n\To \infty}\|Y_\cdot^n- Y_\cdot\|_{\s^\beta}=0.
\end{equation}
Now, we show the convergence of the sequence $\{Z_\cdot^n\}$. Indeed, for each $n\geq 1$, observe that
$$
\begin{array}{lll}
\Dis (\bar Y_\cdot,\bar Z_\cdot,\bar V_\cdot)&
:=&\Dis (Y_\cdot^n-Y_\cdot,Z_\cdot^n-Z_\cdot,\\
&&\Dis \ \ \int_0^\cdot \left(g(s,Y_s^n,Z_s^n)-g(s,Y_s,Z_s)\right){\rm d}s+\left(K_\cdot^n-K_\cdot\right)
-\left(A_\cdot^n-A_\cdot\right))
\end{array}
$$
satisfies equation \eqref{eq:2-BarY=BarV}. It follows from (i) of \cref{lem:2-Lemma1} with $t=0$ and $\tau=T$ that there exists a constant $C'>0$ such that for each $n\geq 1$ and $\beta\in (0,1)$,
$$
\begin{array}{lll}
\Dis \|Z_\cdot^n-Z_\cdot\|_{\M^\beta}
&\leq & \Dis C'\E\left[\sup\limits_{t\in [0,T]}|Y_t^n-Y_t|^\beta+\sup\limits_{t\in [0,T]}\left[\left(\int_t^T (Y_s^n-Y_s)\left({\rm d}K_s^n-{\rm d}K_s\right)\right)^+\right]^{\beta\over 2}\right] \\
&&\Dis +C'\E\left[\sup\limits_{t\in [0,T]}\left[\left(\int_t^T (Y_s^n-Y_s)\left({\rm d}A_s-{\rm d}A_s^n\right)\right)^+\right]^{\beta\over 2}\right]\\
&&\Dis + C'\E\left[\left(\int_{0}^T|Y^n_s-Y_s|
\left|g(s,Y_s^n,Z_s^n)-g(s,Y_s,Z_s)\right| {\rm d}s\right)^{\beta\over 2}\right].
\end{array}
$$
It then follows from the fact of $L_\cdot\leq Y_\cdot\leq U_\cdot$ and the definitions of $K^n_\cdot$ and $A^n_\cdot$ as well as H\"{o}lder's inequality that
\begin{equation}
\label{eq:6-convergenceOfZn}
\begin{array}{lll}
\Dis\|Z_\cdot^n-Z_\cdot\|_{\M^\beta}
&\leq & \Dis C'\|Y_\cdot^n-Y_\cdot\|_{\s^\beta}+ C'\|Y_\cdot^n-Y_\cdot\|_{\s^\beta}^{1\over 2}\cdot\left(\left(\E[|K_T|^\beta]\right)^{1\over 2}+ \left(\E[|A_T|^\beta]\right)^{1\over 2}\right)\\
&&\Dis +C'\|Y_\cdot^n-Y_\cdot\|_{\s^\beta}
^{1\over 2}\cdot \left(\E\left[\left(\int_{0}^T
\left(|g(t,Y_t^n,Z_t^n)|+|g(t,Y_t,Z_t)|\right) {\rm d}t\right)^{\beta}\right]\right)^{1\over 2}.
\end{array}
\end{equation}
Thus, in view of \ref{A:(H1)} and \ref{A:(H2)} of $g$, it follows from \eqref{eq:6-ConvergenceOfYnKnAn} and \eqref{eq:6-convergenceOfZn} together with \cref{lem:2-EstimateOfZandg} that for each $\beta\in (0,1)$,\vspace{-0.1cm}
\begin{equation}
\label{eq:6-ConvergenceOfZn}
\lim\limits_{n\To\infty}\|Z_\cdot^n-Z_\cdot
\|_{\M^\beta}=0.
\end{equation}
Furthermore, by \eqref{eq:6-ConvergenceOfYnKnAn} and \eqref{eq:6-ConvergenceOfZn}, a similar argument to \eqref{eq:6-LimitOfgnUnderline} yields that for each $\beta\in (0,1)$,\vspace{0.1cm}
\begin{equation}
\label{eq:6-LimitOfgn}
\lim\limits_{n\To\infty}\left\| \int_0^\cdot g(s,Y_s^n,Z_s^n){\rm d}s-\int_0^\cdot g(s,Y_s,Z_s){\rm d}s\right\|_{\s^\beta}=0.\vspace{0.1cm}
\end{equation}
Finally, \eqref{eq:6-ConvergenceOfYnZnKnAn} follows from \eqref{eq:6-ConvergenceOfYnKnAn}, \eqref{eq:6-ConvergenceOfZn} and \eqref{eq:6-LimitOfgn}. The proof of \cref{thm:6-ExistenceandUniquenessUnder(H2)} is then complete.
\end{proof}

By virtue of (i) of \cref{thm:6-ExistenceandUniquenessUnder(H2)}, \cref{cor:3-CorollaryOfComparisonTheorem} and (ii) of \cref{cor:5-ComparisonForKofRBSDE}, a similar argument to that in \cref{cor:5-ComparisonForKofRBSDE} yields the following corollary.\vspace{-0.1cm}

\begin{cor}
\label{cor:6-ComparisonForKofDRBSDE}
Assume that $\xi^1,\xi^2\in\LT$ with $\xi^1\leq \xi^2$, $V_\cdot^1, V_\cdot^2\in\vcal^{1}$ with ${\rm d}V^1\leq {\rm d}V^2$, and both generators $g^1$ and $g^2$ satisfy \ref{A:(H1)} and \ref{A:(H2)} with
$$
\as ,\ \RE\ (y,z)\in \R\times \R^d,\ \ g^1(t,y,z)\leq g^2(t,y,z).
$$
For $i=1,2$, let \ref{A:(H3)} hold for $\xi^i$, $L_\cdot^i$, $U_\cdot^i$ and $X_\cdot^i$ associated with $g^i$, and $(Y_\cdot^i,Z_\cdot^i,K_\cdot^i,A_\cdot^i)$ be the unique $L^1$ solution of DRBSDE $(\xi^i,g^i+{\rm d}V^i,L^i,U^i)$ (recall \cref{thm:6-ExistenceandUniquenessUnder(H2)}). If $L^1_\cdot=L^2_\cdot$ and $U^1_\cdot=U^2_\cdot$, then
$${d}K^1\geq {\rm d}K^2\ \ {\rm and}\ \  {d}A^1\leq {\rm d}A^2.$$
\end{cor}

\begin{thm}
\label{thm:6-ExistenceofDRBSDEunderH2'Penalization}
Let $V_\cdot\in\vcal^1$, $g^1$ satisfy assumptions \ref{A:(H1)} and \ref{A:(H2')}, $g^2$ satisfy assumption \ref{A:(AA)}, the generator $g:=g^1+g^2$, and assumption \ref{A:(H3)}(i) hold true for $L_\cdot$, $U_\cdot$ and $\xi$. Then, DRBSDE $(\xi,g+{\rm d}V,L,U)$ admits an $L^1$ solution iff \ref{A:(H3)}(ii) is satisfied for $X_\cdot$, $L_\cdot$, $U_\cdot$ and $g$ (or $g^1$). Moreover, we assume that \ref{A:(H3)}(ii) holds also true for $X_\cdot$, $L_\cdot$, $U_\cdot$ and $g$ (or $g^1$).
\begin{itemize}
\item [(i)] For each $n\geq 1$, let $(\underline Y^n_\cdot,\underline Z^n_\cdot,\underline A^n_\cdot)$ be the minimal (resp. maximal ) $L^1$ solution of $\bar{R}$BSDE $(\xi,\bar g_n+{\rm d}V,U)$ with $\bar g_n(t,y,z):=g(t,y,z)+n(y-L_t)^-$ and $\underline K^n_\cdot$, i.e., \eqref{eq:6-PenalizationForRBSDEwithSuperBarrier},
    (recall \cref{thm:5-ExistenceofRBSDEunderH2'Approximation}(ii)).
    Then, DRBSDE $(\xi,g+{\rm d}V,L,U)$ admits a minimal $L^1$ solution (resp. an $L^1$ solution) $(\underline Y_\cdot,\underline Z_\cdot,\underline K_\cdot,\underline A_\cdot)$ such that for each $\beta\in (0,1)$,
$$
\lim\limits_{n\To \infty}\left(\|\underline Y_\cdot^n-\underline Y_\cdot\|_{\s^\beta}+\|\underline Z_\cdot^n-\underline  Z_\cdot\|_{\M^\beta}+\|\underline A_\cdot^n-\underline A_\cdot \|_{\s^1}\right)=0,
$$
and there exists a subsequence $\{ \underline K_\cdot^{n_j}\}$ of $\{ \underline K_\cdot^n\}$ such that
    $$
    \lim\limits_{j\To\infty}\sup\limits_{t\in\T}
    |\underline K_t^{n_j}-\underline K_t|=0;
    $$

\item [(ii)] For each $n\geq 1$, let $(\bar Y^n_\cdot,\bar Z^n_\cdot,\bar K^n_\cdot)$ be be the maximal (resp. minimal) $L^1$ solution of $\underline{R}$BSDE $(\xi,\underline g_n+{\rm d}V,L)$ with $\underline g_n(t,y,z):=g(t,y,z)-n(y-U_t)^+$ and $\bar A^n_\cdot$, i.e., \eqref{eq:6-PenalizationForRBSDEwithLowerBarrier},
    (recall \cref{thm:5-ExistenceofRBSDEunderH2'Approximation}(i)).
Then, DRBSDE $(\xi,g+{\rm d}V,L,U)$ admits a maximal $L^1$ solution (resp. an $L^1$ solution) $(\bar Y_\cdot,\bar Z_\cdot,\bar K_\cdot,\bar A_\cdot)$ such that for each $\beta\in (0,1)$,
$$
\lim\limits_{n\To \infty}\left(\|\bar Y_\cdot^n- \bar Y_\cdot\|_{\s^\beta}+\|\bar Z_\cdot^n-\bar  Z_\cdot\|_{\M^\beta}+\|\bar K_\cdot^n-\bar K_\cdot \|_{\s^1}\right)=0,
$$
and there exists a subsequence $\{ \bar A_\cdot^{n_j}\}$ of $\{\bar A_\cdot^n\}$ such that
    $$
    \lim\limits_{j\To\infty}\sup\limits_{t\in\T}
    |\bar A_t^{n_j}-\bar A_t|=0.
    $$
\end{itemize}
\end{thm}

\begin{proof}
We only prove (i), and (ii) can be proved in the same way. In view of \cref{thm:5-ExistenceofRBSDEunderH2'Approximation}, \cref{cor:5-ComparisonForYandKofRBSDEUnderH2'},
\cref{lem:2-SublinearOfg}, \cref{lem:2-gBelongstoH1} and \cref{pro:3-PenalizationOfRBSDE}, by a similar argument to that in the proof of \cref{thm:6-ExistenceandUniquenessUnder(H2)} we can prove that all conclusions in (i) of \cref{thm:6-ExistenceofDRBSDEunderH2'Penalization} hold true except for the minimal property of the $L^1$ solution $(\underline Y_\cdot,\underline Z_\cdot,\underline K_\cdot,\underline A_\cdot)$ of DRBSDE $(\xi,g+{\rm d}V,L,U)$ when $(\underline Y^n_\cdot,\underline Z^n_\cdot,\underline A^n_\cdot)$ is the minimal $L^1$ solution of ${\bar R}$BSDE $(\xi,\bar {g}_n+{\rm d}V,U)$ for each $n\geq 1$. Now, we will show this property.

Indeed, for any $L^1$ solution $(Y_\cdot,Z_\cdot,K_\cdot,A_\cdot)$ of
DRBSDE $(\xi,g+{\rm d}V,L,U)$, it is not hard to check that $(Y_\cdot,Z_\cdot,A_\cdot)$ is an $L^1$ solution of $\bar R$BSDE $(\xi,\bar {g}_n+{\rm d}\bar V,U)$ with $\bar V_\cdot:=V_\cdot+K_\cdot$ for each $n\geq 1$. Thus, in view of the assumption that $(\underline Y^n_\cdot,\underline Z^n_\cdot,\underline A^n_\cdot)$ is the minimal $L^1$ solution of $\bar R$BSDE $(\xi,\bar {g}_n+{\rm d}V,U)$ for each $n\geq 1$, \cref{cor:5-ComparisonForYandKofRBSDEUnderH2'} yields that for each $n\geq 1$,
$$\underline Y^n_t\leq Y_t,\ \ t\in \T.$$
Furthermore, since $\Lim \|\underline Y^n_\cdot-\underline Y_\cdot\|_{\s^\beta}=0$ for each $\beta\in (0,1)$, we know that
$$\underline Y_t\leq Y_t,\ \ t\in \T,$$
which is the desired result.
\end{proof}

In view of \cref{thm:6-ExistenceandUniquenessUnder(H2)}, \cref{cor:6-ComparisonForKofDRBSDE} and \cref{pro:3-Approximation}, a similar argument to that in \cref{thm:5-ExistenceofRBSDEunderH2'Approximation}
yields the following convergence result, whose proof is omitted.\vspace{-0.1cm}

\begin{thm}
\label{thm:6-ExistenceofDRBSDEunderH2'Approximation}
Let $V_\cdot\in\vcal^1$, $g^1$ satisfy assumptions \ref{A:(H1)} and \ref{A:(H2')}, $g^2$ satisfy assumption \ref{A:(AA)}, the generator $g:=g^1+g^2$, and \ref{A:(H3)} holds true for $L_\cdot$, $U_\cdot$, $\xi$, $X_\cdot$ and $g$ (or $g^1$). Then DRBSDE $(\xi,g+{\rm d}V,L,U)$ admits a minimal (resp. maximal) $L^1$ solution $(Y_\cdot,Z_\cdot,K_\cdot,A_\cdot)$ such that for each $\beta\in (0,1)$,
$$
\lim\limits_{n\To \infty}\left(\|Y_\cdot^n-Y_\cdot\|_{\s^\beta}+
\|Z_\cdot^n-Z_\cdot\|_{\M^\beta}
+\|K_\cdot^n-K_\cdot\|_{\s^1}
+\|A_\cdot^n-A_\cdot\|_{\s^1}\right)=0,
$$
where for each $n\geq 1$, $(Y^n_\cdot,Z^n_\cdot,K^n_\cdot,A^n_\cdot)$ is the unique $L^1$ solution of
DRBSDE $(\xi,g_n+{\rm d}V,L,U)$ with a generator $g_n$ satisfying \ref{A:(H1)},  \ref{A:(H2)} and \ref{A:(H3)} (recall \cref{thm:6-ExistenceandUniquenessUnder(H2)}).
\end{thm}

By \cref{cor:3-CorollaryOfComparisonTheorem}, \cref{cor:6-ComparisonForKofDRBSDE} and the proof of \cref{thm:6-ExistenceofDRBSDEunderH2'Approximation}, it is not hard to verify the following comparison result for the minimal (resp. maximal) $L^1$ solutions of DRBSDEs.\vspace{-0.1cm}

\begin{cor}
\label{cor:6-ComparisonOfYKAUnder(H2')}
Assume that $\xi^1,\xi^2\in\LT$ with $\xi^1\leq \xi^2$, $V_\cdot^1, V_\cdot^2\in\vcal^{1}$ with ${\rm d}V^1\leq {\rm d}V^2$, $g^{1,1}$ and $g^{2,1}$ satisfy \ref{A:(H1)} and \ref{A:(H2')}, $g^{1,2}$ and $g^{2,2}$ satisfy \ref{A:(AA)}, $g^1:=g^{1,1}+g^{1,2}$ and $g^2:=g^{2,1}+g^{2,2}$ with\vspace{-0.1cm}
$$
\as,\ \RE\ (y,z)\in \R\times \R^d,\ \  g^{1,1}(t,y,z)\leq g^{2,1}(t,y,z)\ \ {\rm and}\ \ g^{1,2}(t,y,z)\leq g^{2,2}(t,y,z).
\vspace{-0.1cm}
$$
For $i=1,2$, let \ref{A:(H3)} hold for $L_\cdot^i$, $U_\cdot^i$, $\xi^i$ and $X_\cdot^i$ associated with $g^i$ (or $g^{i,1}$), and $(Y_\cdot^i,Z_\cdot^i,K_\cdot^i,A_\cdot^i)$ be the minimal (resp. maximal) $L^1$ solution of DRBSDE $(\xi^i,g^i+{\rm d}V^i,L^i,U^i)$ (recall \cref{thm:6-ExistenceofDRBSDEunderH2'Approximation}). If $L^1_\cdot\leq L^2_\cdot$ and $U^1_\cdot\leq U^2_\cdot$, then $Y^1_t\leq Y^2_t$ for each $t\in\T$, and if $L^1_\cdot=L^2_\cdot$ and $U^1_\cdot=U^2_\cdot$, then
$$
{d}K^1\geq {\rm d}K^2\ \ {\rm and}\ \ {d}A^1\leq {\rm d}A^2.
$$
\end{cor}

\section{Examples and remarks}
\label{sec:7-ExamplesandRemarks}
\setcounter{equation}{0}

We first introduce several examples which the results of this paper can be applied to. Note that to the best of our knowledge, all conclusions of these examples can not be obtained by any existing results.\vspace{-0.1cm}

\begin{ex}
\label{ex:7-7.1}
Let the generator $g$ be defined as follows:
$$
g(\omega,t,y,z)=h(|y|)+e^{-y|B_t(\omega)|^2}+(e^{-y}\wedge 1)\cdot (\sqrt{|z|}+\sqrt[3]{|z|})+{1\over \sqrt{t}}1_{t>0}
$$
where, with $\delta>0$ small enough,
$$h(x)=\left\{
\begin{array}{lll}
-x\ln x& ,&0<x\leq \delta;\\
h'(\delta-)(x-\delta)+h(\delta)& ,&x> \delta;\\
0& ,&{\rm other\ cases}.
\end{array}\right.\vspace{0.1cm}$$
It is not very hard to verify that this $g$ satisfies assumption \ref{A:(H1)} with $\rho(x)=h(x)$, $g(t,0,0)={1\over \sqrt{t}}1_{t>0}+1$, and $\psi_t(\omega,r)=h(\delta)+h'(\delta-)r
+e^{r|B_t(\omega)|^2}+1$, and assumption \ref{A:(H2)} with $\phi(x)=\sqrt{|x|}+\sqrt[3]{|x|}$, $\gamma=2$, $f_t(\omega)\equiv 1$ and $\alpha=1/2$. We have
\begin{itemize}
\item [1)] It follows from  \cref{thm:4-ExistenceanduniquenssofBSDEunderH2} that for each $\xi\in\LT$ and $V_\cdot\in\vcal^1$, BSDE $(\xi,g+{\rm d}V)$ admits a unique $L^1$ solution;

\item [2)] It follows from \cref{thm:5-ExistenceanduniquenssofRBSDEunderH2} that if \ref{A:(H3L)} (resp. \ref{A:(H3U)}) is satisfied and $V_\cdot\in \vcal^1$, then $\underline R$BSDE $(\xi,g+{\rm d}V,L)$ (resp. $\bar R$BSDE $(\xi,g+{\rm d}V,U)$) admits a unique $L^1$ solution;

\item [3)] It follows from \cref{thm:6-ExistenceandUniquenessUnder(H2)} that if \ref{A:(H3)} is satisfied and $V_\cdot\in \vcal^1$, then
DRBSDE $(\xi,g+{\rm d}V,L,U)$ admits a unique $L^1$ solution.
\end{itemize}
\end{ex}

\begin{ex}
\label{ex:7-7.2}
Let the generator $g:=g^1+g^2$ with
$$
g^1(\omega,t,y,z)=h(|y|)-y^3e^{|B_t(\omega)|^4}-e^y\cdot \sin^2 |z|+\sqrt{|z|}\cos |z|+{1\over \sqrt[3]{t}}1_{t>0}
$$
and
$$
g^2(\omega, t,y,z)=\sqrt[3]{|y|} +y\cos y+\sqrt[4]{|y|\cdot |z|}+|B_t(\omega)|,\vspace{0.1cm}
$$
where $h(\cdot)$ is defined in \cref{ex:7-7.1}.
It is not very hard to verify that $g^1$ satisfies \ref{A:(H1)}(i) with $\rho(x)=h(x)$ and \ref{A:(HH)} with $f_t(\omega)=1+{1\over \sqrt[3]{t}}1_{t>0}+h(\delta)$, $\varphi_t(\omega,r)=h'(\delta-)r
+r^3e^{|B_t(\omega)|^4}+e^r-1$, $\lambda=1$ and $\alpha=1/2$, and that $g^2$ satisfies \ref{A:(AA)} with $\tilde f_t(\omega)=|B_t(\omega)|+2$, $\tilde\mu=3$, $\tilde\lambda=1$ and $\tilde\alpha=1/2$. We have
\begin{itemize}
\item [1)] It follows from  \cref{thm:4-ExistenceofBSDEunderHH} that for each $\xi\in\LT$ and $V_\cdot\in\vcal^1$, BSDE $(\xi,g+{\rm d}V)$ admits a maximal and a minimal $L^1$ solution;

\item [2)]It follows from  \cref{cor:5-ExistenceofRBSDEunderAA} that for each $V_\cdot\in\vcal^1$, $\xi\in \LT$, and $L^+_\cdot\in \s^1$ with $\xi\geq L_T$ (resp. $U^-_\cdot\in \s^1$ with $\xi\leq U_T$), $\underline R$BSDE $(\xi,g^2+{\rm d}V,L)$ (resp. $\bar R$BSDE $(\xi,g^2+{\rm d}V,U)$) admits a maximal and a minimal $L^1$ solution;

\item [3)]It follows from  \cref{thm:6-ExistenceofDRBSDEunderH2'Approximation} that if \ref{A:(H3)} is satisfied for $g^2$, and $V_\cdot\in \vcal^1$, then DRBSDE $(\xi,g^2+{\rm d}V,L,U)$ admits a maximal and a minimal $L^1$ solution.
\end{itemize}
We also note that this $g^1$ satisfies neither assumption \ref{A:(H2)} nor assumption \ref{A:(H2')}.
\end{ex}

\begin{ex}
\label{ex:7-7.3}
Let the generator $g:=g^1+g^2$ with
$$
g^1(\omega,t,y,z)=\bar h(|y|)-e^{y|B_t(\omega)|^3 }+(e^{-y}\wedge 1)\cdot \sqrt{|z|}\cos |z|+{1\over \sqrt[4]{t}}1_{t>0}
$$
and
$$
g^2(\omega, t,y,z)=y\cos |z|+ \sqrt[3]{|z|}\sin y +\sqrt{1+|y|+|z|}+|B_t(\omega)|^2,
$$
where, with $\delta>0$ small enough,
$$\bar h(x)=\left\{
\begin{array}{lll}
x|\ln x|\ln|\ln x|& ,&0<x\leq \delta;\\
\bar h'(\delta-)(x-\delta)+\bar h(\delta)& ,&x> \delta;\\
0& ,&{\rm other\ cases}.
\end{array}\right.\vspace{0.1cm}$$
It is not very hard to verify that $g^1$ satisfies assumption \ref{A:(H1)} with $\rho(x)=\bar h(x)$, $g(t,0,0)={1\over \sqrt[4]{t}}1_{t>0}-1$, and $\psi_t(\omega,r)=\bar h(\delta)+\bar h'(\delta-)r
+e^{r|B_t(\omega)|^3}+1$,  and assumption \ref{A:(H2')} with $f_t(\omega)\equiv 0$, $\mu=0$, $\lambda=1$ and $\alpha=1/2$, and that $g^2$ satisfies \ref{A:(AA)} with $\tilde f_t(\omega)=|B_t(\omega)|^2+2$, $\tilde\mu=2$, $\tilde\lambda=2$ and $\tilde\alpha=1/2$. We have
\begin{itemize}
\item [1)]It follows from  \cref{thm:4-ExistenceofBSDEunderHH} that for each $\xi\in\LT$ and $V_\cdot\in\vcal^1$, BSDE $(\xi,g+{\rm d}V)$ admits a maximal and a minimal $L^1$ solution;

\item [2)]It follows from \cref{thm:5-ExistenceofRBSDEunderH2'Approximation} that if \ref{A:(H3L)} (resp. \ref{A:(H3U)}) is satisfied and $V_\cdot\in \vcal^1$, then $\underline R$BSDE $(\xi,g+{\rm d}V,L)$ (resp. $\bar R$BSDE $(\xi,g+{\rm d}V,U)$) admits a maximal and a minimal $L^1$ solution;

\item [3)]It follows from \cref{thm:6-ExistenceofDRBSDEunderH2'Approximation} that if \ref{A:(H3)} is satisfied and $V_\cdot\in \vcal^1$, then
DRBSDE $(\xi,g+{\rm d}V,L,U)$ admits a maximal and a minimal $L^1$ solution;

\item [4)] It follows from  \cref{cor:5-ExistenceofRBSDEunderAA} that for each $V_\cdot\in\vcal^1$, $\xi\in \LT$, and $L^+_\cdot\in \s^1$ with $\xi\geq L_T$ (resp. $U^-_\cdot\in \s^1$ with $\xi\leq U_T$), $\underline R$BSDE $(\xi,g^2+{\rm d}V,L)$ (resp. $\bar R$BSDE $(\xi,g^2+{\rm d}V,U)$) admits a maximal and a minimal $L^1$ solution.
\end{itemize}
We also note that this $g^1$ does not satisfy assumption \ref{A:(H2)}.
\end{ex}

Finally, we give the following remark to end this paper.\vspace{-0.1cm}

\begin{rmk}
\label{rmk:7-7.1}
With respect to the work of this paper, we would like to mention the following things.
\begin{itemize}
\item [1)] The basic assumptions \ref{A:(H1)} and \ref{A:(H2)} of the generator $g$ in this paper are strictly weaker than the corresponding assumptions used in \citet{BriandDelyonHu2003SPA}, \citet{Klimsiak2012EJP}, \citet{Klimsiak2013BSM}, \citet{RozkoszSlominski2012SPA} and
    \citet{BayraktarYao2015SPA} for the $L^1$ solutions, where $\rho(x)=kx$ and $\phi(x)=kx$ for some constant $k\geq 0$. Furthermore, assumption \ref{A:(H2')} is weaker than assumption \ref{A:(H2)}, and assumption \ref{A:(HH)} is weaker than \ref{A:(H2')} and \ref{A:(H1)}(ii)(iii);

\item [2)] All of conditions \eqref{eq:2-SemiLinearGrowthCondition},
    \eqref{eq:2-LinearGrowthConditionOfg}, \eqref{eq:3-SubLinearGrowthofgYnZn},
    \eqref{eq:3-SemiLinearGrowthCondition},
    \eqref{eq:3-LinearGrowthofgnYnZn--k}
 and \eqref{eq:3-LinearGrowthofgnYnZn}
used respectively in \cref{lem:2-EstimateOfZandg},
\cref{lem:2-EstimateOfZKAg}, \cref{pro:3-PenalizationOfRBSDE}, \cref{pro:3-PenalizationOfBSDE}, \cref{pro:3-Approximation} and \cref{rmk:3-rmk3.4} are very general, which is strictly weaker than the usual linear/sub-linear growth condition of $g$ in $(y,z)$. Indeed, when these conditions are satisfied, the generator $g$ can still have a general growth in $(y,z)$, as can be seen in the proof of our main results in \cref{sec:4-ExistenceBSDEs}, \cref{sec:5-ExistenceRBSDEs} and \cref{sec:6-ExistenceOfDRBSDEs};

\item [3)] The way by which the comparison theorem (\cref{pro:3-ComparisonTheoremofDRBSDE}) is used in \cref{thm:4-ExistenceanduniquenssofBSDEunderH2} and \cref{thm:4-ExistenceofBSDEunderHH} is interesting for me;

\item [4)] It is uncertain that the generator $g$ used in \cref{thm:4-ExistenceofBSDEunderHH}, \cref{thm:5-ExistenceofRBSDEunderH2'Penalization},
    \cref{thm:5-ExistenceofRBSDEunderH2'Approximation},
    \cref{thm:6-ExistenceofDRBSDEunderH2'Penalization}
    and \cref{thm:6-ExistenceofDRBSDEunderH2'Approximation}
    satisfies assumption \ref{A:(H1)}(i), as can be seen in \cref{ex:7-7.2} and \cref{ex:7-7.3};

\item [5)] Generally speaking, under the assumptions of \cref{thm:5-ExistenceofRBSDEunderH2'Penalization}, we do not know whether the maximal $L^1$ solution of $\underline R$BSDE $(\xi,g+{\rm d}V,L)$ (resp. the minimal $L^1$ solution of $\bar R$BSDE $(\xi,g+{\rm d}V,U)$) can be approximated by a sequence of $L^1$ solutions of BSDEs;

\item [6)] Generally speaking, under the assumptions of \cref{thm:6-ExistenceofDRBSDEunderH2'Penalization}, we do not know whether the maximal (resp. minimal) $L^1$ solution of DRBSDE $(\xi,g+{\rm d}V,L,U)$ can be approximated by a sequence of $L^1$ solutions of $\bar R$BSDEs with upper barrier $U_\cdot$ (resp. $\underline R$BSDEs with lower barrier $L_\cdot$). In particular, under the same assumptions we also do not know whether an $L^1$ solution of DRBSDE $(\xi,g+{\rm d}V,L,U)$ can be approximated by a sequence of $L^1$ solutions of BSDEs in general;

\item [7)] The continuity condition of $g^2$ (resp. $g$) in $(y,z)$ used in \cref{thm:4-ExistenceofBSDEunderHH}, \cref{thm:5-ExistenceofRBSDEunderH2'Penalization},
    \cref{thm:5-ExistenceofRBSDEunderH2'Approximation},
    \cref{thm:6-ExistenceofDRBSDEunderH2'Penalization}
    and \cref{thm:6-ExistenceofDRBSDEunderH2'Approximation} (resp. \cref{cor:5-ExistenceofRBSDEunderAA}) can be relaxed to the left-continuity and lower semi-continuity condition in case of the minimal $L^1$ solution and the right-continuity and upper semi-continuity condition in case of the maximal $L^1$ solution, with a similar argument as in \citet{FanJiang2012SPL},  \citet{Fan2016SPL} and \citet{Fan2017arXivLpSolutionofDRBSDEs}. The same is in \cref{cor:4-ComparisonOfMinSolutionUnderHH},
    \cref{cor:5-ComparisonForYandKofRBSDEUnderH2'}
    and \cref{cor:6-ComparisonOfYKAUnder(H2')};

\item [8)] Since the associated assumptions are more general, the results of this paper strengthen some known corresponding works on the $L^1$ solutions obtained, for example, in \citet{ElKarouiPengQuenez1997MF}, \citet{BriandDelyonHu2003SPA}, \citet{BriandHu2006PTRF}, \citet{FanLiu2010SPL}, \citet{Fan2017arXivL1SolutionofBSDEsBulletin}, \citet{Klimsiak2012EJP}, \citet{RozkoszSlominski2012SPA}
    and \citet{BayraktarYao2015SPA};

\item [9)] Under assumptions \ref{A:(H1)}(i), \ref{A:(HH)} and \ref{A:(H3L)} (resp. \ref{A:(H3U)}), the existence of an $L^1$ solution for $\underline R$BSDE $(\xi,g+{\rm d}V,L)$ (resp. $\bar R$BSDE $(\xi,g+{\rm d}V,U)$) is still open. And, Under assumptions \ref{A:(H1)}(i), \ref{A:(HH)} and \ref{A:(H3)}, the existence of an $L^1$ solutions for DRBSDE $(\xi,g+{\rm d}V,L,U)$ is also open.\vspace{0.2cm}
\end{itemize}
\end{rmk}

%\section*{Acknowledgements}
%The authors are greatly grateful to the referee for his/her careful reading
%and many constructive suggestions.

\setlength{\bibsep}{2pt}
%\bibliographystyle{elsarticle-harv}
%\bibliography{BSDEs}

%\section*{References}

\def\cprime{$'$} \def\cfgrv#1{\ifmmode\setbox7\hbox{$\accent"5E#1$}\else
  \setbox7\hbox{\accent"5E#1}\penalty 10000\relax\fi\raise 1\ht7
  \hbox{\lower1.05ex\hbox to 1\wd7{\hss\accent"12\hss}}\penalty 10000
  \hskip-1\wd7\penalty 10000\box7}

\end{document}